\documentclass[12pt]{amsart}
\usepackage{amsmath}
\usepackage{amsxtra}
\usepackage{enumitem}
\usepackage{amscd}
\usepackage{amsthm}
\usepackage{hyperref}
\usepackage[T1]{fontenc}
\usepackage{amsfonts}
\usepackage{amssymb}
\usepackage{eucal}
\usepackage[all]{xypic}
\usepackage{color} 
\newtheorem{thm}[subsection]{Theorem}
\newtheorem{cor}[subsection]{Corollary}
\newtheorem{lem}[subsection]{Lemma}
\newtheorem{prop}[subsection]{Proposition}

\theoremstyle{definition}

\newtheorem{defn}[subsection]{Definition}

\newtheorem{rem}[subsection]{Remark}

\newtheorem{remark}[subsection]{Remark}
\newtheorem{question}[subsection]{Question}
\newtheorem{example}[subsection]{Example}

\usepackage{yhmath}
\DeclareSymbolFont{largesymbols}{OMX}{yhex}{m}{n}
\DeclareMathAccent{\widetilde}{\mathord}{largesymbols}{"65}

\newcommand{\thmref}[1]{Theorem~\ref{#1}}
\newcommand{\secref}[1]{\S~\ref{#1}}
\newcommand{\lemref}[1]{Lemma~\ref{#1}}

\newcommand{\defref}[1]{Definition~\ref{#1}}
\newcommand{\propref}[1]{Proposition~\ref{#1}}
\newcommand{\corref}[1]{Corollary~\ref{#1}}

\newcommand{\remref}[1]{Remark~\ref{#1}}
\newcommand{\exref}[1]{Example~\ref{#1}}
\newcommand{\exmref}[1]{Example~\ref{#1}}
\def\qref#1{equation~(\ref{#1})}

\newcommand{\nc}{\newcommand}
\nc{\renc}{\renewcommand}
\nc{\ssec}{\subsection}
\nc{\sssec}{\subsubsection} 
\nc\ol{\overline}
\nc\wt{\widetilde}
\nc\wh{\widehat}

\renc{\d}{{\delta}}
\nc{\Aa}{{\mathbb{A}}}
\nc{\Bb}{{\mathbb{B}}}
 \nc{\Gg}{{\mathbb{G}}}  
\nc{\Hh}{{\mathbb{H}}}
 \nc{\Nn}{{\mathbb{N}}}
\nc{\Pp}{{\mathbb{P}}}
\nc{\Rr}{{\mathbb{R}}}
\nc{\BV}{{\mathbb{V}}}
\nc{\BW}{{\mathbb{W}}}
\nc{\Zz}{{\mathbb{Z}}}
\nc{\Qq}{{\mathbb{Q}}}
\nc{\Ss}{{\mathbb{S}}}
\nc{\Cc}{{\mathbb{C}}}
\nc{\Ff}{{\mathbb{F}}}
\nc{\EL}{{L_\infty}}
 
\nc{\CA}{{\mathcal{A}}}
\nc{\CB}{{\mathcal{B}}}

\nc{\CE}{{\mathcal{E}}}
\nc{\CF}{{\mathcal{F}}}

 \def\ram{^{\mathfrak{r}}}
 \def\ram{^{\it {ram}}}
\def\td{^{def}}
\def\LS{\CL}
\def\TS{\mathcal{T}}
\def\TSb{\bar{\TS}}
\def\LSb{\bar{\CL}}
\nc{\Las}{\mathsf{Las}}
\nc{\CG}{{\mathcal{G}}}
\def\bCG{\bar{\CG}}
\nc{\CL}{{\mathcal{L}}}
\nc{\R}{{\mathcal{R}}}
\nc{\CC}{{\mathcal{C}}}
\nc{\CM}{{\mathcal{M}}}

\nc{\CN}{{\mathcal{N}}}
\nc{\Oog}{{\mathbb{O}}}
\nc{\Oo}{{\mathcal{O}}}
\nc{\CP}{{\mathcal{P}}}
\nc{\CQ}{{\mathcal{Q}}}
\nc{\CR}{{\mathcal{R}}}
\nc{\CS}{{\mathcal{S}}}
\nc{\CT}{{\mathcal{T}}}
\nc{\CU}{{\mathcal{P}}}
\def\CUb{\bar{\CU}}

\nc{\CV}{{\mathcal{V}}}
 
\nc{\CW}{{\mathcal{W}}}
\nc{\CZ}{{\mathcal{Z}}}

\nc{\cM}{{\check{\mathcal M}}{}}
\nc{\csM}{{\check{\mathcal A}}{}}
\nc{\oM}{{\overset{\circ}{\mathcal M}}{}}
\nc{\obM}{{\overset{\circ}{\mathbf M}}{}}
\nc{\oCA}{{\overset{\circ}{\mathcal A}}{}}
\nc{\obA}{{\overset{\circ}{\mathbf A}}{}}
\nc{\ooM}{{\overset{\circ}{M}}{}}
\nc{\osM}{{\overset{\circ}{\mathsf M}}{}}
\nc{\vM}{{\overset{\bullet}{\mathcal M}}{}}
\nc{\nM}{{\underset{\bullet}{\mathcal M}}{}}
\nc{\oD}{{\overset{\circ}{\mathcal D}}{}}
\nc{\obD}{{\overset{\circ}{\mathbf D}}{}}
\nc{\oA}{{\overset{\circ}{\mathbb A}}{}}
\nc{\op}{{\overset{\bullet}{\mathbf p}}{}}
\nc{\cp}{{\overset{\circ}{\mathbf p}}{}}
\nc{\oU}{{\overset{\bullet}{\mathcal U}}{}}
\nc{\oZ}{{\overset{\circ}{\mathcal Z}}{}}
\nc{\ofZ}{{\overset{\circ}{\mathfrak Z}}{}}
\nc{\oF}{{\overset{\circ}{\fF}}}
\def\bJ{\bar{J}}
\nc{\fa}{{\mathfrak{a}}}
\nc{\fb}{{\mathfrak{b}}}
\nc{\fg}{{\mathfrak{g}}}
\nc{\fgt}{{\fg}_!}
\nc{\fgl}{{\mathfrak{gl}}}
\nc{\fh}{{\mathfrak{h}}}
\nc{\fj}{{\mathfrak{j}}}
\nc{\fm}{{\mathfrak{m}}}
\nc{\ft}{{\mathfrak{t}}}
\nc{\fn}{{\mathfrak{n}}}
\nc{\fu}{{\mathfrak{u}}}
\nc{\fp}{{\mathfrak{p}}}
\nc{\fr}{{\mathfrak{r}}}
\nc{\fs}{{\mathfrak{s}}}
\nc{\fsl}{{\mathfrak{sl}}}
\nc{\hsl}{{\widehat{\mathfrak{sl}}}}
\nc{\hgl}{{\widehat{\mathfrak{gl}}}}
\nc{\hg}{{\widehat{\mathfrak{g}}}}
\nc{\chg}{{\widehat{\mathfrak{g}}}{}^\vee}
\nc{\hn}{{\widehat{\mathfrak{n}}}}
\nc{\chn}{{\widehat{\mathfrak{n}}}{}^\vee}

\nc{\fA}{{\mathfrak{A}}}
\nc{\fB}{{\mathfrak{B}}}
\nc{\fD}{{\mathfrak{D}}}
\nc{\fE}{{\mathfrak{E}}}
\nc{\fF}{{\mathfrak{F}}}
\nc{\fG}{{\mathfrak{G}}}
\nc{\fK}{{\mathfrak{K}}}
\nc{\fL}{{\mathfrak{L}}}
\nc{\fM}{{\mathfrak{M}}}
\nc{\fN}{{\mathfrak{N}}}
\nc{\fP}{{\mathfrak{P}}}
\nc{\fU}{{\mathfrak{U}}}
\nc{\fV}{{\mathfrak{V}}}
\nc{\fZ}{{\mathfrak{Z}}}

\nc{\bb}{{\mathbf{b}}}
\nc{\bc}{{\mathbf{c}}}
\nc{\bd}{\partial}
\nc{\be}{{\mathbf{e}}}
\nc{\bj}{{\mathbf{j}}}
\nc{\bn}{{\mathbf{n}}}
\nc{\bp}{{\mathbf{p}}}
\nc{\bq}{{\mathbf{q}}}
\nc{\bF}{{\mathbf{F}}}
\nc{\bu}{{\mathbf{u}}}
\nc{\bv}{{\mathbf{v}}}
\nc{\bx}{{\mathbf{x}}}
\nc{\bs}{{\mathbf{s}}}
\nc{\by}{{\bar{y}}}
\nc{\bw}{{\mathbf{w}}}
\nc{\bA}{{\mathbf{A}}}
\nc{\bK}{{\mathbf{K}}}
\nc{\bI}{{\mathbf{I}}}
\nc{\bB}{{\mathbf{B}}}
\nc{\bG}{{\mathbf{G}}}
\nc{\bC}{{\mathbf{C}}}
\nc{\bD}{{\mathbf{D}}}
\nc{\bP}{{\mathbf{P}}}
\nc{\bH}{{\mathbf{H}}}
\nc{\bM}{{\mathbf{M}}}
\nc{\bN}{{\mathbf{N}}}
\nc{\bV}{{\mathbf{V}}}
\nc{\bU}{{\mathbf{U}}}
\nc{\bL}{{\mathbf{L}}}
\nc{\bT}{{\mathbf{T}}}
\nc{\bW}{{\mathbf{W}}}
\nc{\bX}{{\mathbf{X}}}
\nc{\bY}{{\mathbf{Y}}}
\nc{\bZ}{{\mathbf{Z}}}
\nc{\bS}{{\mathbf{S}}}
\nc{\bSi}{{\bar{\Sigma}}}
\nc{\sA}{{\mathsf{A}}}
\nc{\sB}{{\mathsf{B}}}
\nc{\sC}{{\mathsf{C}}}
\nc{\sD}{{\mathsf{D}}}
\nc{\sF}{{\mathsf{F}}}
\nc{\sG}{{\mathsf{G}}}
\nc{\sK}{{\mathsf{K}}}
\nc{\sM}{{\mathsf{M}}}
\nc{\sO}{{\mathsf{O}}}
\nc{\sQ}{{\mathsf{Q}}}
\nc{\sP}{{\mathsf{P}}}
\nc{\sZ}{{\mathsf{Z}}}
\nc{\sfp}{{\mathsf{p}}}
\nc{\sr}{{\mathsf{r}}}
\nc{\sg}{{\mathsf{g}}}
\nc{\sff}{{\mathsf{f}}}
\nc{\sfb}{{\mathsf{b}}}
\nc{\sfc}{{\mathsf{c}}}
\nc{\sd}{{\ltimes}}

 \def\e{\epsilon}

  \nc{\vol}{{\mathop{\operatorname{\rm vol\,}}}}
\nc{\co}{{\mathop{\operatorname{\rm Core\,}}}}
  \nc{\gal}{{\mathop{\operatorname{\rm Gal\,}}}}
  \nc{\cl}{{\mathop{\operatorname{\rm cl}}}}
  \nc{\disc}{{\mathop{\operatorname{\rm disc}}}}
  \nc{\Sym}{{\mathop{\operatorname{\rm Sym}}}}
   \nc{\Aut}{{\mathop{\operatorname{\rm Aut}}}}
 \nc{\Spec}{{\mathop{\operatorname{\rm Spec}}}}
  \nc{\spec}{{\mathop{\operatorname{\rm Spec}}}}
\nc{\Ker}{{\mathop{\operatorname{\rm Ker}}}}
 \nc{\dom}{{\mathop{\operatorname{\rm dom}}}}
\nc{\End}{{\mathop{\operatorname{\rm End}}}}
 \nc{\Hom}{\operatorname{Hom}}
 \nc{\GL}{{\mathop{\operatorname{\rm GL}}}}
 \nc{\Id}{{\mathop{\operatorname{\rm Id}}}}
 \nc{\rk}{{\mathop{\operatorname{\rm rk}}}}
 \nc{\length}{{\mathop{\operatorname{\rm length}}}}
\nc{\supp}{{\mathop{\operatorname{\rm supp} \, }}}
\nc{\val}{{\rm val}}
 
\nc{\res}{{\mathop{\operatorname{\rm res}}}}

\def\Ind#1#2#3{{#1} {\downarrow}_{#3} {#2} }

\def\meet{\cap}
\def\union{\cup}

\def\si{\sigma}
\def\g{\gamma}
\def\G{\Gamma}

\def\m{\smallsetminus}

\nc{\seq}[1]{\stackrel{#1}{\sim}}

\def\inv{^{-1}}
\def\claim#1{\smallskip {\noindent \bf Claim #1.\quad}}

\def\beq#1{\begin{equation} \label{ #1}}
\def\eeq{\end{equation}}

\def\prf{\begin{proof}}
\def\pv{\end{proof} }
 \def\eprf{\end{proof} }

\def\acl{\mathop{\rm acl}\nolimits}

\def\liminv{\underset{\longleftarrow}{lim}\,}

 \def\lbl#1{ { \color{green}{[#1]}}  \label{#1}  }
 \def\lbl#1{     \label{#1}  }

\def\ba{{\mathbf{a}}}

 \renc{\b}{{\beta}}

\def\genl{\langle L_1 \rangle}

\def\Ind#1#2{#1\setbox0=\hbox{$#1x$}\kern\wd0\hbox to 0pt{\hss$#1\mid$\hss}
\lower.9\ht0\hbox to 0pt{\hss$#1\smile$\hss}\kern\wd0}

\def\two{\bf{2}}

 \def\CH{\CU _h} 
 
    \def\Te{T_{ext}}  
  
\setlist[itemize]{leftmargin=*}

\title{ Definability patterns and their symmetries \footnote{\today}}

\author{Ehud Hrushovski}
\begin{document}

\begin{abstract}  

 We identify a canonical   structure $\co(T)$ associated to any first-order theory $T$, 
reflecting   patterns of partial definability of types, uniformly in any type space over a model of $T$.   The core
generalizes the  imaginary algebraic closure  in a stable theory and the hyperimaginary bounded  closure in simple theories.  The core admits a compact topology, not necessarily Hausdorff, but the Hausdorff part can already be bigger than the Kim-Pillay space of $T$, and in fact accounts for the general Lascar group.

Using the core structure, we obtain simple  proofs of a number of results previously obtained using topological dynamics, but working one power set level lower.  
 The Lascar neighbour relation  is represented by a canonical   relation $L_1$ on  $\co(T)$;
the general Lascar group $G_{\Las}$ of $T$ can be read off this compact structure.  
       This gives  concrete form  to  results of Krupi\'nski, Newelski, Pillay, Rzepecki and Simon, who used topological dynamics applied to  large models of $T$ to show the existence of  compact groups mapping onto $G_{\Las}$.   

In an appendix, we  show that a construction analogous to the above but using infinitary patterns  recovers the Ellis group   of \cite{kns},  and use this to sharpen the cardinality bound for their Ellis group
from $\beth_5$ to $\beth_3$, showing the latter is optimal.    
              
              There is also a close connection to  another school of topological dynamics, set theory and model theory, centered around   the Kechris-Pestov-Todor\v cevi\'c correspondence.   We define the Ramsey property for a first order theory,  and show - as a simple application of the core construction  applied to an auxiliary theory -  that any theory  $T$ admits a canonical minimal Ramsey expansion $T\ram$.   This  was envisaged and proved 
for certain Fraiss\'e classes, first 
    by Kechris-Pestov-Todor\v cevi\'c for expansions by orderings, then by Melleray,  Nguyen Van Th\'e,  Tsankov and Zucker for more general expansions.  
We also show  that for  a complete theory $T$ in a countable language
  with prime model $M$, the universal minimal flow $F$ of $\Aut(M)$ can described as the space of expansions of $M$ to a model of $T\ram$.

\end{abstract}
       

  \maketitle

 \begin{section}{Introduction} 
 
\def\CU{\mathcal{J}}
 Among the gems uncovered in   Shelah's  work on stable theories, but applicable
 to all first order theories, not least was  Galois theory for imaginary algebraic elements (\cite{poizat}).   Following the introduction of imaginaries -  quotients by definable equivalence relations -  there is a duality between the definable closures of algebraic elements, and closed subgroups of finite index of a certain profinite group $\gal_{sh}$, the Galois group of the theory.  For Shelah this served as background to a fundamental result, the finite equivalence relation theorem, that will be  recalled below.
 
Extending the independence theorem for general simple theories, Kim and Pillay \cite{KP} required quotients by equivalence relations that are not definable, but only intersections of definable relations.  This again led to a beautiful Galois correspondence in any first order theory between bounded hyperimaginaries and subgroups of 
 a   compact Hausdorff group $\gal_{KP}$.       It was later incorporated into continuous logic \cite{byu}, and used in relating  finite combinatorial structures with Lie groups.

 Kim and Pillay were guided by   work of Lascar, who studied  small quotients of the automorphism 
group of large saturated models; he showed the existence of a maximal
such quotient group ${G_{\Las}}$.      Lascar denoted his  group   by $G$, writing
in parentheses:  `$G$ for Galois' and asking whether it coincides with the compact group $\gal_{KP}$. 
The latter question was answered negatively by Ziegler.  
The Galois nature of $\gal_{KP}$ has been clearly demonstrated; closed subgroups correspond to
definably closed subsets of the bounded (a.k.a. compact, a.k.a. algebraic) closure of $\emptyset$.  
But no evidence of the Galois
nature of the full group $G_{\Las}$ has emerged; it
does not take part in a meaningful Galois correspondence, and is not the automorphism group of any known structure associated with $T$. 
 
One can of course define, on each sort $V$ of a model $M$ of $T$, a {\em set} $V_{\Las}^M$ of Lascar types as a quotient of the type space of $M$;  but as no topology or algebraic structure is defined on it.  Given $T$ alone, 
only the cardinality (or  `Borel cardinality', see \cite{kms}) of this set is actually well defined; and certainly $G_{\Las}$ cannot be recovered from it.  
Moreover, the Lascar group  may leave no trace on any sort belonging to finitely many variables.

 Our aim is to find an alternative Galois group canonically associated with $T$,
that incorporates $G_{\Las}$ within it.

 Krupi\'nski, Pillay and  Rzepecki, \cite{kp}, \cite{kpr}, \cite{kr}      showed intriguingly that ${G_{\Las}}$ is (in many ways) a quotient group
of a compact Hausdorff topological group; this suggests that $V_\Las$ too is a quotient of
some canonical   space, carrying some structure deduced from $T$.  
In \cite{kns} a cardinality bound was found on the cardinality of the Ellis groups and thus in principle 
provided a canonical compact cover of ${G_{\Las}}$, though at some remove from definable sets of $T$.  (\cite{kns} found a cardinality bound
of $\beth_5$; we will show that the correct bound for their group is $\beth_3$.)  
 This followed a program initiated by Newelski and bringing to work the Ellis groups of topological dynamics, going through certain semigroups. 

Here we start over, in a sense going back to the original setting of the finite equivalence relation theorem.
  Morley   realized that type spaces {\em over models} carry information about a theory
that goes far beyond the space of types over $\emptyset$.  The difficulty, of course, is to extract model-independent information; Morley introduced 
a topology, with properties independent of the (sufficiently saturated) base model, that led quickly to the 
notion of $\omega$-stability and Morley rank.  Shelah saw that one can work with local type spaces:  for a finite set $\g$ of formulas and a distinguished variable $y$, consider
the Boolean algebra of formulas $\phi(x,b)$ with $b$ from $M$, $\phi \in \g$, and the Stone space $S_\g(M)$ \footnote{more properly denoted $S_{\g;x}$ to point out the distinguished variables; but we usually view this data as embedded in $\g$.}.  This led to stability.   In both cases, the ranks essentially  exhaust the information available from the topology alone.

We will enrich the topology to a relational structure on these type spaces, in a certain language $\LS$, the language of definable patterns of $T$.  We will then find a canonical structure $\co(T)$  - the universal pattern space of $T$ - organizing this
 information, depending on the theory alone.  $\co(T)$ is embeddable in the type space of any model, hence of cardinality $\leq 2^{|T|}$.  
  The automorphism group $G$ of $\co(T)$
thus acts on a geometry directly constructed from $T$. 

   $\CU=\co(T)$ is compact, but not necessarily Hausdorff; however each complete type of the core has a canonical compact Hausdorff quotient structure by
a quantifier-free definable equivalence relation, 
and can be viewed as an imaginary sort.   There is a compact Galois duality between these sorts
and their automorphism groups.

Let $\fg(T)$ be the 
group of {\em infinitesimal} elements of $\Aut(\co(T))$
 in its action on $\co(T)$, namely those that stabilize the closure of any open set.
Then $\CG=G/\fg$ is  compact Hausdorff quotient group of $\Aut(\co(T))$.  As $\co(T)$ is homogeneous for atomic types and may have a continuum of pairwise orthogonal ones, $\Aut(\co(T))$
can have cardinality $\beth_2(|T|)$; but in any case  we have $|\CG| \leq 2^{|T|}$ (\corref{size0}).

  We now come to the Lascar neighbour relation $L_1$; it holds between two elements of a
sufficiently saturated model if they  have the same type over some elementary submodel,
 see \ref{lascardistance}.   The pattern language $\LS$   can define it on $S(M)$;   it is represented on $\CU$  by the same formulas.   $L_1$ further induces a relation $L_1$ on  the Hausdorff part $\CH$.
This also determines a distinguished compact subset $L_1$ of $\CG$, namely  the automorphisms that move elements no further than to their Lascar neighbours.      The general Lascar group can then be interpreted as $\CG/ {\genl}$.    Of course  
at this point one may prefer  {\em not} to factor out ${\genl}$, but treat  
and $(\CG,L_1 )$  
as the right invariant  of $T$ in the  world of compact topological groups.   As a check, 
  while the Lascar group and Lascar strong types may not
be visible at the level of finitely many variables, $(\CG,L_1)$ behaves in the model-theoretically expected way, reducing as a projective limit of automorphism groups of the finitary spaces.
  
 \ssec{Definability patterns}
To explain the relational structure on $S_\g(M)$, recall the 
`fundamental order' of the Paris school presentation of stability theory \cite{lp}, and specifically the maximal classes of this order.
For each type $p(x)$ over a model $M$, and formula $\phi(x,y)$, we let $(d_p x)\phi(x,y)$ denote the set of $b \in M^y$ 
with $\phi(x,b) \in p$.  This is simply a subset of $M$.  In some cases it is $0$-definable, so that $(d_px)\phi(x,y) = \theta(y)$;
equivalently, the formulas $\phi(x,y) \& \neg \theta(y)$ and $\neg \phi(x,y) \& \theta(y)$ are {\em omitted} or {\em not represented}
in $p$, meaning that no substitution instance lies in $p$.  A type $p \in S(M)$ is {\em maximal} in the fundamental order
if   no type represents a strictly smaller  set of formulas.

%

     More generally, given a $k$-tuple $(p_1,\ldots,p_k)$ of   types, 
a $k$-tuple $(\phi_1,\ldots,\phi_k)$ of formulas in matching variables, and  a formula $\alpha$,
we can say that $t=(\phi_1,\ldots,\phi_k; \alpha)$ is represented in $(p_1,\ldots,p_k)$ if
for some    $b \in \alpha(M)$ we have $\phi_1(x,b) \in p_1,\cdots,\phi_k(b) \in p_k$.   We let $\R_t$
be a relation symbol, asserting that $t$ is {\em not} represented.  This we view as a $k$-ary relation on any type space $S(M)$; forming a language $\LS$.   Let $\TS$ be the universal theory of $S(M)$ with this structure; it does not depend on the choice of $M$.

\begin{thm}  \label{intro1}
 ${\TS}$ has a unique universal existentially closed model $\CU$.  It has cardinality at most $\lambda_T$, the number
 of finitary types of $T$.   The automorphism group 
 $\Aut(\CU)$ has a canonical compact topological group quotient $\CG=\Aut(\CU_h)$, where 
 $\CU_h$ is the union of the Hausdorff imaginary sorts of $\CU$.  There exists 
a canonical surjective homomorphism $\CG \to L$, where $L$ is the Lascar group of $T$, with
  compactly generated kernel.
\end{thm}

We will call $\CU$ the {\em core} of $T$; though the essential point is that it is a relational structure,
rather than simply a topological one.       The conjugacy class of a tuple in $\CU$ is determined by the atomic type in $\LS$;
such types will be called {\em   pattern types}.  They will  be defined more syntactically below.   
 
 It may happen that an atomic 2-type of $\LS$ restricts to the same 1- type  $\rho$ in each coordinate, but nevertheless
 includes partial definability relations that rule out equality of the two $1$-types over a model.  In this case,
 two copies of $\rho$ must be included in $\CU$.   It is the symmetry between them that the
 group $G=\Aut(\CU)$ expresses.  In particular when $G=1$ (and only then), $\Aut(\CU)=1$, $\CU$ reduces precisely to the fundamental order.  In general the maximal elements of the fundamental order on a given sort can be viewed as the type space of the corresponding sort of   $\CU$.

As an example, consider an antireflexive relation $R(x;y,z)$.  Let $T$ be the model completion
(the only rule is $\neg R(x,y,y)$.)     We consider
type spaces in this single relation, with distinguished variable $x$ (the case of complete types is not really different.)    Then $1$-types $tp(a/M)$ over a model $M$ describes a directed graph  on $M$, 
defined by $R(a,y,z)$.   Here $\CU$ will have four elements, corresponding to the empty graph, the complete graph, and two copies of a "linear ordering" \footnote{(the pattern type represents the axioms of linear orderings, rather than an ordering on any specific set.)}.   Taken individually there is nothing more to say about the linear orderings, but taken as a pair, $\CU$ asserts that one is precisely the opposite ordering of the other.
The evident symmetry of the two orderings is in this case the automorphism group of $\CU$.

Evidently $\CU$ knows  about
Ramsey's theorem.   Ziegler's examples alluded to above yield other examples of finite $\CU$, permuted by other cyclic groups.
 This connects to the work of another large school connecting  model theory with topological
 dynamics,  around the Kechris-Pestov-Todor\v cevi\'c correspondence \cite{KPT}.   
   There
has been relatively limited interaction between them so far; notably \cite{ks} show 
that $\aleph_0$-categorical structures with the Ramsey  property admit a functorial joint embedding property, and in particular have trivial Lascar group.    We return to this   below, 
extending the connection to     arbitrary
first-order theories.   As in \cite{HKP}, a weaker property than Ramsey suffices, 
  namely total definability of $\CU$ for $T$ itself rather than the `second-order' expansion $T^*$ that forms
 the substrate for Ramsey theory.

 \ssec{}
First order logic will be used in this paper at three levels of generality.  The most basic is a complete first order theory $T$.  It will be convenient to Morley-ize it, i.e declare all formulas to be atomic.  This done, $T$ becomes the model completion of the universal part $T_\forall$ of $T$, and thus carries the same information.

More general is the setting where we are given only a universal theory $T$; we will be interested especially in the class of existentially closed models, that may or may not be elementary.    If the class of models of $T$ has the joint embedding property, we will say it is {\em irreducible}. 
(Equivalently, $T$
is the universal theory of some structure; or the universal theory of a class of models with the joint embedding property.) 
 If $Mod(T)$ admits amalgamation, we will say it is a Robinson theory.    
 
 On the other hand given a complete first order theory, we may
Morleyize it by adding names for each definable set;  it becomes a  relational language with   quantifier-elimination.
A theory with quantifier elimination is completely determined by the universal part.

 Thirdly, we will consider {\em primitive universal theories},  described in the next section;  where closure under negation is not assumed.   Our main task is to describe algebraic closure in the positive setting; it goes a little beyond the bounded
closure of \cite{KP} or the (compact) algebraic closure of \cite{byu}.   

The word 'definable', unqualified, will always mean:  definable without parameters.

\ssec{Patterns and definable types}  It is also possible to introduce the core   as a relational structure by means of a direct description of its types  spaces.  
 
Let $ {T}$ be an irreducible universal theory, with a distinguished sort $V$. We also fix a `parameter sort' $P$, and assume $\g$ is a
pattern sort, i.e. set of  formulas $\phi$  on $P \times V$,  including all formulas on $V$ alone.  \footnote{The focus on $(P,V;\gamma)$ allows a more elementary description, but  does not lose any generality, since we allow $V$ to be a projective limit of sorts; we could deal with any family of sorts by taking products.}

  Let $L_V(X)$ be the language $L$ of ${T}$, and augmented with some additional predicates   $X_1,\ldots,X_m$, standing for   subsets of $V$.   We will write $X=(X_1,\ldots,X_m)$.


Assume given a universal theory $\Te$  of $L_V(X)$, restricting to ${T}$ on $L$.   
 A {\em  (maximal)  pattern  type} for this data is a maximal universal theory $p$ for $L_V(X)$,  containing $\Te$.  

The   case  we will be concerned with in practice is the theory $\Te$ of  $\gamma$-externally definable sets relative to $T$;
this will   correspond to a type of an element of the core at $V$.  Assume here that $\gamma=\{\gamma_1,\ldots,\gamma_n \}$, with variable $x$ and parameter sort $V$; and let   $T_\gamma^{ext}$ be the $L_V(X)$-theory, whose models are the structures $(M,A_1,\ldots,A_n)$ such that  there exist $M\leq N \models T$ and $c \in P(N)$ with $A_i =  \{v \in V(M): N \models \gamma_i(c,v) \}$.     

 Let $M \models T_V, A \subset V(M)$.  
 An $L_V$-universal theory $p$ is {\em finitely satisfiable} 
 in $(M,A)$ if for any existential sentence $\psi$ true in $M$,
 there exists a (finite) substructure $M_0$ of $M$ with $M_0 \models \psi$, and such that $(M_0,A) \models p$. 
  
  Equivalently, there exists an elementary extension $(M^*,X^*)$ of $(M,X)$ and an embedding $f:M \to M^*$, such that $(M, f \inv(X^*)) \models p$.
     
 Let us say that two universal sentences $(\forall x)\psi(x),(\forall y)\phi(y)$ of $L(X)$
are {\em incompatible} if along with $\Te$ they jointly imply a universal sentence of $L$, not already in $T$.

If $p$ is a   type in the language of patterns, maximality amounts to this:  
for any quantifier-free  formula $\phi(x_1,\ldots,x_n)$ of $L(X)$,
{\em either} $(\forall x)\phi(x) \in p$, {\em or} some incompatible universal sentence $(\forall y) \psi(y)$ lies in $p$.

\begin{lem} \label{ramsey-basic}  Let $M \models T$, and let  $A \subset V(M)$ be externally $\g$-definable.
Let $p_0$ be a universal theory of $L(X)$, with $(M,A) \models p_0$.
Then   some  $\g$- pattern type $p \supseteq p_0$ is finitely satisfiable 
 in $(M,A)$. 
  \end{lem}

\prf   By Zorn's lemma, there exists a   universal theory $p \supseteq p_0$ of $L(X)$ that is finitely satisfiable 
 in $A$, and is maximal with this property.  
 We have to show that $p$ is a 
 pattern type.   Let $\phi(x_1,\ldots,x_k)$ be a quantifier-free formula of $L(X)$.  
 Then $p \union \{(\forall x)\phi\}$ is either equal to $p$, or is no longer finitely satisfiable 
  in $(M,A)$. 
 In the latter case, by definition, there exists an   $L$-formula $\theta(x_1,\ldots,x_m)$ consistent with $T$,
 such that if $M \models \theta(a_1,\ldots,a_m)$ and $M_0 = \{a_1,\ldots,a_m\}$ then  $\neg \phi$ is realized in $(M_0,A)$.    Let  
\[ \psi(x_1,\ldots,x_m) =\theta(x_1,\ldots,x_m) \to  \bigvee_{y_1,\ldots,y_k} \neg \phi(y_1,\ldots,y_k) \]
where $y_1,\ldots,y_k$ range over $k$-tuples from among  $x_1,\ldots,x_m$.   Then $p \models \psi $,
and $(\forall x)\psi$, $(\forall x)\phi$ are incompatible.   \eprf

A {\em definable} pattern type is  one that simply asserts that $X$ coincides with some $0$-definable set of $T$ (by a qf formula without parameters).   Maximality is then clear.
We say $T$ is {\em has definable patterns} (for a given pattern sort $\g$) if  every   maximal pattern type (in the sort $\g$) is definable.

  \begin{thm} \label{expand} Let ${T}$ be an irreducible universal theory.
There exists a unique minimal expansion of ${T}$ to a universal theory $T\td$ 
that   has definable patterns.  
The self-interpretations of $T\td$ over $T$ form a group, isomorphic to $\Aut(\CU)$.
\end{thm}

If $T$ is stable, then $T\td$ simply names the imaginary algebraic constants, and so amounts to  `working over $\acl^{eq}(\emptyset)$' in the
sense of Shelah.  If $T$ is NIP, then $T\td$ is NIP.


The proof, and precise definition of minimality and uniqueness, will be given at the end of \secref{section-rel} (\propref{expandd}.)

\ssec{Elementary   Ramsey theory}

Structural Ramsey theory is usually defined in terms of   isomorphism types of substructures, or complete types.   See \cite{glr}, \cite{nr}, \cite{ehn}  and references there. \footnote{Krzysztof Krupi\'nski informed me that
with Anand Pillay, he developed an extension of \cite{KPT} to such a setting.}  
  However it is also very natural
  in a   first-order   setting,   using formulas in place of complete types.     
  This extends to continuous logic, and brings out the unity in  instances of structural Ramsey theory such as   
 for affine spaces over a finite field, Dvoretzky-Milman for Hilbert spaces, and Van den Waerden for arithmetic progressions.     
   
 In Ramsey's theorem, one is given a set $M$.  We then consider $D=M^n$ or $D=M^{[n]}$,
 and an {\em arbitrary} subset $A$ of $D$.   The desired  outcome is  a `large' subset $M_0$ of $M$,
 such that $A$, restricted to $D(M_0)$, has a simple, explicitly described structure.   
 
 In the   structural generalization, $M \models T$ is a structure.   $D$ is a  sort (or a {}definable set of $M^n$).  Again $A$  
 is an arbitrary subset of $D(M)$.  We seek a large $M_0$ such that $A$ has as regular  a structure as possible on $D(M_0)$.   Here `large' means:  a finite substructure of $M$ realizing a prescribed existential sentence of  $T$.
 Equivalently, $M_0$ can be taken to be a copy of $M$ in some ultrapower $(M^*,A^*)$ of $(M,A)$.

While Ramsey theory appears to be second-order, considering arbitrary colorings on $D$,
 a simple device does present it as a special case of the theory of patterns:  we simply introduce a new 
 sort $P$ and a relation that allows $P$ to parameterize colorings on $D$, with no constraints.  A   pattern type associated
 to the new theory will be referred to as a free pattern type for the original one.   Thus a free  pattern type is simply a  maximal universal theory  for $L_V(X)$,  whose restriction to $L$ is ${T}$.

 Now \lemref{ramsey-basic} tells us immediately what 'as regular as possible' can mean.  We cannot do better
 than a free pattern type, and some free pattern type can always be achieved.  Thus the basic structural Ramsey question  for a theory   $T$ changes from a yes/no question to a more qualitative and functorial one:  describe the free pattern types of $T$.   The simplest case is that every free pattern type is {\em definable}; in this case we will call the theory {\em Ramsey}.    We will see that 
 every universal theory has a canonical Ramsey expansion, whose automorphism group is an interesting invariant of $T$.

   Expand the language by an  an additional sorts $P$  
 and  binary relation $R \subset P  \times V$. 
   However we keep the same universal theory 
 $T_\forall$, adding no axioms on $R$; we denote it $T^*_V$.   Clearly $T^*_V$ is an irreducible universal theory.
    We define  $\co\ram(T): = \CU^*:=\co_R(T^*_V)$, i.e. the parameter sorts of the core of the derived theory $T^*_R$.
 We also write $\LS\ram(T)$ for the language of $\co\ram(T)$.

        Let $\g=\g\ram(T)$ be generated by finitely many formulas 
  $R(x_i,y)$ and arbitrary formulas of $L$.  A 
   $\g$-pattern type
    for $T^*$   will
    be called a {\em free pattern type} for $T$ at $V$.

\begin{defn} \label{ramsey-def}
 We say that a  theory $T$ is a {\em Ramsey theory}  at  $V$ (or has the Ramsey property at $V$)   if   all free  pattern types for $T$ on $V$ are  definable.     It is {\em everywhere Ramsey}   if it is Ramsey at   $V$
for  all  sorts of $T$.
 
 \end{defn}
  
In view of \lemref{ramsey-basic}, an equivalent form: 

$T$ is Ramsey at $V$ iff  for every $M \models T$ and every $A \subset V(M)$, there exists a formula  $\theta(x)$ such that 
for  every formula $\phi(x_1,\ldots,x_n)$ consistent with $T$,   there exist
$a_1,\ldots,a_n \in M$ with 
 $M \models \phi(a_1,\ldots,a_n)$, and   \[ M \models \theta(a_i) \iff a_i \in A \] 
 Equivalently, some elementary extension $(M*,A*)$ of $(M,A)$ contains a copy $M'$ of $M$ such that
 $A* \meet M'$ is a $0$-definable subset of $M'$.

 \begin{rem}   Assume $T$ is $\aleph_0$-categorical, with quantifier-elimination in a relational language, and with a single sort $V$; let $M \models T$.
 The above definition  relates to the terminology of   \cite{KPT}, \cite{ehn} in the following way.    
Let ${\mathcal A}$ be the class of finite models of $T_\forall$.
  For any $A \in \mathcal{A}$, we have an imaginary sort $V^A$ coding embeddings of $A$ into $V$;
so $V^A(M)$ can be canonically identified with  $\Hom(A,V(M))$.  We also have $V^{[A]} = V^A/Aut(A)$, the set of substructures of $V$ isomorphic to $A$. 
Now $T$ (or rather the class of finite models of $T$) has the Ramsey property
in the sense of \cite{KPT}, \cite{ehn} iff $T$ has the Ramsey property at $V^{[A]}$ for all $A \in \mathcal{A}$.  Also by the KPT correspondence,
$\Aut(M)$ is extremely amenable iff   $T$ has the Ramsey property at $V^A$ for all $A$ (or equivalently at $V^n$ for all $n$.)     Note that $V^A, V^{[A]}$ are complete types; so definability of the coloring amounts to constancy.
  \end{rem}

\begin{thm} \label{expand2} Let $T$ be a complete theory.  
There exists a unique minimal everywhere Ramsey expansion $T\ram$.

The self-interpretations of $T\ram$ over $T$ form a group, $G\ram(T)$.
  \end{thm}
  
  See \propref{expand2c} for a precise definition of minimality and uniqueness.
  
In case $T$ is the theory of pure equality, 
$T\ram$ will be the theory of dense linear orderings
(up to a   strong bi-interpretability).      In general, 
   $T\ram$  is is a complete first order theory in a bigger language, whose additional relations
 are indexed by the elements of the dual sorts of the theory $T^*$.   The proof and various examples 
 are in \secref{ert-sec}.

     In Appendix B, we will show
 also that when $T$ has countable language and a prime model $M$ ,  
 $G=\Aut(M)$, the space of expansions of $M$ to $T\ram$ is the universal minimal flow $U$ of $G$,
 i.e. it has no closed $G$-invariant subspaces, and admits a continuous $G$-invariant map into any other 
 compact space with this property.   We can also write:  $\Hom(J^*, S(M)) = U$.  
 Assuming $L$ is countable,  $J$ can have cardinality continuum, but the statement nonetheless has a taming effect on 
 $U$;   whereas a priori we know only that $|U| \leq \beth_2$, $J$  is parameterized by 
a compact structure of cardinality continuum, uniquely determined by its own theory, which in turn is  closely controlled by $T$.

 \thmref{expand2}, and the result of the Appendix,  generalize a line of theorems, beginning with   the theory of pure equality by Glasner and Weiss; then for a wide class of $\aleph_0$-categorical theories by \cite{KPT}, explaining the  connection to structural Ramsey;
the restriction to  $\aleph_0$-categorical theories, or Fraiss\'e classes, was due to the approach using 
topological dynamics of the automorphism group $G$.   The result was extended to  all 
 $\aleph_0$-categorical $T$  provided   the universal minimal flow of  $G$ is metrizable, 
in   \cite{mnt},\cite{zucker}, \cite{bmt}.   In these cases, $T\ram$ has  locally finite  language relative to $T$ (see \remref{bmt/z}.)    \cite{ehn}  showed  that the hypothesis is not always valid.

  \ssec{The Ellis group}  
In Appendix \ref{ellis-sec} we consider  a   type-definability analogue $\CUb$ of $\CU$.   It relates to  the notion of {\em content} of \cite{kns},  as   $\LS$ relates to the   fundamental order  of \cite{lp}. 
  We show that $\Aut(\CUb)$ is isomorphic to the  Newelski-Ellis group.  This   presents the Ellis group 
as the automorphism group of a natural structure, and  leads immediately to a bound of $\beth_3$,   improving the bound $\beth_5$ of \cite{kns}.  We show by example that $\beth_3$ is in fact optimal for the Ellis group.

\ssec{Open problems}
Many directions are left open; here are a few.  
\begin{enumerate}

\item   Develop a { relative} theory, with $\co(T_a)$   parameterized over a definable set.   
\item Develop the Galois correspondence.  See \ref{beth}, but also:
\item  There exists a 1-1 correspondence between a family of subgroups
of  $\Aut(\CU)$, and a family of reducts of $T\ram$ containing $T$; describe the closed subgroups and the closed expansions of $T$
\item Under what circumstances, apart from \propref{ramseyqe},  does the  definable-pattern expansion $T\td$ have a model completion?   
\item  Connection with NIP and with honest definitions.   Characterize the 
pattern types. Show that $T_{def}$ is the universal part of a complete theory with quantifier-elimination, if $T$ is NIP.
\item   Let  $D$ be strictly minimal.   Is the canonical Ramsey expansion at $D$,  a NIP theory?  
\item Develop $\co(T)$ for a primitive universal theory $T$.
\item  Explore further the duality between the parameter and variable sorts. 
  \item Develop  the continuous logic generalization of \cite{H-cigha}, i.e. generalized finite imaginaries
to generalized compact Hausdorff imaginaries.  Does the   Hausdorff part of $\CU$
comprise the entire generalized algebraic closure in this sense?  
\item The same  
for positive logic.  Is $\co(T)$  the entire absolute algebraic closure?  
  \item  The definable group analog; a more concrete description of canonical compact covers of
$G/G^{000}$ for a definable or ind-definable group $G$ should become accessible.
This may shed light on the Massicot-Wagner problem of describing general approximate subgroups. 
\item Investigate the degree of effectiveness of $\co(T)$ or rather of the pattern type spaces that determine it.    
(See \remref{complexity}.)
 
 \item      We have not ruled out that $G=\Aut(\CU)$ is outright Hausdorff, i.e. $\fgt=1$;  but see 
  Example \ref{topdyn}.     Nov. 2021 update: this is now shown in Example \ref{Gactions-e}.  For any  discrete group $\G$ we let $T_\G$ be the theory
  of free $\G$-actions.    We show that the  universal minimal flow of any discrete group $\G$ is dual to $\CU(T_G)$, and use this to give 
  an example of a  pp type with non-Hausdorff automorphism group.   Many interesting questions on this connection with topological dynamics remain; see Question   \ref{dynamics}.

\end{enumerate}

 \bigskip

The constructions in the body of the paper are  entirely self-contained;   beyond some elementary lemmas on Hausdorff quotients (Appendix C), no topological dynamics is used.  
Only in the appendices,    where we describe the universal minimal flow and the Ellis group of the type space flow in our terms, 
do we assume knowledge of the definitions of these objects.

I am grateful to  Todor Tsankov and to David Evans for very enlightening conversations on the KPT correspondence;
and to Arturo Rodriguez Fanlo,   Pierre Simon and the Jerusalem group (Christian d'Elb\'ee,  Itay Kaplan, Yatir Halevi,  Tingxiang Zou, Eugenio Colla, Ori Segel)
for their reading and comments on the text.

\end{section}
 
\begin{section}{Existentially closed models}
 
\def\U{\mathcal{U}}

 Following work of Shelah, Pillay, and Ben-Yaacov \cite{lazy}, \cite{ec},\cite{cats}, the  setting of existentially closed structures of universal theories,  and of positive logic, has come to be viewed as a natural and mild generalization of the usual first order context.  
 In particular basic stablity and sometimes simplicity was thus generalized by these three authors.  
 For the more basic theory of saturated models
 this was carried out earlier by Mycielski, Ryll-Nardzewski and Taylor \cite{mycielski}, \cite{mrn},  \cite{wtaylor}, soon after the work in the first-order case by J\'{o}nsson, Keisler, and Morley-Vaught (see \cite{morley-vaught}).
  In this section we give a self-contained treatment of the facts we need, in current terminology.
 
 Let $\CL$ be a (many-sorted) language.  \footnote{We may or may not wish to 
 allow a logical equality symbol in general; but we will use it in the definition of an existentially closed structure.
 In practice, \remref{2.2} (2) will allow us to restrict attention to formulas constructed
 without the equality sign, in particular in the definition of the pp topology on an e.c. model.}   A  {\em positive primitive} (pp) formula is one of the form $(\exists x_1,\ldots, x_k)\bigwedge_{j=1}^l \phi_j(x)$, with $\phi_j$ atomic.   We regard   pp formulas as the fundamental ones for $\CL$,
  though occasionally we will consider slightly higher ones.
A theory axiomatized by negations of {pp} sentences will be called
{\em primitive-universal}.   The set of such sentences true in a structure $M$ is denoted $Th_{p\forall}(M)$.
By \lemref{u}, for existentially closed models, the full universal
theory (in the usual sense) is completely determined by the primitive universal theory.

A   {primitive}  universal theory $\CT$ of $\CL$  is called {\em irreducible} if it is the   {primitive}  universal   theory of some model.  Thus if $\CT \models \alpha \vee \beta$, with $\alpha,\beta$ primitive universal, then  $\CT \models \alpha$ or $\CT \models  \beta$.  Equivalently, any two models of $\CT$ admit  homomorphisms into some third
model of $\CT$. Note that since $\alpha,\beta$ are primitive universal, this can be a weaker condition than irreducibility as a universal theory.

We will consider irreducible   theories  $\CT$, and will be  interested only in such models.    In other words we are really concerned with $\CT^{\pm}:= \CT \union \{\psi:  \CT \not \models \neg \psi \}$ where $\psi$ ranges over {pp} sentences.

\begin{defn}  A model $A$  of $\CT$ is {\em existentially closed} (abbreviated e.c.) if for every homomorphism $f: A \to B$,
where $B \models \CT$, and 
any $\CL_A$-{pp} sentence $\phi$ allowing equality, if $B_A \models \phi$ then $A_A \models \phi$.  Here $\CL_A$ is $\CL$
expanded by constants for the elements $a \in A$; they are interpreted as $a$ in $A_A$ and as $f(a)$
in $B_A$.  \footnote{If we begin with a logic without equality, allowing equality in $\phi$ has the effect of considering only  e.c. models where two elements with the same atomic type over the entire model are equal;   this is needed for a reasonable definition of the cardinality
of the model.  
Any model can   of course be collapsed to one with this (`Barcan') property, losing no meaningful information.} \end{defn}

The usual direct limit construction shows that any model $A$ of $\CT$ admits a homomorphism $f: A \to B$ into an existentially closed model $B$ of $\CT$, with $|B| \leq |\CL_A| + \aleph_0$.   Any existentially closed model of $\CT$ is a model of $\CT^{\pm}$.

\begin{rem}  \label{2.2} \begin{enumerate} \item  Any  homomorphism from  an e.c. model of $\CT$ to a model of $\CT$ must be injective,
and indeed an embedding, i.e. an isomorphism onto the image.  
\item  Let $E$ be a conjunction 
of {pp}-formulas.  Assume it is a strong congruence:  in any model $A$ of $\CT$, and for any  non-logical
symbol $R(x,x_1,\ldots,x_n)$ in the language, if $a,b \in A$ and $A \models E(a,b) \wedge \phi(a,a_1,\cdots,a_n)$ then $A \models \phi(b,a_1,\cdots,a_n)$.   Then $E$ coincides with equality in any existentially closed model of $\CT$.
(Such an $E$, a conjunction of basic formulas in fact, will exist in the theories of interest to us.
 and will save us the need to consider separately atomic formulas built from the equality symbol.)
\end{enumerate}
\end{rem}

 Let $\CT$ be a primitive-universal  theory.   For two pp formulas $\phi(x),\psi(x)$, write $\phi \perp \psi$
 if $\CT \models (\neg \exists x)(\phi \wedge \psi)$.  Part (1) of the following syntactical lemma is the substitute for the law of excluded middle.   Part  (2)   refers briefly to possibly infinitary sentences,  beyond the pp level.
 It follows from (2) that any two e.c. models of $\CT$ share the same universal theory, and the same universal quantifications of Boolean combinations of pp formulas.

 \begin{lem}  \label{u}  Let $\CT$ be the {primitive} universal  theory of $M$, and let $E$ be an e.c. model of $\CT$.
 Let $\phi, A_i (i \in I), B_j (j \in J)$ be {pp} formulas of $\LS$.   Assume  $I,J$ are finite, or more generally that any set of cardinality $|I \union J|$ of pp formulas that is finitely satisfiable  in $M$ is satisfiable in $M$.    
   Then: \begin{enumerate}
 \item   If $a \in E^x$,   then either $E \models \phi(a)$ or $E \models \phi'(a)$
 for some $\phi' \perp \phi$.  
 \item    Let $\psi$ be the (possibly infinitary) sentence:  \[(\forall x)(\bigwedge_{i \in I} A_i\to \bigvee_{j \in J} B_j(x)) \] 
 If $M \models \psi$, then $E \models \psi$.
 \item If $M$ is $|E|^+$-pp-saturated, and $M \models (\forall x) \theta$ where $\theta$ is any
 (possibly infinitary) Boolean combination of pp formulas, then $E \models (\forall x)\theta$.
 \item  If $\theta$ is any finite Boolean combination of pp formulas and $M \models (\forall x) \theta$,
 then $E \models (\forall x) \theta$.
%
 \end{enumerate}  
   \end{lem}
 \prf     
 (1)   Say $\phi(x) = (\exists y)\psi(x,y)$, with $\psi$ atomic.  Let  $\Delta$ be the atomic diagram of $E$,
\footnote{the diagram consists of atomic sentences of $L_E$ true in $E$.} 
 and $\Delta' = \Delta \union \{\psi(x,c) \} \union \CT$.  If $\phi(a)$ fails, then $\Delta'$ is
 inconsistent since $E$ is e.c..  So some finite conjunction $\psi'(a,d)$ is inconsistent with
 $\{\psi(x,c) \} \union \CT$.  Let $\phi'(x) = (\exists u)\psi'(x,u)$.   It follows that  $\phi \perp \phi'$
 and $E \models \phi'(a)$.   \\
 
 (2)  We prove the contrapositive.    Suppose  that  there exists $c$ with $E \models A_i(c)$ for each $i$, but $E \models \neg B_j(c)$ for each $j$.    Since $E$ is e.c. and $B_j(c)$ fails, there must exist a {pp} formula $B_j' \perp B_j$ such that $E \models B_j'(c)$.  
Then for any finite $I_0 \subset I, J_0 \subset J$,  $\{A_i(x), B_j'(x): i \in I_0, j \in J_0\} $ is 
satisfiable in $E$, so $(\neg \exists x) \bigvee_{i \in I_0} A_i \vee \bigvee_{j \in J_0} B'_j$ is {\em not}
a sentence of $\CT$.  Hence  $\{A_i: i \in I \} \union \{B'_j: j \in J\}$ is finitely satisfiable and hence satisfiable
in $M$.  But this  means the implication does not hold in $M$.

(3)  Under the assumptions, there exists an embedding $h:E \to M$; pp formulas and their negations are preserved by $h$ since $E$ is e.c.; hence arbitrary Boolean combinations are preserved; and universal quantifiers descend to substructures, as usual.

(4) Let $M^*$ be an $|E|^+$-saturated extension of $M$.  Then the pp-theory of $M^*$ is $\CT$ and $M^* \models (\forall x) \theta$; so (3) applies, and $E \models (\forall x) \theta$.
 
 \eprf
 In particular  if a pp formula $R$ has   $k$ distinct solutions in some e.c. $M \models \CT^{\pm}$, then $R$ has $k$ distinct solutions
 in any $M \models \CT^{\pm}$.  Thus  the number of solutions of $R$  is any e.c. model (viewed as a finite number or $\infty$) is the same.   
 

\ssec{Morleyzation} \label{morleyzation}   

It is sometimes desirable to modify the language by a definitional expansion, so that every 
pp formula becomes atomic.\footnote{In the present paper, this section is used only in \secref{quotients} and in the footnote of Appendix A,
but not in any of the main results.}    This can be done without changing the category of e.c. models.

Let us see this for  the existential quantifier.  Let $\phi(x,y)$ be atomic in the language $\CL$, and let
$\CL^+=\CL \union \{\Phi(x)\}$, where 
$\Phi(x)$ intended to stand for  $(\exists y)\phi(x,y)$.    Let $\CT^+$ be the theory 
consisting of all sentences $\neg (\exists x_1,\cdots,x_n) \bigwedge _j Q_j$, where each $Q_j$ is
a symbol of $\CL^+$, and where if $Q_j $ is replaced by $(\exists y)\phi$ in each occurrence, we obtain
a consequence of $\CT$.    For any $M \models \CT$, define a $\CL$-structure $M^+$
by interpreting $\Phi$ as $(\exists y)\phi(x,y)$.  

\claim{}  
$M \mapsto M^+$, $N \mapsto N|\CL$  define  a 1-1 correspondence between e.c. models
of $\CT$ and of $\CT^+$.
\prf Let $N$ be an e.c. model of $\CT^+$,  $M = N | \CL$, and let $a \in \Phi(N)$.  Then the definition of $\CT^+$ implies that 
$(\exists y ) \phi(x,y)$
is   consistent with $\CT$ along with any pp formula true of $a$.  Since $N$ is e.c., we have $N \models 
(\exists y ) \phi(x,y)$.  Conversely, if $N \models (\exists y ) \phi(x,y)$, 
then the axioms  $\CT^+$ continue to hold if we modify $N$ by setting $\Phi(a)$ to be true
(since in any potential counterexample to an axiom, replacing each occurence of $\Phi(a)$ by 
$ (\exists y ) \phi(x,y)$ would yield a counterexample in $M$ to an axiom of $\CT$.)  Again since
$N$ is e.c., we have $\Phi(a)$.  So we have $N=M^+$.    

Next let us see that $M$ is e.c.  Let $f: M \to M'$ be a homomorphism.   Then $f$  extends to a homomorphism
$f:M^+ \to (M')^+$, and the existential closudness of $M^+=N$ immediately implies the same for $M$.

Conversely, assume $M$ is e.c.  Then $M  = (M^+)|L$.  It remains to show that $M^+$ is e.c.
Let $g: M^+ \to N$ be a  $\CL^+$- homomorphism.      To prove the Tarski-Vaught property,
i.e. existential closedness of $M^+$ with respect to this map, we may compose $g$ with any homomorphism $N \to N'$.  So we may assume $N$ is e.c.; and thus by the above, $N\models \Phi  \iff (\exists y)\phi$.
This easily implies the  existential closedness of $M^+$.
\eprf
In particular, $\CT^+$ is irreducible if $\CT$ is; and $\CT^+$ eliminates the quantifier in $(\exists y)\phi(x,y)$.

 One could similarly deal with   finite disjunctions.  If $P$ is added to stand for $P_1 \vee P_2$,
 the axioms would be $\neg \exists x \bigwedge_j Q_j$, where each $Q_j$ is $P$ or an existing symbol,
 such that replacing each $P$ with $P_1$ or $P_2$ (chosen arbitrarily) yields a consequence of $\CT$.   
 For $M \models \CT$ we define $M^+$ naturally, and show as above that an e.c. model $N$ of $\CT^+$ 
 has the form $M^+$, with $M$  an e.c. model $M$ of $\CT$.    Conversely if $M \models \CT$ is e.c.,
 then $M^+$ is e.c.; for if $f: M^+ \to N$ is a homomorphism, $Th_{p\forall}(N)=\CT^+$, we may assume $N$ is e.c., etc.

   In the setting of $|L|^+$-pp-saturated e.c. models one can even
 eliminate  an {\em infinite} conjunction, $\bigwedge_i P_i$, by introducing a symbol $P$ for it,
 obtaining a language $\CL^+$.   (For simplicity we consider a single conjunction, but any family can be handled in the same way.)
   We let 
   \[\CT^+ = \{ \neg (\exists x_1,\cdots,x_n) \bigwedge _j Q_j' :   \CT \models \neg (\exists x_1,\cdots,x_n) \bigwedge _j Q_j' \} \]
  where $Q_j'= Q_j$ if $Q_j$ is not one of the $P_i$, and $Q_j' \in \{P_i,P\}$ if $Q_j=P_i$.  
Note that $\CT^+$ contains $\CT$, and that in any model $N$ of $\CT^+$, if we re-interpret $P_i$ by
as $P_i(N) \union P(N)$, we obtain again a model of $\CT^+$.

Any model $M$ of $\CT$ expands canonically to a model $M^+$ of $\CT^+$,
with  $P$   interpreted as $\meet_i P_i$.  (If $M$ is $|\CL|^+$- saturated, and then the universal primitive theory of $M$ is precisely $\CT^+$, showing that $\CT^+$ is irreducible  if $\CT$ is.)

In case $M$ is e.c., this is the largest possible interpretation of $P$:
if $P'$ is another, then $P'$ implies $\neg Q$ for every  $Q \perp P_i$, so $P'$ implies $P_i$ for each $i$,
and hence $P'$ implies $P$.  Moreover $M^+$ retains the property that every homomorphism on $M^+$ is an embedding:    if $f: M^+ \to N$ is a homomorphism,
then $f$ is an $L$-embedding on $M$, and $P' = f \inv(P(N))$ is a possible alternative interpretation of $P$,
containing $P(N^+)$, so equal to it; hence $f$ is an $L \union \{P\}$-embedding.   In particular, endomorphisms
of $M^+$ are automorphisms.    (However, we do not necessarily have $M^+ \models (\CT^+)'$, if $M$ is not 
sufficiently saturated; and in particular $M^+$ may not be e.c.  This issue disappears if the conjunction is finite.)

Conversely, let $N$ is an e.c. model of $\CT^+$, $M$ the reduct to a model of $\CT$.    As noted
above, reinterpreting  $P_i$ by
as $P_i(N) \union P(N)$ results in another model $N'$ of $\CT^+$ with the identity map a homomorphism
$N \to N'$; so we must have $N=N'$, i.e. $P$ implies $P_i$ in $N$.     It follows that if $f: M \to M'$ is a homomorphism, then it is also a homomorphism $N \to (M')^+$; 
since $N$ is e.c., the Tarski-Vaught property holds for $N \to (M')^+$ and in particular for $M \to M'$.
 Thus $M$ is an e.c. model of $\CT$.      Now $N \to M^+$ is
a homomorphism, so as $N$ is e.c., $N=M^+$.    

This shows in particular that the  e.c. models of $\CT$ and of $\CT^+$ can be canonically identified, when a finite conjunction is eliminated.  In 
 case $\CT$ is p.p. bounded, the universal (and thus pp-saturated) e.c. model does not change; except that $\meet_i P_i$ is now also named by $P$.  

 Of course, even if each $P_i$ admits a complement, $P$ may not; thus naming a type here does not have the effect it does in \cite{morley-number}.

\ssec{Saturated models and bounded theories}  \label{saturated}

The category of e.c. models with embeddings admits amalgamation:   if $f_i: A \to B_i$, we may embed each $B_i$ in an ultrapower $A^*$ of $A$, then compose with a homomorphism to an e.c. model.  

Since homomorphisms need not be injective,    there may be an upper bound $\theta$ on the cardinality of 
 existentially closed models.   Call $\CT$ {\em ec-bounded} in this case.    This is indeed the case that concerns us;
 {\em assume from now on that $\CT$ is ec-bounded.}   
 
 An e.c. model $M$ is called {\em $\kappa$-saturated} if for any e.c. $A \leq M$ and any embedding $f: A \to B \models \CT$ with 
 $|B| < \kappa$,   there exists a homomorphism $g: B \to M$ with $g \circ f = Id_A$.  
The  usual existence theorem for $\kappa$-saturated models remains valid:   for any cardinal $\kappa \geq |L|$, there exists a $\kappa^+$-saturated e.c. model (of cardinality $\leq 2^\kappa$).  Thus   there exists a  $\kappa$-saturated e.c. model  $\U$ of $\CT$     of  cardinality $\leq \theta$, which is $\kappa$-saturated
for all $\kappa$.  In particular,   the irreducibility  assumption on $\CT$ implies that $\U$ is {\em universal} in the sense that any model $N$ of $\CT$  admits a homomorphism  into $\U$; if $N$ is
e.c., $N$ embeds into $\U$.    Note that  $\U$ is    homogeneous  for {pp} types.  

\begin{prop}\label{prop1}  Assume $\CT$ is ec-bounded.  Then it has a unique universal e.c. model $\U$ (up to isomorphism.)  
Any homomorphism on $\U$ is an embedding, and 
    any endomorphism $f: \U \to \U$ is an isomorphism.   If $\U \leq V \models \CT^{\pm}$ then 
    there exists a homomorphism $r: V \to \U$ with $r|\U = Id_\U$.  
\end{prop}
\prf  Existence of a saturated (in any cardinality)   ${\U}$ was seen above; it is in particular universal. 
We also noted that homomorphisms on $\U$ are embeddings.   Let $f: {\U} \to {\U}$ be an endomorphism,
with image $U'$.  Then $f \inv: U' \to {\U}$ is an embedding, that extends (by the $|{\U}|^+$- saturation of ${\U}$)
to an embedding $g: {\U} \to {\U}$.  Since $g$ is injective, while $g|U'$ is surjective, we must have ${\U}=U'$.
For the last statement, as ${\U}$ is universal there exists a homomorphism $f: V \to {\U}$;
 on ${\U}$ it induces an isomorphism $g$;  so $r=g \inv \circ f: V \to {\U}$ is as required.  
 
 It remains to show that any universal e.c. model $U$ is isomorphic to the saturated $\U$.  Since  $U$ is universal,
 there exists a homomorphism $f: U \to \U$, which must be an embedding; so we may assume $U \leq \U$.
 Then there exists a retraction $r: \U \to U$.  But endomorphisms of $\U$ are isomorphisms, so $\U \cong U$.
\eprf


 \begin{rem} We are dealing here with the analogue of {\em finite} structures in first-order logic, or {\em compact} ones in continuous logic.     This is the basic material  of algebraic closure in positive logic. 
 In \cite{wtaylor},  the universal e.c. structure  of a pp-bounded theory is studied  under the name of {\em minimum compactness}.
  
 We observe (though we will make no use of the fact) that   an ec-bounded  theory is 
 {\em equational}, in particular stable, in the following sense:  
 
(E) If $p(x,y),q(x,y)$ are ${\TS}$-contradictory {pp} partial types, there is a finite bound on the length of a sequence $(a_i,b_j)$
with $p(a_i,b_j)$ for $i<j$ and $q(a_i,b_i)$.

Otherwise,
in some $M \models {\TS}$ there will be a long chain of such elements $(a_i,b_i)$.  By homomorphically mapping into an e.c. model, we may assume $M$ is e.c.  For $i<j$ we have $p(a_i,b_j)$ and so not
$q(a_i,b_j)$, yet we do have $q(a_i,b_i)$; so $b_i \neq b_j$.  This contradicts the bound on the size of e.c. models.  
\end{rem}
  
 \label{1}

 \ssec{The {pp} topology on ${\U}$ and $\Aut({\U})$.}    
 \label{pptop}
 Let us topologize ${\U}$, taking as a pre-basis the complements of sets of the form 
 $\{x:  R(x,c_1,\ldots,c_k) \}$, with $R$   {pp} and $c_1,\ldots,c_k \in {\U}$.   
 Under this topology, {\em ${\U}$ is T1:}  if $a \neq b \in {\U}$, there exists a {pp} $R$  
such that   ${\U} \models R(a,b) \ \wedge \ (\forall x) \neg R(x,x)$. 
\footnote{Let $E$ be the congruence generated by $(a,b)$, e.g. if there are no function symbols $E$ is the equivalence
relation identifying $a,b$ only.  Let $\U' = \U/E$ so that $\U \to \U'$ is a homomorphism.  As $\U$ is   e.c.,
$\U'$ cannot be a model of $\TS$, so that $\TS \models \neg S(z,x,x)$ and $\U \models S(c,a,b)$ for some
conjunction $S$ of atomic formulas.  Let $R(x,y) = (\exists z)S(z,x,y)$.}
  Then $\neg R(x,b)$ is
 an open set including $b$, but not $a$; so $b \notin cl(a)$.  As this holds for all $b \neq a$
 we have $cl(a)=a$.
 
 Also, {\em $\U$ is compact}:  consider a family $F_i$ of basic closed sets with the finite intersection property.
 $F_i$ is defined  by $R_i(x,c_i)$ with $R_i$ {pp}.  In an elementary extension $U'$ of ${\U}$,
 one can find $d'$ with $\R(d',c_i)$ holding for all $i$.  By \propref{prop1} there exists $r: U' \to {\U}$,
 $r|{\U}=Id_{\U}$.  Let $d=r(d')$.  Then $R_i(d,c_i)$ holds for each $i$, so $d \in \meet_i F_i$ and
 $\meet_i F_i \neq \emptyset$.

 Let $G= \Aut({\U})$.   Since $\U$ is many-sorted, a function $f: \U \to \U$ 
 is actually a sequence of functions $f_S: S(\U) \to S(\U)$, indexed by the sorts $S$; we take the sorts to be closed under finite products.  Each $f: \U \to \U$ can be viewed as a certain function on the disjoint union of sorts,
 respecting the projection maps from products to factors.       We give $G$ the   topology
 of pointwise convergence, induced from the space of functions $\U \to \U$.  Thus if $a_1,\ldots,a_n,b_1,\ldots,b_m \in \U$ and $R$ is pp,
 then  \[\{g: \neg R (ga_1,\ldots,ga_n,b_1,\ldots,b_m) \}\]  is a pre-basic open set.     As $\U$ is T1, so is $G$.  Let us see that 
 {\em   $G$ is compact.}  Let $u$ be an ultrafilter on a set $I$, and let $g_i \in G$, $i \in I$;
  we need to find a limit point $g$ of $(g_i)_i$ along $u$.  Let $\U^*$ be the ultrapower of $\U$ along $u$,
  and let $g_*: \U \to \U^*$ be the ultraproduct of the maps $g_i: \U \to \U^*$; let $j$ be the diagonal embedding
  $\U \to \U^*$, ultrapower of $Id: \U \to \U$.  
      As $\U^* \models  \CT^{\pm}$,
    \propref{prop1} provides    a homomorphism $r: \U^* \to \U$ with $r \circ j = Id_{\U}$.  Let $g=r \circ g_* $.
 Then $g \in End(\U)=\Aut(\U)$.  If $R(g_ia_1,\ldots,g_ia_n,b_1,\ldots,b_m)$ holds for $u$-almost all $i \in I$,
 then $\U^* \models R(g_*a_1,\ldots,g_*a_n,jb_1,\ldots,jb_m)$ so $\U \models R(ga_1,\ldots,ga_n,b_1,\ldots,b_m)$.
Hence   $g$ is indeed a limit point of $(g_i)_i$ along $u$.

  {\em left and right translation are continuous.}  Indeed a pre-basic open
 set has the form $B=\{g: g(p) \in W\}$ where $W$ is a pre-basic open subset of ${\U}$, and $p \in {\U}$.
 For $a \in G$ we have $aB = \{h:  h(p) \in W' \}$ and $Ba = \{g:  g(p') \in W \}$ 
 where $W'=aW$ and $p'=a(p')$.   These are also pre-basic open.      Further, {\em inversion on $G$ is continuous.}
 Indeed a pre-basic closed subset of $\U$ has the form $\{z:  (z,q) \in R\}$ where $R$ is a basic relation, 
 and $q$ is a tuple from $\U$.  Thus a pre-basic closed subset of $G$ has the form 
 $W=\{g \in G:  (g(p),q) \in R\}$ where $p$ is a    tuple  of elements of $\U$.  So  $W \inv = \{g \in G:  (p,g(q)) \in R\}$, another pre-basic closed set of $G$, with the parameters and test points interchanged.   
  
  Let $\U_h$ denote the union of all $P(\U)$, with $P$ a  pp partial type that is Hausdorff
in the pp topology.  (Including imaginary sorts, defined below.)  We will see in the case of interest to us that the restriction $\Aut(\U) \to \Aut(\U_h)$ is surjective.   In any case,
with the topology described above, it is clear that 
 {\em $\Aut(\U_h)$ is Hausdorff.}
  
 At this level of generality, it follows from Ellis' joint continuity theorem \cite{ellis} (relying on a Baire category argument) that $\Aut(\U_h)$ is a compact Hausdorff topological group, acting continuously 
 on $\U_h$.    In our setting,  with $\U=\co(T)$
 the pattern space of a theory $T$, we can easily see this directly; 
  the compact-open and finite-open topologies coincide on $G$ by \remref{closedopen}.

\ssec{Logical complexity}    \lbl{complexity}  Assume $\CL$ is countable.  What is the logical complexity of 
the above construction; for instance of 
determining, given $\CT$, whether $\Aut(\U)=1$?  We have $\Aut(\U) \neq (1)$ iff  there exist conjugate but distinct elements in $\U$; this is iff
  there exists a maximal pp type $p(x,y)$ with (a) equal restrictions to $x,y$, and (b) $p(x,y)$
  guaranteeing distinctness of $x,y$.  Now (b)  holds  iff for some pp $\theta(x,y)$ in $p$,
  $\CT \models \neg (\exists x) \theta(x,x)$.   On the other hand, (a) holds iff for all 
  $\phi(x)$, and $\phi'$ orthogonal to $\phi$, $p(x,y)$ contains a formula   orthogonal to
  $\phi(x) \& \phi'(y)$.  
  This is (at worst) an analytic ($\Sigma^1_1$) condition on $\CT$. 
  
    Likewise for existence of a homomorphism into a fixed finite group or compact Lie group; also for $\Aut(\U_h)$.

It is also worth nothing that if $\Aut(\U)=1$, then $\U$ admits a Borel structure.  The natural map taking an element of the core to a 
pattern type is in this case 1-1,   and the image of $\U$, as well as of the relations $\R_t$, is Borel.

\ssec{Imaginary quotients} 
\label{quotients}

Assume $\U$ admits $\bigwedge,\exists$-elimination, in the sense that a conjunction of finitely many atomic formulas
is atomic, and a pp relation is also atomic.  This can be achieved by an appropriate Morleyzation, 
 see \secref{morleyzation}.

Let $E$ be a closed equivalence relation on some sort $\Sigma$ of $\U$; i.e. $E$ is an intersection of 
pp-definable subsets $E_n$ of $\Sigma^2$, and is an equivalence relation on $\Sigma$.  For simplicity, we will consider only the sort $\Sigma$; we will write $E$ and $E_n$ also for the diagonal relations on $\Sigma^k$,
i.e. $(x_1,\ldots,x_k)E(y_1,\ldots,y_k)$ iff $\bigwedge_{i=1}^k x_iE y_i$.  
We can add an additional sort  $\bSi=\Sigma/E$; with the natural map $\pi: \Sigma \to \bSi$.  
But we are  interested at the moment in $\bSi$ on its own right.   We let $\U'$ be the $\LS$-structure
with universe $\bSi$, and with $R(\U'):= \pi R(\U)$ for every $n$-ary atomic relation $R$.  
{\em Note that a sentence  $R=R' \meet R''$, true in $\U$, need not remain true in $\bSi$.
 Thus $\U'$ may not admit $\bigwedge$-elimination.}  Let $\TS' $ be the primitive universal   theory of $\U'$.
What are the axioms of $\TS'$?   For a single atomic relation $R$, the sentence $\neg \exists x R$ will
be in $\TS'$ if and only if it is in $\TS$.   But for a conjunction, say of two conjuncts $R,R'$, we have:
\[(*) \ \       \TS' \models \neg  \exists x (R(x) \wedge R'(x)) \hbox{
iff for some } n, \ \ \ 
 \TS \models  \neg  \exists x, x' (R(x) \wedge R'(x') \wedge E_n(x,x')) \]

\begin{lem}  $\U'$ is the universal e.c. model of $\TS'$.  Any endomorphism of $\U'$ lifts to
an automorphism of $\U$.  \end{lem} 

\prf  We first check that $\U$ is universal.  Let $A \models \TS'$, and let $(a_i)_{i \in I}$
enumerate the universe of $A$.   We introduce  variables $(x_{i}: i \in I)$.  Also for each
instance of an atomic $k$-place relation $R(a_{i_1},\ldots,a_{i_k}) $ valid in $A$, we introduce new variables $y_1,\ldots,y_k$ especially for this instance of $R$, and let 
\[\Gamma_R = \{R(y_1,\ldots,y_k) \wedge y_\nu E_n x_{i_\nu} : \nu=1,\ldots,k\}.\]
This collection of formulas can be realized in $\U$, using the saturation of $\U$
and the description (*) of $\TS'$ above.   
Such a realization defines a map $f: A \to \U$, mapping $a_i$ to the realization of $x_i$,
such that the composition $\pi \circ f : A \to \U'$ is a homomorphism.  This proves universality of $\U'$. 

Next let $ f: \U' \to \U'$ be an endomorphism.  Let $a=(a_i)_{i \in I}$ enumerate the universe of
$\U$.  Choose $c_i \in \U$ with $\pi(c_i) = f(\pi(a_i))$.  Let $\G(x)$ be the atomic type of $a$
in appropriate variables $x=(x_i)_{i \in I}$.  
  We seek $b$ realizing  
  \[ \Gamma'(y) = \Gamma(y) \union \{ y_i Ec_i: i \in I \} \]  
  By saturation of $\U$, it suffices to prove consistency; so consider finitely many formulas of $\G'$; for instance
$R_1(y_1,y_2) \wedge R_2(y_1,y_2) \wedge y_1 E_n c_1 \wedge y_2 E_n c_2$.  
By the $\bigwedge$-elimination assumption about $\CL$, there exists an atomic $R $ with  
\[ \U \models  R \iff (R_1 \wedge R_2) .\]
Thus we reduce to the case of a single $R$:  we have to solve $R(y_1,y_2) \ \wedge y_1 E_n c_1 \wedge y_2 E_n c_2$.  Existence of such $y_1,y_2$ follows  from (*) and the fact that $\U' \models R(\pi(a_1),\pi(a_2))$ and hence, $f$ being a homomorphism, $\U' \models R(\pi(c_1),\pi(c_2))$.  
Thus there exists $b$ with $\U \models \Gamma'(b)$.  Define $F(a_i)=b_i$; then $F: \U \to \U$
is an endomorphism (and hence an automorphism) of $\U$, and $\pi \circ F = f \circ \pi$.

To see that $\U'$ is e.c., let $f: \U' \to A$ be a homomorphism into a model of $\TS'$.
Compose $f$ with some homomorphism $A \to \U'$.  To show that $f$ is an embedding,
it suffices to show the same of the composition; so we may assume $f: \U' \to \U'$.
In this case we saw that there exists $g: \U \to \U$ inducing $f$.  As $g$ is an automorphism, so must be $f$.   

\eprf
 \ssec{Type spaces}

   For a model $M$ of $T^{\pm}$, and a finite set of variables $y$ of $L$, we let $M^y$ be the set $y$-tuples of elements of $M$, i.e. the set of  functions from $y$ to $M$ preserving sorts.
Given  a set $\gamma$ of formulas of $L$  along with a distinguished   set $x$ of variables,   
  a $\gamma$-type  $p$ over $M$ is a set of formulas of the form
  \[p=tp_{\gamma}(a/M) = \{\phi(x,b): \phi(x,y) \in \gamma, b \in M^y,  N \models \phi(a,i(b)) \}\]
  where $i: M \to N$ is a homomorphism, $a \in N^x$, $N \geq M$.  
  The set $S_\g(M)$ of $\gamma$ types over $M$ has a natural compact topology, with basic open sets of the form
  $\{p: \phi(x,b) \in p\}$.    The subspace of {\em maximal} types is Hausdorff.   When $\gamma$ includes all formulas with distinguished variables $x$, we write $S_x(M)$.    We will assume that $\gamma$ is closed under negations (at least in the sense that any $\phi(x,a)$ with $\phi \in \g$ is equivalent to the negation of some $\phi'(x,b)$ with $\phi' \in \g$.)

%

 Type spaces will be treated, notationally, as simplicial 
spaces (\cite{morley-types}), meaning that we can write $S(M)$ for the data associating  to any   $\gamma$ the space $S_\g(M)$. 
For infinite  sets of formulas $\Gamma$, $S_\Gamma$ can be defined in the same way, or equivalently as the inverse limit of $S_\g$ over all finite $\g \subset \Gamma$.

  \begin{rem}[Interpolation] \label{interpolation} Let $\CL \subset \CL'$ be relational languages, $\CT'$ a primitive universal theory,
 $\CT = \CT' | \CL$.  
Let $N'$ be an e.c. model for $\CT'$, $N:= N'_L $ the reduct of $N'$ to $L$.  Assume
     interpolation holds in this form:   if 
 $R' \in \CL'$, $R \in \CL$ are pp formulas, $\CT' \models \neg (R \wedge R')$,  and $N' \models R'(a)$, 
  then for some pp $S \in \CL$ we have $N' \models S(a)$ and $\CT \models \neg (R \wedge S)$. 
 Then $N$ is an e.c. model of $\CT$.    If $N'$ is universal e.c., so is $N$. 
 
 In particular, this is the case if $\CL'$ is obtained from $\CL$ by adding constant symbols.

 \end{rem}
        
 \prf  Let $M \models \CT$, and let $f: N \to M$ be a homomorphism.  Expand $M$ to an $\CL'$-structure $M'$ by
  interpreting any basic relation symbol $R'$ of $\CL$ as $(R')(M') = f(R'(N))$; thus $f:N' \to M'$ is a $\CL'$- homomorphism.
 It is easy to see from the assumed  interpolation property   that $N' \models \CT'$.  Given this, the 
 e.c. property of $N'$ with respect to $f$ includes the same for $N$.    Hence $N$ is e.c.  If $N'$ is universal e.c.,
 let $M \models \CT$; then the diagram of $M$ is consistent with $\CT$ and hence with $\CT'$; so 
  there exists a $\CL$-homomorphism $g: M \to M'$ into a model of $\CT'$.    By universality of $N'$, there exists
  $h: M' \to N'$, and hence by composing we have a homomorphism $M \to N$.
   \eprf


\end{section}

\def\CU{\mathcal{J}}

\begin{section}{A relational structure on type spaces}
\label{section-rel}

Let $T$ be a   universal theory.    We assume that any two   models of $T$ can be embedded into a single model
(joint embedding property).   
   We allow $T$ to be many-sorted, and sometimes refer to a product of sorts, or a {}definable subset, as itself a sort.\footnote{Formally, these are indeed imaginary sorts.}  We take a fixed countable set of   variables for each sort.  
$|L|$ is the number of formulas of $L$.    {\em Unless otherwise stated, we consider only quantifier-free
formula in this section.}  \footnote{ Let us say that  $T$ is {\em QEble} if there exists a complete theory $T_1$   with quantifier elimination, whose universal part is $T$.   If we {\em begin} with a complete first-order theory $T'$, we first Morley-ize to obtain a thery $T_1$ with QE , then let $T=(T_1)_{\forall}$ be the  
the universal part, and apply the theory below to $T$ in order to obtain results about $T'$.}
Let 
\[T^{\pm} = T_{\forall} \union \{\neg \phi: \ \phi \hbox{ universal }, \phi \notin T_{\forall} \}\]
  
    We aim to associate with $T$ a  {\em language $\LS$ } (the {\em pattern language}), a canonical irreducible primitive  universal theory $\TS$ of $\LS$, and a canonical model $\CU :=:\co(T)$ of $\TS$, the {\em core} of $T$.  
    
     and an enrichment
of the type spaces of models of $T$ to models of this theory.

The language ${\LS}$ has the same sorts as the type spaces of $T$, i.e. a sort   for each set of formulas
$\gamma$ along with  a set of distinguished variables $x$.  For an $\LS$-structure $A$, this sort will be denoted by $S_\gamma$.
\footnote{We take only {\em finite} sets of formulas $\g$ for the official sorts. 
Still for infinite $\Gamma$, we can define $S_{\Gamma}$ as the projective limit of $S_{\gamma}$ over all finite $\gamma \subset \Gamma$.   
This will be compatible with definitions below .
In particular   a homomorphism defined on the official sorts extends uniquely to the derived infinite ones.  }



Let $x_i$ be an $n$- tuple of variables, for $i=1,\ldots, n$; they will be referred to as the distinguished variables.  Let $y$ be an additional tuple of variables (the {\em parameter variables}.)  
  Let $t=(\phi_1,\ldots,\phi_n;\alpha)$ be
an $n$-tuple of formulas $\phi_i(x_i,y)$ of $\gamma$, and let  $\alpha(y)$ be a formula.

To each such $t,\alpha$ we associate a   relation symbol $\R_t$ of ${\LS}$, taking variables $(x_1,\ldots,x_n)$.
 
For any $M \models {T^{\pm}}$, we define an ${\LS}$-structure whose  sorts are $S_\g(M)$ for the various sorts $\g$.   
When $t=(\phi_1,\ldots,\phi_n;\alpha)$ and $\phi_i \in \gamma_i$ we define $\R_t$ 
on $S=S_{\g_1} \times \cdots \times S_{\g_n}$
thus:  

 \[\R_{t;\alpha} ^S =\{(p_1,\ldots,p_n) \in S:   \neg (\exists a \in \alpha(M)) \bigwedge_{i \leq n} (\phi_i(x_i,a) \in p_i) \} \]


We omit $\alpha$ from the notation in case $\alpha$ is universally true, i.e.  $\alpha(M)=M^y$.    If $\gamma_i$ is closed under conjunctions with the  formula $\alpha(y)$, then $\R_{t;\alpha} \equiv \R_{t'}$ where
$t'= (\phi_1(x,y)  { \wedge } \alpha(y),\cdots, \phi_k(x,y)  { \wedge } \alpha(y))$.

\begin{example}\label{emptyset}  If $\phi(x)$ has no parameter variables, then 
$\phi \notin p$ iff $S \models \R_{\phi}(p)$. 
Thus the atomic type of $p$ in $S_x(M)$ determines the restriction of $p$ to $S_\g(\emptyset)$.
\end{example}

\begin{example} \label{def0} 
  $\R_{\phi(x,y) \iff   \theta(y)}$ captures the set of types $p(x)$, admitting $\neg \theta(y)$
as a $\phi$-definition. 
\end{example}

 \begin{example} 
 Let $\phi(x,y) \in \g$  .   In $S_\g(M)$, the relation
  \[E_\phi: \equiv \R_{(\phi  , \neg\phi ) }  { \wedge } \R_{(\neg \phi , \phi)}\] holds of a pair $p,p'$ iff they restrict to the same $\phi$-type over $M$. 
   In any $S_\g(M)$ and also in any e.c. model, $E_\phi$ is an equivalence relation,
  and the intersection of all $E_\phi$ is the diagonal.   Similarly, for a  finite set of formulas $\gamma$,   equality
  is definable by a {pp} formula, as is more generally the restriction map $S_\gamma \to S_{\gamma'}$ for $\gamma' \subset \gamma$.
 \end{example}
  
 \begin{example}  \lbl{completeba}  Assume $\phi(x,y)$ is {\em free} in the sense that for any distinct $b_1,\ldots,b_n \in M^y$,
 there exists $a$ such that $\phi(a,b_i)$ iff $i \leq  n$ is odd.   (A strong negation of NIP.)   Then $S_\phi(M)$ carries a Boolean algebra structure:
 for any $p,q \in S_\phi(M)$ there exists a unique $r \in S_\phi(M)$ with $\phi(x,b) \in r$ iff $\phi(x,b) \in p \& \phi(x,b) \in q$,
 and likewise for the other Boolean connectives; they are all described by basic $\LS$-formulas; these formulas
 will define a Boolean algebra structure on any e.c. model of the universal primitive theory of $S_\phi(M)$.  Any {\em compact} model for the pp topology (such as $\co(T)$ below) will in fact be a complete Boolean algebra.  
 \end{example}

 \begin{lem} \label{3.5} \begin{enumerate}
 \item   If $M, N \models T$, $M \leq N$, the restriction map  $r_{N,M}: S(N)\to S(M)$ is  an $\LS$- homomorphism.  \item  Let $u$ be an ultrafilter, $M^u$ the ultrapower of $M$.  There exists a canonical ultrapower
 map $j_u: S(M) \to S(M^u)$; it is an  $\LS$-embedding.
 \item  If $M, N \models T^{\pm}$, then in (1),  $r_{N,M}$   admits a section, i.e. 
  a homomorphism $j: S(M) \to S(N)$ with $r \circ j = Id_{S(M)}$.
  
   \end{enumerate}  \end{lem}
\prf  (1) is clear.

So is (2):    if $M \prec M', b \in M'$ and $p=tp(b/M)$, let $j_u(p) = tp(b/M^u)$
where $b$ is identified diagonally with its image in the ultrapower $(M')^u$.  Note that the
 relations $\R_t(tp(b_1/M),\cdots,tp(b_k/M))$ are first-order definable in the pair $(M',M)$, hence persist.
 
 For (3)  let $i: N \to M^u$ be an   embedding over $M$,   let $i_*$ be the pullback by $i$  of types over $M^u$ to types over $N$, and
let $j= i_* \circ j_u$.  Then $r \circ j = d_* \circ j_u = Id_{S(M)}$.  \eprf

 \begin{cor}\label{onetheory} The primitive universal theory of $S(M)$  does not depend on the choice of model  $M \models T^{\pm}$.    \end{cor}
 \prf  Let $M, N \models T^{\pm}$.  Then $N$ embeds into an ultrapower $M^*$ of $M$.
 By \lemref{3.5} (1,2) we have $Th_{p\forall}(N) \subset Th_{p\forall}(M^*) = Th_{p\forall}(M)$.
 \eprf 
 
  One can also see this directly - if $\CT \models \neg (\exists x) \bigwedge_{i=1}^m \R_t(x)$ then (as this is due to a finite inconsistency) the same is true in $S(M)$.  If restricting to a set of formulas $\gamma$, it suffices to have  $Th_{\forall}(M)=T_\forall$ in the parameter sorts of  $\g$.
  
 \begin{defn} \label{put}  The {\em theory of $T$-patterns} is the  common  primitive universal theory of all type spaces $S(M)$.  It will be denoted by ${\TS}$. \end{defn}
   
     It is easy to write down the axioms of ${\TS}$ explicitly.    For instance, 
   $(\forall \xi) \neg \R_\phi(\xi)$ 
 will be an axiom of ${\TS}$ iff for some $\theta(u_1,\ldots,u_n)$, 
 
 \[T^{\pm} \models (\exists u_1,\ldots,u_l)\theta \&  (\forall x)(\forall u_1,\ldots,u_n)(\theta \implies  \  \bigvee_j  \   \phi(x,u_j))\]
 In other words, the definable partial type $\{\phi(x,u): u \in \alpha(M)\}$ is inconsistent, for 
 any model $M \models T$.  
 

\begin{remark} ${\TS}=\TS(T)$ varies continuously with $T^{\pm}$, in the sense that if $\si \in \TS(T)$ then there
exists a finite $T_0 \subset T^{\pm}$ such that for all $T'$ with $T_0 \subset (T')^{\pm}$ we have $\si \in \TS(T')$.  
\end{remark}

   \begin{lem} \label{compact} Let $A \models T^{\pm}$.  
  Then any 
  model of  ${\TS}$
 admits a homomorphism into $S(A)$.   In particular if $E$ is an e.c. model of $\TS$, then $E$ admits an embedding into $S(A)$. \end{lem}
  \prf     Let $S=S(A)$, made into an $\LS$-structure by the natural interpretation of $\R_t$. 
    Consider the space
      of  sort-preserving  functions $E \to S $, with pointwise convergence topology, relative to the topology of $S$.
        Since $S$ is (on each sort) compact, the space of functions
 is compact.  The subspace of functions preserving finitely many given instances of the   relations $\R_t$ is closed, and non-empty since any {pp} sentence  true in $E$ is is true in $S$.   Hence  a map exists preserving all instances of all relations $\R_t$.
     This is a homomorphism, and in case $E$ is e.c. it must  be an embedding.
\eprf

  
The argument of \lemref{compact} was given  in  \cite{mycielski} for general compact topological algebras,   generalizing earlier results in the theory of modules.    
   
  It follows from \lemref{compact} that ${\TS}$ is ec-bounded.  
Thus by the results of \secref{1}, a unique universal e.c. model of ${\TS}$ exists.  

\begin{defn}  The core of $T$, $\co(T)$ is  the universal e.c. model of ${\TS}$.  \end{defn}

When $T$ is fixed, $\co(T)$ will  be denoted by ${\CU}$.   

 We view $ \CU$ as an 
$\LS$-structure; it is thus endowed also with the pp topology.   Likewise we give $G=\Aut(\CU)$
the topology described in \secref{pptop}.   Thus $\CU$ and $G$ are compact T1 spaces. 

  Let $\fg=\fg_{\CU}$ be the normal subgroup of
$G=\Aut(\CU)$ described in \secref{topg}.   Let $E_{\fg}$ be the equivalence relation on (each sort of)
$\CU$ given by $\fg$-conjugacy.   $E_{\fg}$ appears in general to be a complicated 
equivalence relation on $\CU$; the  visible complexity upper bound, when $|L|$ is countable,  is:  no worse than analytic.  But on each   atomic type we will see that it is closed.

The following proposition can be regarded as a  form of quantifier elimination.  It implies, in particular, that it is possible to compute the core separately for each sort.    Let us call a set $\g(x;y)$ of formulas  {\em full}  if it includes
all formulas $\theta(y)$ in  (any) parameter variables alone.

\begin{remark}  At the level of generality we are working with,  of irreducible universal theories,   the fulness assumption can easily be relativized.  Suppose $\g$ is a finite set of quantifier-free  formulas.  By a slight Morleyzation we can take them to be atomic; let $L_\g$ be the sublanguage of $L$ generated by $\g$.
Let $T_\g$ be the given  universal theory $T$, restricted to $L_\g$.  Then $T_\g$ is itself an  irreducible universal theory, and fulness of $\gamma$ is now tautological.\end{remark}

\begin{prop}  \label{qe}  Let $T$ be an irreducible  universal theory,
and consider the sort $\g$ of  $\LS,\TS$ for any full $\g$.   Then: \begin{enumerate}
\item For any pp formula $A(\mu)$ there exist atomic formulas $\Xi_k(\mu)$
 such that in  in ${\CU_\gamma}$, as well as in $S_\g(M)$ for any e.c. model $M$ of $T$, 
\[ A  \iff \bigwedge_k \Xi_k \]
Here $k$ ranges over an index set $K$ of cardinality at most $|L|$.
\item  If $T$ is the universal part of a complete first-order theory with QE, $K$ can be taken countable, and the fulness assumption on $\gamma$ can be restricted
to any given (finite) family of parameters sorts.  
\item 
  $\CU$ is homogeneous for atomic types.  
  \item  An atomic type of $\TS$ is the type of an element of  $\CU$ if and only if
is maximal.  
\item   $\CT$   admits   elimination of finite conjunctions, at least if models of $T^{\pm}$ have more than one element.
\item If $\g$ is the Boolean closure of a finite set with a set formulas in the parameter variables alone,
then   equality  can be defined in terms of the basic symbols $\R_t$.  
\end{enumerate}
\end{prop}

\begin{proof}

 Let us  first consider the the case of $S_\gamma(M)$ where $\gamma$ is a finite set of formulas, along with 
 formulas in the parameter variables alone.   
  For simplicity (and using a standard trick, without loss of generality) assume $\gamma(x,u)$ is a single formula.

  We can write $A(\mu) $ in `normal form' as 
\[ A(\mu) \equiv (\exists \xi)(\R_\phi(\xi)  { \wedge } \pi(\xi) = \mu),\]
 where $\pi$ is the coordinate projection $S_{xy} \to S_x$,
 and $\R_\phi$ asserts that $\phi(xy,u)$ is not represented in $\xi$.  (See (5).)  Now given $p(x)$,
 there exists $q  \in S_{\g} \union {\phi}(M)$ extending $p(x)$ and omitting $\phi(xy,u)$ {\em unless} 
  $p(x) \union \{\neg \phi(xy,b): b \}$ is inconsistent with $T$ and the quantifier-free diagram of $M$;
   i.e. for some $c=(c_1,\ldots,c_{m})$,  $b=(b_1,\ldots,b_l)$,$e$ from $M$ and $\theta \in L$, 
 we have $\g(x,c_i) \in p$ for each $i$, $M \models \theta(c,b,e)$ and
\[T \models (\forall x,y,v,u,w) \neg ( \bigwedge_{i=1}^{m} \g(x,v_i) \wedge \bigwedge_{j=1}^l \neg \phi(xy,u_j) 
\wedge \theta(u,v,w) )\]


 Then 
 $S (M) \models   (\exists \xi)(\R_t(\xi)  { \wedge } \pi(\xi) = p)$ iff for each such $\theta$, 
 \[ \g(x,v_1)  { \wedge } \cdots  { \wedge } \g(x,v_m)  { \wedge } \theta(u,v,w) \]
 is omitted in $p$.   Let $\Xi_{m,\theta} = \R_{\gamma,\cdots,\gamma; \theta}$.    
  
 Then we have shown that
 \[ S (M) \models A(\mu) \iff \bigwedge_{m} \Xi_{m}(\mu) \]
 
 In case $T$ is the universal part of a complete first order theory with QE, $\theta$ may be taken to be a quantifier-free
 formula equivalent to $  (\forall x)(\forall y) \neg  ( \bigwedge_{i=1}^{m} \g(x,v_i) \wedge \bigwedge_{j=1}^l \neg \phi(xy,u_j) )$; thus the $\bigwedge$ above can be taken to range over a countable set, regardless of the cardinality of the language.
 
This concludes (1,2) in the case of finite $\g$, for $S_\g(M)$.


 The case of  an arbitrary $\gamma$ follows, since we will have  $(\exists \xi)(\R_t(\xi)  { \wedge } \pi(\xi) = \mu)$
 iff for every finite $\gamma' \leq \gamma$, letting $\mu_{\gamma'} $ be the restriction of $\mu$ to
 $\gamma'$, $(\exists \xi)(\R_t(\xi)  { \wedge } \pi(\xi) = \mu_{\gamma'})$.
 
 Using the compactness of $S(M)$, \lemref{u} (2) shows that the infinitary equivalence above is also valid in $\CU$.  

It was already noted in \secref{saturated} (just above \propref{prop1}) that $\CU$ is homogeneous for pp types.  Thus (1) implies (3).

Next we show maximality of the atomic types of elements of $\CU$.   
Let $P$ be the atomic type   of $a$ in $\CU$.  Consider
any pp formula $\psi$ not true of $a$ in $\CU$.  Then (since $\CU$ is e.c.) some pp formula $\phi$ is true of $a$ and contradicts
$\psi$ in models of $\CT$.  By the above, $\phi$ is equivalent to $\bigwedge \Xi_k$, with $\Xi_k$ atomic.   Each $\Xi_k$
must be in $P$; and some finite conjunction $\Xi$ of the $\Xi_k$ must contradict $\psi$ (otherwise
realize $\psi  { \wedge } \bigwedge \Xi_k$ in some elementary extension, and retract to $\CU$.)  So $\neg \psi$
follows from $P$.  Hence $P$ is maximal; no $\psi$ can be properly added to it.      The converse, that a maximal atomic type is represented in $j$, is clear since it is realized by some tuple $a$ in some $A \models \CT$ and
there exists a homomorphism $A \to \CU$, which must  by maximality be an embedding on $a$.    Thus (4).

(5) E.g.  $(p,p')$
 omits $\psi(x,u) \& \psi'(x',u)$ {\em and} omits $\phi(x,u) \& \phi'(x',u)$ iff
 $(p,p')$ omits $\theta(x,u,v,v') \& \theta'(x,u,v,v')$, where $v,v'$ are additional variables, 
 and $\theta$ agrees with $ \phi$ if $v=v'$, with $\psi$ if $v \neq v'$; and similarly $\theta'$.  
 
 (6)  If $\g$ is generated by the single formula $\gamma$ along with parameter formulas, as we may assume,
 then $p=q$ in any type space, and hence in the core, if and only if $\gamma(x,a) \in p \& \neg \gamma(x,a)$
 is omitted, and dualy.
\end{proof}

\ssec{Duality}

Let $T$ be a universal theory, and $M$ a universal domain, i.e. a highly qf- saturated and qf-homogenous model of $T^{\pm}$.    Existence of $M$ is equivalent to $T$ being Robinson, i.e. $Mod(T)$ admitting amalgamation
under embeddings.
Types over $M$ will be referred to as global types, and  types means: qf types.

\begin{prop} \lbl{duality1} Let  $A \leq M$ and let $B \models T$.       Let $b$ enumerate $B$.
There is a canonical 1-1 correspondence between:
\begin{itemize}
\item $\LS$-homomorphisms $h: S(A) \to S(B)$.
\item  Extensions of $tp(b/\emptyset)$ to a global type,   finitely satisfiable in $A$.
\end{itemize}
This is also valid locally for $\g$-types,  with $\g(x,y)$ closed under Boolean combinations, and $\LS$ restricted
to the formulas $\R_t$ with $t=(\phi_1,\ldots,\phi_n)$, $\phi_i \in \g$.  
\end{prop}

\prf  Let $h: S(A) \to S(B)$ be an $\LS$-homomorphism.  Define a global type $p(y)$:
\[ \phi(a,y) \in p(y) \iff  \phi(x,b) \in h(tp(a/A)) \]
If $\phi_i(a_i,y) \in p(y)$ for $i=1,\ldots,n$, let $q_i = tp(a_i/A)$, and let $t=(\phi_1,\ldots,\phi_n)$;
suppose $\bigwedge_i \phi_i(a_i,y)$ is not satisfiable in $A$; then $S(A) \models \R_t(q_1,\ldots,q_n)$;
so $S(B) \models \R_t(hq_1,\ldots,hq_n)$; but $\phi_i(x,b) \in hq_i$ for each $i$, a contradiction.

Conversely, let $p(y)$ be an extension of $tp(b/\emptyset)$ to a global type,   finitely satisfiable in $A$.  
In particular, when $a,a'$ realize the same type $q$ over $A$, $\phi(a,y) \& \neg \phi(a',y)$ cannot be in $p$;
so we can define $(d_q x)\phi(x,y) \in p$ to hold iff $\phi(a,y) \in p$ for some/all $a \models q$.   
Define $h:S(A) \to S(B)$ by:
\[ h(q) = \{\phi(x,b): (d_q x)\phi(x,y)  \in p(y) \}  \]
If $S(A) \models \R_t(q_1,\ldots,q_n;\alpha)$, $q_i=tp(a_i/A)$, $t= (\phi_1(x_1,y),\cdots,\phi_n(x_n,y))$, then 
there is no $b \in \alpha(A)$ with $\phi_i(a_i,b)$.   As $p$ is  finitely satisfiable in $A$, it is not the case that
each $\phi(a_i,y) $ is in $p$, where $y$ is a variable corresponding to a finite tuple $b_1$ of coordinates of $b$,
with $\alpha(b_1)$.     Thus $S(B) \models \R_t(h(q_1),\ldots,h(q_n))$.

\eprf

 
\begin{rem}\lbl{dual2}
 Composition of homomorphisms  corresponds by duality to   an operation on invariant  types, related to tensor product.  Consider a $b$-invariant type $p_b$,
 and an $a$-invariant type $q_a$.   We define a third $a$-invariant type $r_a$.  Namely let $a \subset E$; to define $r_a |E$, let $b \models q_a |E$,
 and $c \models p_b | E \union \{b\}$; let $r_a |E = tp(c/E)$.     When $p_b$ is finitely satisfiable in $b$ and $q_a$ in $a$,     it is easy to see that $r_a$ is
 also finitely satisfiable in $a$.    
 
 In terms of this product, one can characterize minimal retractions $S(M) \to S(M)$, and so carry out the whole theory on the dual level.
 
  \end{rem}

\ssec{Expansion with definable patterns}

 We repeat the statement of \propref{expand} in more detail.

 Let ${T}$ be an irreducible  universal theory, $V$ a distinguished sort
  and  $\g$   a set of formulas on $V \times P$, for various parameter sorts $P$, closed under negations.   We view
  products of parameters sorts as parameter sorts themselves.
  
  We consider irreducible universal theories $T'$ expanding $T$ by new relations on the parameters sorts.
By an {\em interpretation  of $T'$ in $T''$ over $T$}, we mean here a map $\alpha$ of quantifier-free formulas of $T'$
on parameter sorts, into quantifier-free formulas  in the same variables for $T''$, compatible with change of variables and finite Boolean combinations, and such that $T' \models (\forall u) \psi$ iff $T'' \models (\forall u) \alpha(\psi)$; and $\alpha(\phi)=\phi$
for any quantifier-free formulas $\psi$ of $L'$ and $\phi$ of $L$.   The notion of composition of interpretations over $T$ is clear;
we thus have a category $\mathcal{C}_T$.
In this setting, a bi-interpretation is simply an interpretation with a 2-sided inverse.

  \begin{prop} \label{expandd}     (=\propref{expand}.)
  There exists a unique minimal expansion $T\td_\g$ of ${T}$ 
that   has definable patterns at $\g$.  \end{prop}
  
More precisely,  let $\mathcal{C}\td_T$ be the full subcategory  of  of $\mathcal{C}_T$ consisting of those $T'$ that have definable patterns at $\g$.  Then $ \mathcal{C}\td_T$ has an object $T\td$ that 
maps into any other; and $T\td$ is unique up to bi-interpretation.    The
 bi-interpretation is
unique up to composition with a self-interpretation of $T\td$ over $T$.   
The self-interpretations of $T\td$ over $T$ form a group, isomorphic to $\Aut(\co(T))$.
 
Any model of $T^{\pm}$ expands to a model of $T\td$.

\prf    
We may assume $\gamma$ includes all qf formulas on $P$ alone, and is closed under Boolean combinations. 
  The language $\widetilde{L}$ consists of the language of $T$, along with new   relations $Q_{\phi;a} \subset P$ for each $a \in  J:=\co_{\g}(T)$   and $\phi \in \gamma$.   The theory $T\td$
is read tautologically off the pattern types, so that
if $J \models  \R_{\phi_1,\ldots,\phi_n; \alpha}(a_1,\ldots,a_n)$ then  $T\td$ includes
\[ (\forall x_1,\cdots,x_n) \neg (\alpha(x_1,\ldots,x_n) \wedge \bigwedge Q_{\phi_i;a_i}(x_i)) \] 
as well as
\[  Q_{\neg \phi ;a } \vee Q_{\phi; a} \]
To expand a model $M$ of $T$ to a model of $\TS$ thus amounts to specifying an $\LS$-
homomorphism $h:J  \to S_\gamma(M)$, and interpreting $Q_{\phi,a}$ as the $\phi$-definition of $h(a)$.
\begin{equation} \label{exp}
 Mod(T\td) = \{(M,h): M \in Mod(T), h: J=\co_\g(T) \to S_\gamma(M) \} \end{equation}
 
 \claim{}    $T\td$ is irreducible. 
\prf  Note that up to $T\td$-equivalence, the $Q_{\phi;a}$ are closed under Boolean combinations (for instance
 $Q_{\neg \phi;a} = \neg Q_{\phi;a}$); and include the   qf formulas on $P$ alone.   Thus to see that $T\td$ is irreducible, we have to show that 
 if $x,y$ are disjoint tuples of variables and 
 \[ T \td \models (\forall x)Q_{\phi;a}(x)  \vee (\forall y)Q_{\phi';a'}(y) \ \ \ \ (*) \]
 $T \td \models (\forall x)Q_{\phi;a}(x)$ or $T \td \models (\forall y)Q_{\phi';a'}(y) $.  
 
 Indeed assume  (*).    Let $M \models T$ be existentially closed.  Let $j: J \to S_\gamma(M)$ be an embedding.  Note that
 $j(a)$ and $j(a')$ are $\g$-types over $M$, finitely satisfiable in $M$ by the existential closedness of $M$.  Let $M'$ be an $|M|^+$-saturated
 elementary extension of $M$.   Then $j(a)$ and $j(a')$ are realized in $M'$.  Let $c \models j(a)$ and let $c' \models j(a')$.  
 Then $\neg \phi(c,m) \& \neg \phi'(c',m')$ is not represented by $m,m'$ from $M$; so one of them is not represented, say
 the former; 
 reading back this implies that $S_\g(M) \models \R_{\neg \phi}(j(a))$; as $j$ is an embedding, $J \models \R_{\neg \phi}(a)$,
 so $T \td \models (\forall x) \neg Q_{\neg \phi;a}(x)$.  Hence $T g \models (\forall x)   Q_{ \phi;a}(x)$, as required.
 \eprf

Let $T'$ be any expansion of $T$, having definable patterns at $\gamma$;  we will compare $T\td$ with $T'$.   
Let $\gamma'$ consist of $\gamma$ along with all $T'$-qf-$0$-definable subsets of the parameter sorts
  (close under Boolean combinations.)  
Let $J=\co_{\gamma}(T)$, $J' = \co_{\gamma'}(T')$, $\LS,\LS'$ their languages.  Let $M' \models (T')^{\pm}$, $M = M'|L$ the restriction to the language of $T$,
$S'=S_{\gamma'}(M')$, $S=S_\gamma(M)$.  We have a natural restriction map $r: S' \to S$.
Then $r$ is an $\LS$-homomorphism, and any section $s: S \to S'$ (i.e. map satisfying $r \circ s = Id_S$, in this case unique) is also an $\LS$-homomorphism.  Let $j: J \to S$ be an $\LS$-homomorphism,
and $j': J' \to S'$ an $\LS'$-homomorphism; use $j'$ to identify $J'$ with an $\LS'$-substructure of $S'$.  
We also have an $\LS'$-retraction $\rho: S' \to J'$. Then $\rho \circ s \circ j: J \to J'$ is an 
$\LS$-homomorphism.   By \eqref{exp}, this corresponds to an enrichment of $M$ to a model of
 $T\td$.  But by total definability, for  $q\in J'$ and $\phi \in \gamma$, the $q$-definition
 of $\phi$ is qf definable in $T'$.  This gives a map of $\widetilde{L}$ to $L'$ over $L$,
 mapping $Q_{\phi;a}$ to the  $\rho \circ s \circ j (a)$-definition of $\phi$.  It is clear that the pullback
 of $T'$ is precisely $T\td$.   Thus we have interpreted $T\td$ in $T'$, fixing $T$.

We now  check that $T\td$ has definable patterns at  $\gamma$.   
 Let $T'=T\td$, and let notation   ($\gamma', J',\LS',M',S',r: S' \to S, s: S \to S', j':J' \to S'$,
 $\rho: S' \to J'$).   In particular $M' \models T\td$, and the $\LS'$-structure on it is given by 
 a homomorphism $j: J \to M=M'|_L$.  Thus the $p$-definition of any $\phi \in j(J)$ is definable in $\LS'$.
 Now as we saw above, $\rho \circ s  : j(J) \to J'$ is an
$\LS$-isomorphism.  Thus if $q=\rho(s(p))$ then the   $q$-definition of $\phi$ equals to $p$-definition of $\phi$.
Now  $\rho \circ s(J)  = J'$ (since $r \circ \rho \circ s$ is an isomorphism on $j(J)$ into $r \circ \rho \circ s \circ j(J)$,
and $r$ is 1-1.)  Thus every element of $J'$ is $\LS'$-definable, as asserted.


Uniqueness is proved as in the first paragraph: let $T''$ be another universal theory expanding $T$, with  
definable patterns at $\gamma$, and minimal.   We have  found an interpretation $f$ of $T\td$ in $T''$ over $T$.
By the assumed minimality of $T''$, we also have an interpretation $g$ of $T'$ in $T\td$ over $T$.  
The composition $g \circ f$yields a self-interpretations of $T\td$.  That corresponds
to endomorphisms of $J$; but we know that endomorphisms of $J$ are automorphisms; equivalently
self-interpretations of $T\td$ over $T$.  We may assume, by twisting with such a self-interpretation, that $g \circ f=Id$.
On the other hand, the interpretation of $T''$
in $T\td$ must be 1-1 (if two qf formulas are intepreted by the same relation of $T\td$, they are equal.)
Hence from $g \circ f \circ g= g$ we obtain $f \circ g=Id$ also.   

Thus the two self-interpretations amount to a renaming of the new predicates indexed by $J$ (by an automorphism 
of $J$), showing that $T\td,T''$ agree after a bijective matching  of their new predicate symbols.

The last statement comes from \eqref{exp} and \lemref{compact}.
  \eprf

 \begin{prop}\lbl{beth} ($T$ QEable).   $T$ has definable patterns iff $T\td$ is an expansion by definition of $T$ iff $|\Hom(J,S(M))|=1$ for all $M \models T$ (or for some sufficiently saturated $M$.)   More generally, let $J_0 \subset J$; if for all $M$, the restriction map $\Hom(J,S(M)) \to \Hom(J_0,S(M))$
 is injective, then $T\td$ is an expansion by definition of $T$ along with the predicates of $T\td$
 corresponding to $J_0$. \end{prop}
 \prf  Let $h \in \Hom(J,S(M))$.  Let $j \in J$, and consider a typical predicate of $T\td$ corresponding
 to $q=h(j)$, namely $d_qx \phi(x,y)$ for some $\phi$.   The fact that $h$ is a homomorphism is equivalent to
 implicit definability constraints on such predicates.  The assumption is that these definability constraints determine
 the interpretation of the predicate uniquely (given the interpretation for $q' \in J_0$.)  By Beth's theorem,
 $d_qx \phi(x,y)$ is definable (relative to similar predicates for $J_0$.)  
   \eprf

\begin{rem}   If $T$ has definable patterns then every type over $\emptyset$
has an extension to a definable type over $\emptyset$;  and also, by    \propref{duality1}, to invariant type that is co-definable over $\emptyset$. \end{rem}
     \ssec{Topology of $\co(T)$}

{\em For the rest of the section,   simply because the proofs were written with this assumption,
we assume $T$ is QEble; or at least, where indicated, Robinson.  It is likely that much can be generalized.}

 For stable theories, Shelah's finite equivalence relation theorem can be read as saying that 
 distinct elements of $\CU$ are separated by a finite {}definable partition.   Here we will consider $0$-definable family $\bar{E}=(E_d: d \in D)$ of 
(parameterically) definable $m$-partitions.  The condition that two types over $M$ are separated by $E_d$ for {\em any} $d \in D(M)$  can be formulated as a basic formula $\Xi'_{\bar{E}}$ of $\LS$, namely $\R_{E_u(x,x');D(u)}$.

\begin{lem}\label{equality} Let $p,p' \in \CU$ be distinct.   \begin{enumerate}
\item There exists a    formula  $\Xi$, finite conjunction of atomic formulas, with $\CU \models \Xi(p,p')$, such that 
$\TS \models \neg (\exists \xi) \Xi(\xi,\xi)$.  
\item  (QEble case.)  
Let $\Xi$ be as in (1).  Then there exists $m$ and a nonempty $0$-definable family $\bar{E}=(E_d: d \in D)$ of 
$m$-partitions, so that  $\TS \models \Xi \to \Xi'_{\bar{E}}$.   
 
\end{enumerate}
\end{lem}

\prf  (1) By the maximality of atomic types realized in $\CU$, \lemref{qe} (4), applied to the type of $(p,p')$.
\\
(2)  Let $M \models T$.    $\Xi$ is   a finite conjunction of basic formulas $\R_{\psi,\psi'}$; 
we consider a single one for simplicity (or using the elimination of finite conjunctions.)  
 Since $\TS \models \neg (\exists \xi) \Xi(\xi,\xi)$, there is no type $p(x)$ satisfying $\Xi(p,p)$.
 So $tp(x/M)=tp(x'/M)$ is inconsistent with the conjunction $C$ of all $ \neg (\psi(x,c)  { \wedge } \psi'(x',c))$.
 By compactness, there exists a finite $C_0 \subset C$ and
a finite $M$-definable partition of $M$ into definable sets $\phi_1(x,d),\ldots,\phi_n(x,d)$
such that for each $i$, each $\phi_i(x,d)  { \wedge } \phi_i(x',d)$ is already inconsistent  with $C_0$.
Let $E_d(x,x')$ be the equivalence relation:  \[\bigwedge_i (\phi_i(x,d) \iff \phi_i(x',d)) \] 
Then $E_d$ is part of a $0$-definable family of definable equivalence relations with $\leq 2^n$ classes,
and each having the required property (i.e. no  pair of equivalent elements can satisfy $C_0$.) 
\eprf

Similarly:
\begin{rem} \label{cov} 
  Let  $\g, \g'$ be two sets of formulas.   Assume:  whenever $\phi(x,y) \in \g$, 
there exist variables $x',y'$ with $\phi'(x',y') \in \g '$,  and vice versa.  
  Let $J = \CU_\g, J' = \CU_{\g'}$, and let $M$ be any model of
$T^{\pm}$.   Then there exists a canonical bijection between $\Hom(J,S(M))$ and $\Hom(J',S(M))$.\end{rem}

\prf     We may assume that $\g \subset \g'$, by comparing both to $\g \union \g'$. 
 In this case it suffices
to show that if $p$ is a $\g'$-type,  then
 $h(p)=q$ iff $h(p|\gamma)=q|\gamma$.     This is clear since
any homomorphism $h$ must preserve the `change of variable' relations $\R_t$ with 
$t=(\phi(x,y),\neg \phi'(x',y'))$ and $t=(\neg \phi(x,y),  \phi(x',y'))$.   

%
%
%

\eprf%

\begin{lem} \label{hf1} Let $a \in \CU$.  
\begin{itemize}
\item  We have $Ga=P(\CU)$ where $P$ is the $\LS$-atomic type of $a$.
\item The map $e_a: G \to \CU,  \, g \mapsto ga:=g(a)$ is continuous and and closed.
\item  $G$-conjugacy is an intersection of $|L|$ open relations on $\CU$.
\end{itemize}
\end{lem} 
\prf   (1)  This is  the homogeneity for atomic types,   \lemref{qe}.

 (2)  A basic closed subset of $G$ has the form 
\[F=\{g:  (gb,c) \in P \} \]
where $P \subset \CU^n$ is pp definable.  This makes continuity evident.

Since the basic closed sets are closed under finite intersections, and $\CU$ is compact,
it suffices for closedness to prove that the image of a basic closed set $F$ is closed; this is 
a set of the form 
\[ e_a(F) = \{ga:  (\exists g \in G) (gb,c) \in P \} = \{a':  (\exists b')\, P(b',c)   \wedge (a,b) \sim (a',b') \}\]    
where $\sim$ denotes $G$-conjugacy.  Let $Q$ be the maximal atomic type of $(a,b)$; 
then \[e_a(F) = \{a':  (\exists b')P(b',c)  \wedge   Q(a',b')  \} = \{a': R(a',c) \}\]  
where $R(x,z)$ is the pp formula $(\exists y) ( P(y,z) \wedge Q(x,y) )$; it is closed by definition. 

(3)  $(a,b)$ are $G$-conjugate iff for all atomic $Q,Q'$ such that $\TS \models \neg (\exists x)(Q(x) \wedge Q'(x))$,
we have $\neg (Q(a) \wedge Q'(b))$.  
 

\eprf

Let $\fgt$ be the set of elements of $G$ that act  infinitesimally  on  each type of $\CU$; this may be smaller
than $\fg$.  
 In the notation of \secref{topg}, $\fgt= \fg_X$ where $X$ is the disjoint union of all maximal atomic types   of $\CU$.
 $X/\fgt$ denotes the orbit space, i.e. the quotient of $X$ under $\fgt$-conjugacy.

\begin{prop} \label{max-atomic-metrizable}  Let $P$ be a maximal atomic type of $\CU$.
 \begin{enumerate}
\item $P_h:=P/{ {\fgt}}$ is Hausdorff
\item If $L$ is countable, 
$P/{{\fgt}}$ is metrizable.     
\end{enumerate}
  \end{prop}

\prf  (1) Let $N = \fgt=\fg_X$, with $X$ as above.   By \lemref{topglem1}, $N$ is  a closed normal subgroup of 
$G$, and $G/N$ is Hausdorff.  Moreover for $c \in P$, the map $G \to P$, $g \mapsto gc$ is closed, by \lemref{hf1}.
By \lemref{topglem2} (1) and
 \remref{topglem2r}, ${\fgt}$-conjugacy coincides
on $P$ with $\fg_P$-conjugacy, and $P_h=P/{\fgt}$ is Hausdorff.

Since $P_h$ is Hausdorff, the diagonal of $P_h$ is closed, so pulling back to $P$ we see that the graph  $E_{\fgt}$
  of $\fgt$-conjugacy on  $P$ is pp-closed.  Moreover  the pp topology on
$(P_h)^2$ coincides with the product topology.  

(2)   Since $E_{\fgt}$ is pp-closed, it follows from the pp-homogeneity of $\CU$  (quantifying out any parameters)
that $E_{\fgt}$ is $\bigwedge$-pp-definable.  As $\CU$ is e.c., if $(a,b) \notin E_{\fgt}$ then there exists a    pp-{}definable $C$
with $(a,b) \in C$ and $C \meet E_{\fgt} = \emptyset$.  Thus $E_{\fgt}$ is the intersection of $|L|$ open sets,
namely the complements of these sets $C$.   Now metrizability  of the quotient follows from
\lemref{second}.     

\eprf

 \begin{prop} \label{dense} Any $\bigwedge$-pp definable subset  $P$ of $\CU$ has 
  a dense subset of cardinality   $\leq |L|$.   Hence if $P$ is a maximal atomic type, then  $|Aut(P_h)| \leq 2^{|L|}$.

  \end{prop}

 \prf   
  Denote an image of $\CU$
 in $S(M)$ by $J$; we identify $\CU$ with $J$ and $P$ with $P(J)$.
     Three topologies are visible on $P$:  the intrinsic pp topology $\ft_p$;
 the topology $\ft_{p,ext}$ induced from the pp topology on $S(M)$; and the topology  
 induced from the usual 
 logic topology $\ft_l$   on $S(M)$, where a clopen set corresponds to a formula of $L(M)$;
 this last topology has   basis $B_l$ with $|B_l| \leq |L|$.  
 We have $\ft_p \subseteq \ft_{p,ext} \subseteq \ft_l$.
    For each $u \in B_l$ with $u \meet {P} \neq \emptyset$,
 pick $j_u \in u  \meet {P}$, and let ${D}:=\{j_u: u \in B_l, u \meet {P} \neq \emptyset\}$.   Then ${D}$ is $\ft_{p,ext}$-dense in ${P}$: 
  if $ U \in \ft_{p,ext}$   and $U \meet {P} \neq \emptyset$, then $u \meet {P} \neq \emptyset$
   for some $u \in B_l, u \subset U$.    Hence
 $j_u \in U \meet {D}$.   It follows in particular that ${D}$ is $\ft_p$-dense in ${P}$. 
 
 Thus the image $D_h$ of $D$ in $P_h$ is dense in $P_h$.  Since $P_h$ is Hausdorff, 
 an automorphism fixing a dense set is the identity; so any automorphism $\si$ of 
  $P_h$ is determined by $\si|D_h$.   Thus $|Aut(P_h)| \leq 2^{|L|}$.  
  
   \eprf
\begin{cor} \label{size0}  \begin{enumerate} \item $|{\CU}| \leq 2^{|L|}$
\item   $  |\Aut(\CU)| \leq 2^{2^{|L|}}$
\item   $|\CG| \leq 2^{|L|}$.  
\end{enumerate}

\end{cor}
\prf 
Let $M \models T$, $|M| \leq |L|$.  By \lemref{compact}   ${\CU}$  embeds  into   $S(M)$;
 thus $|{\CU}| \leq 2^{|L|}$; and so  $  |\Aut(\CU)| \leq 2^{2^{|L|}}$. 
 
By \propref{dense}, $\CU$ has a dense set $D$ of size $\leq |L|$, 
   Any automorphism $\si$ of $\CU$ fixing   $D$ (pointwise) has the property that
for a nonempty open $U$, $\si(U) \meet U \neq \emptyset$; i.e. $\si \in {\fg}$.   Thus the restriction
$\si | D$   determines $\si$ modulo ${\fg}$.  Since  
$|\CU| \leq 2^{|L|}$, we have 
 $|\CU^D| \leq 2^{|L|} $.   Thus  $|\CG| = |{G}/{\fg}| \leq  2^{|L|}$.      
 
 \eprf
 
 The third item is similar, but not quite comparable, to  the statement in \propref{dense}   that  the automorphism group of any
Hausdorff   type of $\CU$ has cardinality $\leq 2^{|L|}$.   \exref{topdyn} shows a Hausdorff $\Aut(\CU)$ of cardinality $2^{2^{\aleph_0}}$
is possible.

\begin{rem} \label{closedopen} $\{\si \in Aut(\CU): \si(F) \subset U \}$ is open, for any closed $F$ and open $U$.  
 \end{rem}
 
\prf  $F$ is an intersection of a family of basic closed sets $F_i$, that we may take to be closed under
finite intersections.  By compactness, $\si(F) \subset U$ iff
$\si(F_i) \subset U$ for some $i$.  So we may assume $F$ is basic-closed; similarly we may assume
$U$ is basic open.  Say
$F = \{p:  \R_{\phi,\phi'}(p,q) \}$, $U = \{p':  \neg \R_{\psi,\psi'}(p,q) \}$.  
So $\si(p) \in U$ iff $\neg \R_{\psi,\psi'}(p, q')$, with $q'=\si \inv(q)$.
Then $\si(F) \subset U$ iff there is no $p$ with $\R_{\phi,\phi'}(p,q) \wedge \R_{\psi,\psi'}(p,q')$.
Now embed $\CU$ in $S(M)$, with image $J$.  Then such 
 a $p$ exists in $J$ iff it exists in $S(M)$.  In $S(M)$, the existence of such a $p$
is a consistency question that amounts to $\R_{\theta,\theta';\alpha}(q,\si \inv(q))$ for a certain family of  $\theta,\theta',\alpha$.  Hence the set of pairs $(q, \si \inv(q))$ for which a $p$ exists is $\bigwedge$-pp;
the set of pairs for which it does not is pp-open.  Hence the condition on $\si$ is pp-open too.    \eprf  

  \begin{rem}  The natural map $Aut_L(M) \to Aut_{\CL}(S(M))$ is an isomorphism.  Injectivity is clear using the embedding
 $i:M \to S(M)$, mapping $m$ to the algebraic type $x=m$.  The image of $i$ (in any given sort) is the complement 
 of a basic relation of $\CL$, since it is precisely the set of types representing the formula $x=m$.  Any definable relation $\alpha(x_1,\ldots,x_n) $ on $M$ 
 is mapped by $i$ to a $\CL$-definable relation on $i(M)$, namely the negation of $\R_{x_1=y_1,\cdots,x_n=y_n; \alpha}$.  Thus
 any $\CL$-automorphism of $S(M)$ induces an automorphism of $M$ on the copy $i(M)$.   Finally $\phi(x,c) \in p$
 iff $\R_{\phi(x,y),x'=y'; \phi(y,y')}$ does not hold of the pair $(p,i(c))$; so a $\CL$-automorphism fixing $i(M)$ is trivial.\end{rem}

\ssec{Examples}

 \begin{example}    For finite $\gamma$, any
$0$-definable $\gamma$-type is represented by a unique element of $\CU_{\gamma}$;
it is uniquely characterized by an atomic formula of $\LS$ as in \exref{def0}, and so is fixed by any retraction.

More generally almost $0$-definable $\gamma$-types, i.e. definable types whose canonical definitions are
imaginary elements algebraic over $\emptyset$, can only be permuted among themselves by an 
$\LS$-retraction, and so are present in $\CU$. 
 
 When $\gamma$ consists of stable formulas,   ${\CU}_\gamma$ 
 is the discrete finite space of $\gamma$-types definable almost over $\emptyset$;
equivalently definable over  $\acl^{eq}(\emptyset)$.  
It was here that Shelah introduced imaginaries, and algebraic closure.

 A slightly larger class are the {\em densely definable} pattern types:  $p$ is {\em densely definable}
 if for any consistent $\phi $, for some consistent $\phi' $ implying $\phi$, and some
 $\psi$,    $p$ implies that $X=\psi$ on $\phi'$.  (Again one can check that this implies maximality of $p$.)
 When the underlying sort forms a complete type of $T$,   this is the same is definability.  
 Any densely definable is  represented
 by an element of the core.  Moreover if $p,p'$ are densely definable and densely equal, i.e. 
  for any consistent $\phi $, for some consistent $\phi' $ implying $\phi$, $p,p'$ have he same definition on $\phi'$,
  then they are necessarily represented by the same element.

\end{example} 

\begin{example}    For the random graph, in the home sort, $\CU$ has two elements, corresponding
to the two {}definable types (adjacence to all or to none.)   Similarly for DLO.    For the triangle-free graph, it is   the unique definable type.   
 \end{example}

    \begin{example} \label{skolem} Assume  $T$ has a model ${M}$ whose every element is {}definable.
    Then the underlying space of $\CU=\co(T)$ is nothing more than the type space over $\emptyset$.  Indeed we have as usual an $\LS$-embedding $\CU \to S({M})$,
    commuting with the two maps into $S(\emptyset)$; since $S({M}) \to S(\emptyset)$ is an isomorphism, the 
    map $\CU \to S({M})$ must be surjective.  
    
         This remains true for $\g$-types with distinguished variables $x$ and parameter variables $y$,  
         i.e. $\co_\g \cong S_\g(\emptyset)$, 
     provided $(\exists y) \phi \in \g$ for all $\phi \in \g$.    For this homeomorphism to hold, it suffices that   every element of ${M}^y$ be {}definable.
     
        Nevertheless, the  associated expansion  may not be trivial; see for instance \exref{tournaments}.
    
  Similarly, returning for simplicity to complete types, if every element of ${M}$ is algebraic, $\CU$ can be identified as a space with the Shelah strong types.   %
 
 \end{example}
 
 \begin{example}[cf \cite{vt}] \label{ziegler} Consider the basic ingredient of Ziegler's example of a non-G-compact theory:
 an   oriented circle with $\Zz/n\Zz$ action.   Or a relational  variant, taking a random dense subset
 of the  circle $\Rr/\Zz$, with the relation $y< x<y+1/n$.  
 In either case, ${\CU}$ is finite but nontrivial; it is essentially  $\Zz/n\Zz$  with the regular $\Zz/n\Zz$-action.  
  \end{example}

  \begin{example}[Connected Lie groups]
  For the circle $x^2+y^2=1$ in RCF with the rotation-invariant semi-algebraic relations (Poizat's example), 
  or for the oriented circle as in \exref{ziegler} but with the action of an irrational rotation, 
  ${\CU}$ is the standard circle.
  The embedding to $S_x(M)$ for a model $M$ can be taken to be via the Lebesgue-weakly random types.
  The retraction takes a type over $M$ to the unique coset of the infinitesimal subgroup containing it.
     \end{example}

\begin{example}\label{lowerbounds}  Countable theories with $\CU$ Hausdorff, of cardinality $2^{\aleph_0}$,
$|\CG| =2^{\aleph_0}$.   
           
\begin{enumerate}  
\item  Consider the model completion of the theory of graphs with infinitely many disjoint unary predicates $P_n$.  We consider the sort $S_{\gamma}$ where $\gamma$ is the graph adjacency formula
(considering $S_x$ would make no difference.)  
Let $M$ be a countable model.
There are $2^{\aleph_0}$   maximal definability patterns of $1$-types over $M$; one can  choose $\gamma(x,u)$ to hold for all $u \in P_n$, or for none; and this, independently of $n$.   These are the maximal atomic types of $\TS$.
 They must all be represented in 
$\CU$, hence $|\CU| = 2^{\aleph_0}$.  $\CU$ is Hausdorff; if $p \neq q$, say $(d_px)(\gamma(x,u)  { \wedge } P_1(u))  = P(u)$ while  $(d_qx)(\gamma(x,u)  { \wedge } P_1(u))  = \perp$.  Let $R$ be the atomic formula asserting that $\gamma(x,u)  { \wedge }P_1(u)$ is omitted, and $R'$ the atomic formula asserting that $\neg \gamma(x,u)  { \wedge }P_1(u)$
is omitted.  Then $\neg R, \neg R'$ are disjoint open sets separating $p,q$.  
We have $\Aut(\CU)=1$.  
 
  \item Let $L$ have a ternary relation $\gamma(x;y,y')$; we will concentrate on the sort $S_\gamma$
  (with distinguished variable $x$.)  In addition, as above, $L$ has  
  infinitely many disjoint unary predicates $P_n(y)$.  $T$ states that  each$\g(a;y,y')$ is a tournament:
  $(\forall x,y,y') \neg (\g(x;y,y')  { \wedge } \g(x;y',y) )$, $(\forall x,y,y')\g(x,y,y') \vee \g(x,y',y) \vee y=y'$.   
  Further, for each $m<n$, $T \models (\forall x,y,y')(P_m(y)  { \wedge } P_n(y') \to \g(x,y,y'))$.  
  Let $p \in \CU$.  Then $(dp_x)\g(x,y,y')$  defines a linear ordering, with $P_m$ earlier to $P_n$ if $m<n$.  
For any subset 
$W$ of ${\omega}$, there exists an automorphism $\si_W$ of $\CU$, flipping the ordering on $P_n$ for some $n$, such that $\si_W(p),p$
agree above $P_n$ iff $\alpha \in W$.  Thus $\CG \cong {\Zz/2\Zz}^{\Nn}$ and $|\CG| = 2^{\aleph_0}$.

\end{enumerate} 
       \end{example}

An example with $\fg_{\CU}=\Aut(\CU)$ and $|\Aut(\CU)|=\beth_2$:

 \begin{example}\label{topdyn}  Topological dynamics comes back into the picture if both some set theory, and a group action, are built into the theory $T$.   In our approach, the topological dynamics arises as an example via a  specific theory; in \cite{kns},\cite{kpr}, \cite{KPT}, by contrast,
 it is the first-order theory that is treated as an example of a   topological dynamics, via the type spaces of saturated models. 
 In \exref{Gactions} we will see how that universal minimal flow of any discrete group $\G$ is dual to $\CU(T)$ for an appropriate theory $T=T_\G$.  Here we give a  hands-on treatment of the case of $\Zz$.

Let  $T$ be the model completion of a  bipartite graph $R \subset P \times Q$, with an invertible map $s: Q \to Q$  generating a $\Zz$-action on $Q$.  We are interested in $\CU_\g$ where $\g(x,y)= \{R(x,y)\}$, with distinguished variable $x$ of sort $P$.   
 
 When $\CU$ is embedded into $S(M)$, 
we can identify an element $p$ of $\CU$ with a subset $(d_px)R(x,y)$ of $Q(M)$.  We   take $M$ so that  $Q(M)$ is a single $\Zz$-orbit; if we pick  momentarily a point of the $\Zz$-orbit, we can  view $\CU$ as a set of subsets of $\Zz$.  We have:
\begin{itemize}[leftmargin=0in]
\item $\CU$ is translation invariant.  Indeed there exists a basic relation  $R(p,q)$ asserting
that   $(x \in p, \si(x) \notin q) \vee (x \notin p, \si(x) \in q )$ is omitted.  
Then $S(M) \models (\forall p)(\exists q)R(p,q)$, so this must be true in $\CU$.  In particular the 
  family of subsets of $\Zz$ corresponding to $\CU$ does not depend on the choice of point. 
\item  $\CU$ contains all periodic sets.  Indeed the elements of $S(M)$ of order $m$ are captured by a basic relation; $S(M)$ contains $2^m$ such sets, so all of them must be in $\CU$.    In particular $\CU$ is dense in $2^\Zz$ with the   topology of pointwise convergence.
\item  For any $p_1,\ldots,p_k \in \CU$, any configuration that occurs on some interval of length $m$
(i.e. a $k$-tuple of subsets of $[b,\ldots,b+m]$) recurs infinitely often on other intervals.  (Otherwise
we could get rid of this configuration by an ultrapower and restriction to a $\Zz$-orbit,  finding a homomorphism that is not an embedding on $\{p_1,\ldots,p_k\}$.)
\item  Let $m \in \Nn$, and let $a_0,\ldots,a_k   $ be subsets of $\{0,\ldots,m\}$; $\bar{a}:=(a_0,\ldots,a_k)$.  
Let $W_{\bar{a}}  = \{(p_0,\ldots,p_k ):   (\exists b) (p_i | [b,\cdots,b+m] = b+a_i  )\}$
Given $p=(p_1,\ldots,p_k)$, let $W_{\bar{a}} (p)= \{p_0: (p_0,p) \in W_{\bar{a}}  \}$.  
Then the $W_{\bar{a}} (p)$ form a basis for the pp-topology.  We have $W_{\bar{a}} (p) \neq \emptyset$ provided 
$p \in W_{(a_1,\ldots,a_k)}$.   
\item  If $W_{\bar{a}}(p) \neq \emptyset$, and $W_{\bar{a}'}(p') \neq \emptyset$, then their intersection in $S(M)$ is nonempty,
$W_{\bar{a}}(p) \meet W_{\bar{a}'}(p') \neq \emptyset$ and even includes periodic sets; these are necessarily in $\CU$.

 Thus any two nonempty open sets in $\CU$  have a nonempty intersection.
It follows that $\fg = G$.   Any continuous map on $\CU$ into a Hausdorff topological space is constant.   

 \item  We have $|G|=2^{2^{\aleph_0}}$; moreover, unlike $\CU$,  $G$ has a Hausdorff quotient of that cardinality.    To see this let $I$ be a subset of the interval $(0,1)$ such 
 that $I \union \{1\}$ is a basis for $\Rr$ as a $\Qq$-vector space.  
 The dynamical system $(\Rr/\Zz)^I$, with transformation 
$(a_i) \mapsto (a_i +i)$, is a minimal system.\footnote{To see this, reduce to finite linearly independent $J$; if 
  $Y$ is a closed invariant subset of $(\Rr/\Zz)^J$, translate
 so that $0 \in Y$; then $Y$ is the closure of  the subgroup  generated by the element $(j)_j$; so $Y$ is itself a closed subgroup; so some rational linear relation holds along  it; in particular $\sum m_i \alpha_i=0$, contradiction.}
For $i \in I$,  we define an 
 element $c_i$ of $S(M)$, namely the type corresponding to the set of $n \in \Zz$ such that $ni$ lies in $[0,1/2) / \Zz$.    Let $p_i$ be the atomic $\CL$-type of $c_i$.  Then $p_i$ is a maximal atomic type,
 and the realization set of $p_i$ in $\CU$ forms a copy of $\Rr/\Zz$, on which $\Rr/\Zz$ acts $\CL$- automorphically.  (It is easier, and sufficient for our purposes, to find a copy of $\Zz$
 (appearing as the image of $\Zz i$ in $\Rr/\Zz$), on which $\Zz$ acts by $\CL$-automorphisms;
 the sets are shifts of each other, and the $\CL$-2-type sees this.)  
 Moreover,  for distinct elements $i_1,\ldots,i_k$ of $I$, these $1$-types are almost orthogonal - they generate complete   $k$-types.   
This demonstrates an action of  $(\Rr/\Zz)^I$ (or $\Zz^I$ in the easier version) on a subset of 
$\CU$; by \lemref{qe}, it follows that $\Aut(\CU)$ has $(\Rr/\Zz)^I$ as a homomorphic image. 
 
\item 
  We saw that $\fg = G$, while a large Hausdorff quotient exists:  we have $\fg_q=1$ for certain complete pp types $q$, arising from homomorphisms of $G$ into the circle group.
    We will see in \exref{Gactions-e} the existence of other complete pp type $p \subset \CU$, such that $G$ induces a 
 countably infinite group  $G_p$ of automorphisms of $p$.   As $G_p$ is quasi-compact it cannot be Hausdorff.

\end{itemize}
\end{example}

\end{section}

 \begin{section}{Automorphisms of the core,  and the Lascar group}
\label{lascar-s}
%

     We return to the setting of a complete first order theory.  
     
\ssec{Lascar distance} \label{lascardistance}   
    Let $N$ be a  model  of $T$.    We will call two elements $a,a'$ of the same sort in $N$  
  {\em  Lascar neighbors} if for every $0$-definable family $\bar{E}=(D,E_d)_{d \in D} $ of finite partitions,  $N \models (\exists d) aE_d a'$.   
Equivalently, for any formula $\phi(u)$ consistent with $T$, and finite set $\gamma$ of formulas, there exists $b \in \Phi(N)$
with $tp_{\gamma}(a/b)=tp_{\gamma}(a'/b)$.  The Lascar neighboring pairs are the solution set of a  partial type $L_1^2(x,x')$.

 For a type $q(x,x')$ let us write $L_1^2(q)$ if $q(a,a')$ implies $L_1^2(a,a')$.
For a pair of $1$-types $p,p' \in S_\g(M)$, we define:  
\[ L_1(p,p') \iff (\exists q)(L_1^{2}(q)  { \wedge } \pi_1(q)=p  { \wedge } \pi_2(q)=p') \]
 So $L_1$ is a binary $\bigwedge$-pp-definable relation of $\LS$.   In particular, $L_1$ is also defined on $\CU$.

 If $L_1(p,p')$ holds, we say that $p,p'$ are Lascar neighbors, or have 
{\em Lascar distance at most $1$}.  We   define, in any $S(M)$ or in ${\CU}$, the symmetric relations $L_n$
of Lascar distance at most $n$:
\[L_n(x,y) \iff (\exists x=x_1,\ldots,x_n=y)\bigwedge_{i<n} L_1(x_i,x_{i+1}) \]
and the Lascar equivalence relation
\[{\EL} = \union_n L_n \]

  Call $p_1,p_2$  {\em close neighbors} if  for every definable family of finite local partitions
 $(E_b)_{b \in B} $, for some $b \in B(N)$ and $d' \in D(N)$, $x_i E_b d' \in p_i$.  This implies that $p_1 \union p_2 \models {L_1^2}(x_1,x_2)$, and
in particular $p_1,p_2$ are neighbors.

 In any $\aleph_0$- saturated model $N$, we have 
   $|N^x/\EL| \leq 2^{|L|}$.  
 
 By \lemref{qe},  each $L_n$ can also be written as a conjunction of atomic formulas of $\LS$.

We remark that for $p,p' \in S(M)$, we have  $pL_1p'$ iff for any consistent $\phi$, two realizations
of $p,p'$ can have the same type over {\em some} realization of $\phi$, {\em not necessarily in $M$}.
Strengthening the requirement to ask for a witness in $M$ leads, in $J$, to the equality relation $p=p'$;
see \lemref{equality}.

 Let $\Las_{\CU} : =  {\CU}/{\EL}$, and $\Las_M = S(M)/{\EL}$. (Sort by  sort.)

 \begin{prop} \label{lascar1} Let $j: {\CU} \to S(M)$ be an $\LS$-homomorphism.  Then (sort for sort)  $j$ induces a 
 bijection $j_*: {\Las_{\CU}} \to {\Las_{M}}$.  
  \end{prop}
 
 \prf We can find a homomorphism $r: S(M) \to J:=j{\CU}$ with $r| {J} = Id_{{J}}$.
   (\propref{prop1}).   Let $g=j \inv \circ r$.  
 Since $r,j$
 are homomorphisms, for $a,b \in {\CU}$ we have ${\CU} \models L_n(a,b)$ iff $S(M) \models L_n(ja,jb)$.
 Thus $j$ induces an injective map  ${\Las_{\CU}} \to {\Las_{M}}$. 
  It remains to show that it is surjective; it suffices to show that $r: S(M) \to J$ preserves Lascar types. 
We have this is a strong form:

\claim{A}  For all $p \in S(M)$ we have $L_1(p,r(p))$; in fact for {\em any} $a \models p$ and $b \models r(p)$
we have $aL_1b$.  

Indeed let $p  \in S_x(M)$, $p'=r(p)$.  Since $p' \in J$ we have $r(p')=p'$.  Let $q \in S_{x,x'}(M)$ be any type extending $p(x) \union p'(x')$, and let $q' = r(q)$.     Then $q'$ extends $p'(x) \union p'(x')$,  so 
$q' \vdash (xL_1x')$ (witness: $M$.)  But by \ref{emptyset},  $q| \emptyset = q' | \emptyset$.  Thus $q \vdash (x L_1 x')$.
This proves the claim and the proposition.  
\eprf

We would of course prefer to say that 
$j_*: {\Las_{\CU}} \to {\Las_{M}}$ is an isomorphism, not just a bijection.  However $\Las_M$ does not classically carry any structure, beyond that of a set acted on by $\Aut(M)$.  We can thus do not better than compare the two as permutation groups.

Let $G=\Aut(\CU)$, with the topology described in \secref{pptop}, and let $\fg=\fg_{\CU}$ be the subgroup of infinitesimal autormophisms with respect to the action of $G$ on $\CU$ (Appendix \ref{topg}).
Let $\CG = G/\fg$;  so $\CG$ is a compact Hausdorff topological group.   
As in \propref{max-atomic-metrizable}, we also let $\fg_{P}$ denote the infinitesimal subgroup 
with respect to the action of $G$ on $P$, and  let $\fgt := \meet_P \fg_P$ be the intersection of $\fg_{P}$ over all maximal atomic types $P$ of $\CU$.  So $\fgt \leq \fg$.

Fix a homomorphism $j: \CU \to S(M)$.  Then $j_*$ identifies ${\Las_{\CU}}$ with ${\Las_{M}}$,
and induces  a homomorphic embedding of $G=\Aut(\CU)$ into $Sym({\Las_{M}})$ (we will also denote it $j_*$.)

Define the Lascar group ${G_{\Las;\g}}$   as the image   of $\Aut(M^*)$ in 
the group of permutations of $\Las_{M;\g} := S_{\gamma}(M)/{\EL}$;
where $M^* \succ M$ is a sufficiently saturated extension, and ${\Las_{M}}$ is   identified with
 ${\Las_{M^*}}$ via the restriction map on types.   
 
\begin{lem}  \label{4.2}   \begin{enumerate}
\item   The image $j_*(G) \leq Sym({\Las_{M;\g}})$  is  precisely  the Lascar group ${G_{\Las;\g}}$.  
\item (Taking $\g$ rich enough).  If $g \in \fg$ then $j_*(g)$ is the identity on $\Las_{\CU}$.    
In fact, for $g \in \g$ and $p \in \CU$ we have $p L_2 g(p)$. 
\end{enumerate}
 \end{lem}
\prf
(1)  We first show that $j_*(G)$ falls into the Lascar group $Las_S$.  $g \in Aut(J)$.  Let $M^* \succ M$ be a highly saturated
 and homogeneous extension  of $M$
Let $p \in J$ be $Aut(J)$-conjugate, $q=g(p)$, 
and let $a, b \in M^*$, $a \models j(p), b \models j(q)$.  Then in particular $tp(a/\emptyset)=tp(b / \emptyset)$
(\exref{emptyset}).  
So there exists $\gamma \in Aut(M^*)$ with $\gamma(a)=b$.  It follows that $\gamma$ maps the Lascar type of $a$
(and of $p$) to the Lascar type of $b$ (and of $q$).  This applies to $*$-types too, and we can take $p$ to be rich enough to enumerate all Lascar types of elements of $S$.  This will show that the permutation of $Las_S$ induced by $g$ and by $\gamma$ coincide.


In the converse direction, it suffices to show (for any $M$) that for any 
   $\si \in \Aut(S(M))$, the permutation induced by $\si$ in ${\Las_{M}}$ lies in the image of $j_*$.
      Let   $r: S(M) \to J$ be a retraction.  Then $r \circ \si$
defines an automorphism of $J$.  Since we saw that $r$ preserves Lascar types (Claim A of \propref{lascar1}),
$r \circ \si$ induces on $J/{\EL}$  (identified with ${\Las_{M}}$) the same permutation as $\si$.   This shows that $j_*$ is surjective.
 
(2)   Let us identify $\CU$ with $J=j(\CU)$.  
Let $g \in \fg$, $p \in J$, $p'=g(p)$.  We will show that $L_2(p,p')$ holds (in $S(M)$;  equivalently in $J$.)
Let $\bar{E}$ be a {}definable family of finite partitions, as in \ref{lascardistance}.       
  Let $\Xi' = \Xi'_{\bar{E}}$, as defined above \lemref{equality}; so $ \Xi'(\eta,\eta')$ implies that $\eta,\eta'$ are not close
  Lascar neighbors; in particular we have $\neg \Xi'(p,p)$.  
 So $\neg \Xi'(p,\eta)$ defines an open neighborhood of $p$.   Likewise $\neg \Xi'(p',\eta)$ defines an open neighborhood of $p'$.
  Since $g$ is an  infinitesimal automorphism, 
the intersection of these $g$-conjugate open sets is nonempty, so for some $q \in J$ we have
$\neg \Xi'(p,q) \wedge \neg \Xi'(p',q)$.  As $J \subset S(M)$ we can view $p,p',q$ as types over
$M$; let $a,a',c$ be realizations; then there exist  $d,d' \in D(M)$ with $a E_d c E_{d'} a'$.  
Since this holds for all $\bar{E}$, and any finite number of $\bar{E}$ have a common refinement, it follows
that in some elementary extension $M^*$ of $M$ there exists $c^*$ such that for any $\bar{E}=(D,E_d)_{d \in D}$,
for some $d,d' \in D(M^*)$,  $aE_dc^* E_{d'} a'$.  Now by the definition of $L_1^2$ (\secref{lascardistance}),  it follows that $L_1^2(a,c^*)$ and
$L_1^2(c^*,a')$; so with $q^*=tp(c^*/M)$ we have $L_1(p,q^*)$ and $L_1(q^*,p')$,  hence $L_2(p,p')$.


 \eprf

To define the full Lascar group ${G_{\Las}}$, we take $\g$ to be the set of all formulas, in countably many variables in each sort (both distinguished and parameter variables.)
$Las_T$   is {\em not} in general the inverse limit of 
 $Las_{\g'}$ for finite $\g' \subset \g$.     Let $G=\Aut(\CU)=\Aut(\CU_\g)$.   Define a subset $L_1$ of $G$:
 $g \in L_1^G$ iff $(p,g(p)) \in L_1$ for all $p \in \CU$.  Since $L_1$ is a closed relation on $\CU$, $L_1^G$ is a closed subset of $G$ in the pp topology.    Also denote by $L_1^{\CG}$ the image of $L_1^G$ in $\CG$.
 We will potentially just write $L_1$ for any of these.  Note that $L_1$ is a closed, conjugation-invariant subset  of $G$,   hence this is also the case  for $\CG$.  
 
 Let $M$ be a sufficiently saturated model of $T$, $j: \CU \to S(M)$ an $\LS$-embedding, $J=j(\CU)$,
 $r: S(M) \to J$ a retraction.  We have a map:  $\si \mapsto \alpha(\si):= r \circ \si | J$ from $\Aut(M)$ to $\Aut(J)$.  
 Now by Claim A, $r \si \inv (p) L_1 \si \inv(p)$ for any $p \in S(M)$; so
  $\si r \si \inv (p) L_1 p$; or $\si r q L_1 \si(q)$, for $q \in S(M)$; thus $r \si r \tau (p) L_1 r \si \tau (p) $;
  so $\alpha$ induces a homomorphism $\Aut(M) \to \Aut(J)/{\genl}$, 
  where ${\genl}$ is the group generated in $\Aut(J)$ by the closed normal set $L_1$. 
  
 Any automorphism fixing a model satisfies $\si(p)L_1 p$ and thus $r \si(p) L_1 rp $,
since $r$ respects $L_1$; for $p \in J$ this reads $\alpha(\si)(p) L_1 p$.
  Thus the group of strong Lascar automorphisms (generated, by definition, by
automorphisms fixing a model) maps to the identity, and $\alpha$ induces a homomorphism
$\Aut(G)/\Aut_L(G) \to \Aut(J)/{\genl}$.  The kernel of this homomorphism maps to the identity on $\Aut(J)$
and since $J/\EL  = S(M)/\EL$, fixes all Lascar types.   Thus the kernel is the identity, i.e. $\alpha$ induces 
an isomorphism $ \Aut(G)/\Aut_L(G) \to \Aut(J)/{\genl} $.  By \lemref{4.2} (2), $\fg \subseteq {\genl}$.  Hence:
  \begin{prop}  ${G_{\Las}} = \Aut(M)/\Aut_L(M) \cong \CG/{\genl} $. 
 \end{prop}

 We end this section with an example (very similar to one used by Pillay) showing that $L_1$, restricted to a sort $S$, can depend on the full ambient structure in other sorts, and not only on the induced structure on $S$.  Let $S$ carry the structure of a free $\Zz/2\Zz$-action, written $x \mapsto \pm x$, and no additional structure.  Let $S'$ be another sort,  and let $P \subset S  \times S'$ be a `random' relation 
 with the property that $P(x,y) \iff \neg P(-x,y)$.      Then on $S$ as a structure we have $L_1 =S^2$.
 On the other hand on $S$ as part of $(S,S_1)$ we have:  $aL_1 b$ iff it is not the case that $b=-a$.

 \ssec{}

\end{section}
\begin{section}{Elementary Ramsey theory}
\label{ert-sec}

Recall the notion of the Ramsey property from the introduction:  

\begin{defn}  
A complete first order theory $T$  is said to be {\em Ramsey} at a given sort $S$ if   any completion $T'$ of $T$ in the language
$L_P$ with a unary predicate $P \subset S$ adjoined has a model $N'=(N,A)$ ($N \models T$, $A =P^{N'} \subset N$) with an elementary submodel $M$ of $N$,
 such that $P \meet M$ is a 0-definable predicate on  $M$.

 On the other hand, if $T$ is an irreducible universal theory, we say that $T$ is Ramsey (at $V$) if   any
 irreducible universal $T'$ in $L_P$  has a model $N'=(N,A)$ ($N \models T$, $A =P^{N'} \subset N$) with an existentially closed substructure $M$ of $N$,
 such that $P \meet M$ is a qf 0-definable predicate on  $M$.  Here $P$ is a unary predicate of sort $V$.

 $T$ is {\em everywhere Ramsey} if it is Ramsey at $S$ for all $S$.  If $T=T_0^{eq}$ for a   theory $T_0$ with single
 sort $S$, it suffices to  check Ramseyness at $S_0^n$ for each $n$.
 
\end{defn}

Equivalently, for any $M \models T$ and any sufficiently saturated $N'=(N,A) \models T'$, there exists an elementary embedding $f: M \to N$ with $f \inv(A)$ 0-definable in $M$.  

If given two (or more) predictes $P,P',\cdots$, we can first move to a model where $P$ is definable, then to another
where $P'$ is definable; so the  definition would not change if we allow a finite coloring (to be made definable), or several predicates $P$ (or for that matter even an infinite number).
 
 \begin{lem} \label{ramseyqe} Assume $T_{\forall}=Th_{\forall }(M)$ is a Ramsey universal theory at $V$,  $M$ existentially closed. 
  Then $Th(M)$ eliminates quantifiers for formulas on $V$.
 \end{lem}
 
 \begin{proof}Consider a formula $\psi(x)=(\exists y)\phi(x,y)$, with $\phi$ quantifier-free.   Let  $A =\psi^M$.  By \lemref{ramsey-basic},
 some   
 pattern type $q$ containing $T$ is dense in $(M,A)$.   
 By the Ramsey property,
 $q$ is definable by some  quantifier-free  {}definable $D$.   If $T \models (\exists x,y)(\phi(x,y) \wedge \neg D(x))$,
 they by density there exists a finite $M_0 \subset M$ with $M_0 \models  (\exists x,y)(\phi(x,y) \wedge \neg D(x))$,
 and $(M_0,A) \models q$, i.e. $A \meet M_0 = \psi(M_0) = D(M_0)$; this is clearly impossible.  Thus 
 $\psi(M) \subset D(M)$.  Conversely, if $a \in D(M)$ and $M \models \neg \psi(a)$, then since $M$ is e.c.
 there exists a pp formula $\psi'$ incompatible with $\psi$, such that $M \models \psi'(a)$.  Again by density
 there exists a finite $M_0$ with $\psi(M_0) = D(M_0)$, and such that some $a' \in D(M_0)$ satisfies 
 $\psi'(a')$; this contradicts the incompatibility of $\psi,\psi'$.
   
\end{proof}

 \lemref{ramseyqe}   implies in particular that the class of finite models of $T_\forall$ has the amalgamation property, 
 a theorem of  \cite{nr}.     

\begin{prop} \label{ramsey-univ} Let $T$ be an irreducible universal  theory, with  a distinguished sort or family of sorts $V$.
There exists a unique minimal  expansion $T\ram_V$  to an irreducible universal theory  that is Ramsey at $V$. 
Any  $M$ with $Th_\forall(M)=T$ admits an expansion to a model of $T\ram_V$; the space of expansions is
just $Hom(J,S_\g(M))$.
\end{prop}

\prf  For simplicity we consider one sort $V$; form
 $T^*_V$ as above \defref{ramsey-def}.
 Let $\g$ denote the new `second-order' relations introduced in the * operation.
Apply   \propref{expandd} to $T^*_V$ to obtain 
   an irreducible universal theory $\widetilde{T}$ that has definable patterns at $\g$.
     Return now to the original sorts; call the result $T\ram_\forall$. 
   Note that $\widetilde{T} = (T\ram_\forall)^*$: the axioms of $\widetilde{T}$ are explicit and concern the new relations on the parameter sorts of $\g$, so they are visible already for $T\ram_\forall$.  
        It follows that 
    $T\ram_\forall $ is Ramsey.  If $T'$ is  an expansion of $T$ to an irreducible universal theory that is Ramsey,
 then $(T')^*$ has definable patterns with respect to the `second-order' relations introduced in the * operation,
  and so interprets $\widetilde{T}$  (in the  quantifier-free way described above \propref{expand} ).  It follows that
  $T'$ interprets $T\ram_\forall$.  
\eprf

%
%

When $V$ consists of all sorts, it follows from   \lemref{ramseyqe} that $ T\ram:=T\ram_V$ admits quantifier-elimination.  
We can apply these results to the Morleyzation of a complete first-order $T$.
Taking into account the uniqueness in \propref{expandd}, we obtain \thmref{expand2}, that we repeat below in a little more detail as \corref{expand2c}.

 Recall that a pair $T_1 \leq T_2$ of  universal theories satisfies {\em interpolation} if
  whenever  $R(x,y) \to S_2(y) \in T_2$ with $R \in L_1, S_2 \in L_2$ then for some $S_1 \in L_1$, 
$R(x,y) \to S_1(y) \in T_1$ and $S_1(y) \to S_2(y)$ in $T_2$.  

It is easy to see that  $T,T'$ are complete theories with quantifier elimination in languages $L,L'$
with $L \subset L'$, $T_\forall=T'_\forall|L$, and interpolation holds for the pair  $T_\forall,T'_\forall|L$, 
  then   $T = T'|L$.  (This is a special case of \lemref{interpolation}.) 

If $T$ admits quantifier-elimination, $T_1=T_\forall$ is the universal part of $T$, and $T_2$ is the universal theory of some expansion of a model of $T$, then it is clear that interpolation holds (with $S_1$ an $L$-formula equivalent in $T$ to $(\exists x)R(x,y)$.) 
In particular, interpolation holds between $T_\forall$ and the canonical Ramsey expansion of  $T_\forall$.

\begin{cor} \label{expand2c} (=\thmref{expand2})
Let $T$ be a complete theory.  
There exists an everywhere     Ramsey expansion $T\ram$ with this property: 
  if $T'$ is an everywhere Ramsey expansion of $T$ and $N' \models T'$, then  there exists an $L$- embedding
  $j:N' \to N$ with $N \models T\ram$, and so that the pullback of any definable subset of $N$ is definable in 
  $N'$.     
  
  $T\ram$ is unique up to bi-interpretability over $T$.  
The self-interpretations of $T\ram$ over $T$ form a group, $G\ram(T)$.
  \end{cor}
  
\prf   Here we may Morley-ize and so assume $T$ admits quantifier elimination.  
The universal theory $T_{\forall}$ admits a canonical Ramsey expansion  $T \forall \ram$
as a universal theory; this by \propref{ramsey-univ}.
   Let $M$ be an existentially closed model of $T \forall \ram$.  Then by \lemref{ramseyqe},
   $Th(M)$ eliminates quantifiers.  Let $T \ram = Th(M)$.   It is uniquely determined by the universal part of $Th(M)$
   which is just $T_{\forall} \ram$.

Minimality of   $T\ram$, as well as the fact that $T \subset T\ram$, follows from the minimality of   $ T\ram_\forall$ given by 
\propref{expandd}, 
taking into account the above remarks about interpolation.  

Let $T'$ be an everywhere Ramsey expansion of $T$; again we may assume $T'$ eliminates quantifiers.
Let $N' \models T'$.  Then we can expand $N'|_L$ to a model $N''$ of $T_{\forall} \ram$ (choosing a homomorphism
$J \to S(N')$, so that each basic definable set of $N''$ is also $N'$-definable.   Now $N''$ embeds into
some $N \models T\ram$ as it has the correct universal theory, giving the minimality statement.

Conversely, if $T'$ has the same minimality property, we may again assume $T'$ eliminates quantifiers
to prove first-order bi-interpretability with $T\ram$.  The minimality property shows that $T'_\forall$ is minimal in the
sense of universal theories, so in any case $T'_\forall$ and $T\forall \ram = T \ram _\forall$ are qf bi-interpretable over $T$. 
We may assume $T'_\forall = T\ram _\forall$.  
 As $T', T \ram$ admit QE, and have the same universal theory, they are now equal.  
 
Since $T\ram$ admits QE, any self-interpretation of $T\ram_\forall$ over $T$ as a universal theory, extends uniquely to a bi-interpretation of $T\ram$ over $T$ as a 1st-order theory.

\eprf

\ssec{Continuous logic version}  To see how this unifies Ramsey-type phenomena,
we also formulate the continuous logic version. 

  In continuous logic, as presented e.g. in \cite{hensonetal},  $V$ comes with a distinguished metric.   An $n$-place predicate $P$ on $V$ is   interpreted as a bounded real-valued function on $V^n$, uniformly continuous with respect to the metric.  
 A {\em universal theory} is a
  a family of assertions that the   values of a finite number of predicates $P_1,\ldots,P_k$
  lies in a given compact subset $C$ of $\Rr^k$:  $(\forall x)((P_1(x),\ldots,P_k(x)) \in C)$. 
  A {\em free pattern type} is a maximal universal theory in $L(X)$ whose restriction to  $L$ is $T_{\forall} $.    
 $p$ is {\em finitely satisfiable} in $(M,X)$ if for any   quantifier-free $\phi$ in $L(X)$
  any $\e>0$,  and any $a\in M^k$ there exists a (finite) 
$M_0 \subset M$   such that $(M_0,X|M_0) \models p$, and $b$ from $M_0^k$ with
$|X(a)-X(b)| < \epsilon$.   (Similarly for pattern types for externally definable sets.)

  Equivalently, there exists an elementary extension $(M^*,X^*)$ of $(M,X)$ and an embedding $f:M \to M^*$, such that $(M, f \inv(X^*)) \models p$.      \lemref{ramsey-basic} 
remain  unchanged.    We say that a  theory $T$ is a {\em Ramsey theory}  at  $V$ (or has the Ramsey property at $V$)   if   all free  pattern types for $T$ on $V$ are  definable.     
This is also equivalent to the definition given at the beginning of the section (taken verbatim, with $P$ interpreted as usual as 
real-valued.)

\ssec{Examples}
%
 \begin{example}  

\begin{enumerate}
 
\item Let $T$ be the theory of infinite sets $\Omega$, in the language of pure equality.  Then
$T\ram_{all} = DLO$.  

If $J=\co{T}$, $H$ a finite group acting on a sort $V$, we have in general $J_{V/H} = J_V^H$ (the $H$-fixed points of $J_V$.) 

If $V$ is the sort of    
ordered pairs in $T$, and $U = V/Sym(2)$ the sort of unordered pairs,  then $J_V$ is the two-atom Boolean algebra, and $J_U=J_V^{Sym(2)} = \{0,1\}$.   

\item   Infinite affine spaces $V$ over a finite field.  Then $T$  is a Ramsey theory at $V$,
and also at the sort $V^{[n]}$ of $n$-element subspaces of $V$; this is the
  affine  space Ramsey theorem, see \cite{GRS}.   A similar picture holds for projective spaces.  
  
  To study the sorts $V^n$, we may as well pass to the theory $Vect_{\Ff}$ of vector spaces over $\Ff$.
     Then $T$ is not Ramsey at the main sort $V$.   Indeed $T\ram$ is bi-interpretable with the theory of linearly ordered $\Ff$-spaces, such that each finite-dimensional vector space is lexicographically ordered with respect to some basis.  
(Note that this is not to be the same as a `random' linear ordering adjoined to $T$,
that makes an appearance in \cite{KPT}.)     
     
\item    Affine spaces $V$ over $\Qq$ form a Ramsey theory at $V$;
the  only    maximal patterns in $L[X]$ are the ones asserting $X = \emptyset$, or $X=V$.
This is essentially equivalent to Van den Waerden's theorem on arithmetic progressions \cite{vdW}, \cite{GRS}. 
Any consistent formula $\theta(x_1,\ldots,x_n)$ is  implied by another of the form:  
$\bigwedge_{i \geq 2} (x_i-x_0) =   \alpha_i (x_1-x_0)$.  And this formula is realized in any 
sufficiently long arithmetic progression $v_0,v_0+v ,\cdots,v_0+m $.   By Van den Waerden, 
for any set $A \subset V$, $\theta$ is realized either in $A$ or in $V \m A$; i.e. we can find an arbitrarily good approximation $M_0$ to a model, such that $(M_0,A) \models (\forall x)(x \in A)$ or 
  $(M_0,A) \models (\forall x)(x \notin A)$.   (Conversely, given a coloring of arbitrarily long intervals
  in $c$ colors, with no monochromatic arithmetic progression of length $l$, a compactness argument gives a coloring of $\Qq$ with no such arithmetic progresssion; but a model does contain a long arithmetic progression.)

 \item  Let $T$ be the theory of  $\Qq$-vector spaces $V$.  Then $T\ram$ includes the theory of ordered $\Qq$-vector spaces.  By contrast with e.g. \cite{ehn1}, it cannot  be interpreted in the
  the random linear ordering expansion of $T$.   It would be good to determine $T\ram$; 
 is it generated by DOAG along with the unary sets of the Ramsey expansion associated with the $\Qq^*$-action,
 as in   \exref{Gactions}? 
   \item  
 
  Let $V$ be an irreducible variety defined over a field $K$, and admitting a transitive action of
an algebraic group $G$.  
Consider the invariant Zariski structure on $V$:    a basic $m$-ary
relation   is a $G$-invariant $K$-Zariski closed subset of $V^m$.  

For $V=\Aa^1$,  $G$ the two-dimensional group of affine transformations, this theory is Ramsey at $V$.    This can be shown as a consequence of the generalized polynomial van der Waerden Theorem of \cite{bergelson-leibman}, though it uses only a small part of the strength of that theorem.  This is because for any formula $\phi(x_1,\ldots,x_n)$ consistent with the theory, there exist $\alpha_2,\ldots,\alpha_n \in K^{alg}$ such that for any $a \in V$ and $d \in K \m (0)$, 
$V \models \phi(a,a+d,a+\alpha_2 d , \cdots, a+ \alpha_n d)$; and using van der Waerden over
$K(\alpha_2,\cdots,\alpha_n)$ to find $a \in V, d \in K^*$ such that   
$\phi(a,a+d,a+\alpha_2 d , \cdots, a+ \alpha_n d)$ is monochromatic.  

In particular, it follows that affine spaces $V$ over an arbitrary infinite field $K$ are Ramsey.

\item   Hilbert spaces (restricted to unit ball). Here a unary predicate $X$ is interpreted not as a subset, but as a uniformly  continuous function on the unit ball.  The basic  definable predicate here is the norm $X(v) = |v|$.
Any continuous function $f(|v|)$ of the norm is definable, hence determines a pattern type.  One may guess that
these are the only pattern types, and indeed this is a central theorem of Dvoretzky-Milman \cite{milman}
(see \cite{milman2}, Theorem 1.2).     
 
\item  Let $T=\widetilde{BA}$ be the theory of atomless Boolean algebras; the main sort will be denoted $B$, and we will also consider $B^n$ for $n=1,2,\cdots$.  Let  $B_n \subset B^n$ denote the $n$-tuples of  pairwise disjoint nonzero elements, whose sum is $1$; then  $B^{[n]}=B_n/Sym(n)$ is the sort of $n$-partitions of $1$, or equivalently the sort coding  subalgebras of $B$ of size $2^n$. The  dual Ramsey Theorem of 
 \cite{dual} states precisely that $T$ is a Ramsey theory  in the sorts $B^{[n]}$.  
 
Let us compute $T\ram$ in full.  If $B$ is a Boolean algebra with $n$ atoms,
 and a linear ordering $a_1< \cdots < a_n$ on these atoms.  Then an element of $B$ can be identified
 with a subset of $\{a_1,\ldots,a_n\}$ or equivalently an $n$-string of zeroes and ones; viewed this way, we have the reverse lexicographic ordering on $B$, which agrees with the given ordering on the atoms.  An ordering of $B$ obtained in this way will be called an Rlex ordering.     This gives a 1-1 correspondence between finite linear orderings, and rlex-ordered finite Boolean algebras; it extends to an equivalence between the category $BAO$ of rlex-ordered  finite Boolean algebras, with injective, order-preserving Boolean homomorphisms,  and the category of finite linear orderings with   surjective maps $f$ such that $a<b$ iff $f \inv(a) <_{rlex} f \inv(b)$.  This makes it easy to see that $BAO$ 
 admits amalgamation.    A subalgebra  of an Rlex-ordered boolean algebra is also Rlex-ordered, as one can check, with respect to its own atoms.   It follows that $BAO$ is a Fraiss\'e class, with an $\aleph_0$-categorical 
 amalgamation limit $\widetilde{BAO}$.   Note that $B_n$ splits into $n!$ types in $\widetilde{BAO}$, differing only by
 rearrangement of the variables.  Using this one sees easily that a substructure of a model of 
  $\widetilde{BAO}$ realizing all $\widetilde{BA}$   types, also realizes all $\widetilde{BAO}$-types.  This will be useful for checking the Ramsey property.  
 
 Now (up to bi-interpretability over $\widetilde{BA}$) we have  \[T\ram_{all} = \widetilde{BAO}.\]  This is easy to deduce from the previous statement.  Viewed as a definable set in $\widetilde{BAO}$, $B_n$ is definably isomorphic to 
 $B^{[n]} \times Sym(n)$ (map $(a_1,\ldots,a_n) \in B_n$ to the pair $(b,\si)$, where $b$ is the image of $(a_1,\ldots,a_n)$ in $B^{[n]}$ and $\si(a_i)<\si(a_j)$ iff $i<j$.)   It follows that the sort $B_n$ is Ramsey,
 and $B^n $  similarly admits a 0-definable embedding into   a product of $\union_{k \leq n} B_k$ times a finite set.  
 (describing each $a_i$ as a word in the linearly ordered atoms of the algebra generated by $a_1,\ldots,a_n$.) 
 
 At this point the Hales-Jewett theorem becomes visible too, as a consequence of Ramseyness of $\widetilde{BAO}$.  We may think of the Boolean algebra of all subsets of $\{1,\ldots,N\}$;
 then a word in $n$ letters $1,\ldots,n$ of length $N$ can be presented as a $n$-tuple of disjoint elements of $B$,
 with sum $1$.  Let $B_n^* = \{(v_1,\ldots,v_n) \in B_n:  v_1 < \cdots < v_n \}$.  This is a complete type of
  $\widetilde{BAO}$.  
 Let $c$ be a finite coloring of  $B_n$.  Then $c$ (lifted to an elementary extension, then restricted) is definable on some elementary submodel $M$  of  $\widetilde{BAO}$; hence  in particular on a set of the form 
 \[ \{ (v_0 \union v_1,v_2,\cdots,v_{n}), (v_1,v_0 \union v_2,v_3,\cdots,v_{n}), \cdots (v_1,\cdots,v_0 \union v_{n} ) \}\]  where $(v_0,\ldots,v_n) \in B_{n+1}$ and $v_0<\cdots<v_n$.   Since $B_n^*$ is a complete type, $c$ must be constant on this $n$-element  set (called a combinatorial line.)


\end{enumerate}
\end{example}

The strongly minimal theories in (1-6) have a small $T\ram$.  
At the other extreme we have disintegrated strongly minimal sets, specifically free group actions.
Here the  canonical Ramsey expansion is essentially the same construction -
 up to Stone duality - as the universal minimal flow of topological dynamics.   From this point of view,
the canonical Ramsey expansion can perhaps be viewed as a relational generalization of the  universal minimal flow.

\begin{example}\label{Gactions}  Let $\G$ be a group, and $T$ the theory of free $\G$-actions on a set $V$.
Here $\G$ is viewed as discrete, and we assume for the sake of the exposition that $\G$ is infinite, though the same will hold
in the case of finite $\G$.  
 Form $T^*$ as above \defref{ramsey-def}, and let
$J:=\CU\ram(T)_V= \co(T^*_V)$, $\LS\ram=\LS(T^*_V)$.
 Then $J$
is a Boolean algebra $B$ with $\G$-action, and no additional structure.  The Stone space $S$ of this algebra is a compact space  with continuous $G$ action.  We will now show that it  is the universal minimal flow of $\G$.

For the Boolean algebra structure, see \exmref{completeba}.  The natural $\G$ action on types is clearly definable in $\LS\ram$: 
$\g \cdot p = q$ if  $R(a,x)  \iff \neg R(\gamma(a),y)$ is omitted in $(p,q)$.   Using the quantifier elimination enjoyed by $T$, it is easy to see that this generates all of $\LS\ram$.  For instance, when $\G=\Zz$ with generator $s$,
$p$ omits the pattern of three consecutive elements iff $p \meet (s \cdot p) \meet (s^2 \cdot p) = 0$ (in the Boolean algebra.)

Minimality:  suppose $S'$ is a closed nonempty $\G$-invariant subspace of $S$.  Let $B'$ be the Boolean algebra 
of clopen subsets.  We then have a surjective $\G$-Boolean algebra homomorphism $B \to B'$.  It must be an isomorphism, since $J$ is e.c.   But then $S'=S$.

Universality:  let $S'$ be any  minimal flow of $\G$. We must find a $\G$-invariant continuous map $S \to S'$.

First note that any minimal flow  $S'$ is covered by a totally disconnected minimal $\G$-flow   $S^*$, on which $\G$ acts without fixed points; namely any minimal subglow of the  $\G$-flow of ultrafilters on $\G$.  Indeed if we fix $s_0 \in S'$, the map $\g \mapsto g s_0$
extends (uniquely) to a continuous map $f: \G^*\to S'$, where $\G^*$ is the space of ultrafilters on $\G$;  and $f$ is $\G$-invariant.  
%
 Given $1 \neq g \in \G$, it is easy to partition any $g^\Zz$-orbit on $S'$  into two or three disjoint subsets $x$ such that $x \meet gx = \emptyset$; putting these partitions together we find a partition of $\G$ into at most three sets $x$ with the same property.  Thus no ultrafilter  on $\G$ is fixed by $g$.  
 
 Hence, ignoring the topology, $S^*$ is a model of $T$. It can be viewed as a parameter sort in a model of $T^*$;
the $R$-type space over $U$ identifies with the Boolean algebra of all subsets of $S'$.  So there exists a homomorphism from this $\G$-algebra  to $J$; it restricts to a homomorphism from
the algebra of clopen subsets of $S^*$ to $J$; dually we find a $\G$-invariant continuous map $S \to S'$.  It follows that $S$ is a universal minimal flow for $\G$.  \end{example}

   See \propref{B5} and \remref{Gactions2} for an  alternative approach.  
   
   \begin{example}  \label{Gactions-e} As promised earlier, we prove the  the existence of a complete pp type $p \subset J=Core(T^*_V)$, such that  $G$ induces a 
 countably infinite group   of automorphisms of $p$.    Since $Aut(p)$ is quasi-compact, this implies that $Aut(p)$  cannot be Hausdorff.       
 
 Let $Y$ be a totally disconnected compact  flow of $\G$, such $Aut_\G(Y)$ (the group of homoeomorphisms of $Y$ commuting with $\G$)
 is countable,  and any closed subflow of $Y \times Y$ projecting onto $Y$ in either direction is either all of $Y \times Y$ or a finite union of automorphisms of $Y$.
 The Chacon example described in \cite{deljunco} is an instance; more generally,  with $\G=\Aut_\G(Y)=\Zz$, the totally disconnected graphic minimal sets of \cite{jauslander}.  
 
  Let $S$ be the universal minimal flow of $\G$.  If $f, g: S \to \G$ are two surjective $\G$-morphisms, then the image of $S$ in $Y \times Y$ under $(f,g)$ is a minimal subflow
  of $Y \times Y$, hence it must be the graph of an element $\alpha$ of $\Aut_\G(Y)$.  Hence $g=\alpha \circ f$, so the kernels of $f$ and $g$ coincide; so we have a closed, $\G$-invariant
  equivalence relation $E$ on $S$ such that $S/E$ is an isomorphic copy of $Y$; we rename it as $Y$;  this incarnation of $Y$ comes with a canonical quotient map $\pi: S \to Y$.
 
 Let $U$ be a clopen subset of $Y$, and let  $u = \pi \inv(U)$.  Then $u$ is an element of the Boolean algebra $B$ of clopen subsets of $S$, that we have identified with $J=Core(T^*_V)$.
 If $u' = g(u)$ for some $g \in Aut(J)$, then $g$ induces an automorphism of $S$; it respects $E$ and thus induces an automorphism $g_Y$ of $Y$; conversely $g|p$ is determined
 by $g_Y$.  This shows that $p$ and $G_p$ are countable.  
  
   \end{example}

   \begin{question} \label{dynamics}  \begin{enumerate}
   \item 
   Investigate   further the connection of the topological dynamics of a group $G$ to the model theory of $T=T_G$.    Compare the theory of joinings of dynamical systems to the theory of orthogonality of pp types in $\CU$.
It seems plausible that the maximal Hausdorff quotient of $\CU$  corresponds to the distal flows.  
   
   \item  Computing the canonical Ramsey expansion at   other sorts remains interesting;
 for pairs, we  certainly find a linear ordering and hence many other linear orders obtained by Boolean combinations
 with unary sets; I am not sure if these are all, and if anything further is needed at the ternary level and above.
 
\item  Presumably, the model completion of a single unary function behaves similarly, with 'colorings' that on the tree of 
 ancestors of a given element $a$ depend only on the distance from $a$, periodically or 'almost periodically' as above.

 \end{enumerate}

 \end{question}
 
\begin{example}  Let $D$ be a non-Zilberian strictly minimal set, $T=Th(D)$.  Assume more specifically that the language of $D$ is generated by a  symmetric ternary relation $R$, that we view as a set of unordered triples.  
Further assume $R$ occurs for at most $n-2$ unordered triples from any  $n$-element subset of $D$ ($n \geq 2$.)
   In this situation we encounter the striking orientation 
  construction of \cite{evans}.  Namely, by \cite{caro} Theorem 2.3, for any model $M$  there exist partial functions $f,g$ such that   
  \[ R(M) =\{(x,f(x),g(x)):  x \in D(M) \} \]
(let $f(a),g(a)$ be the second and third elements of the orientation if $a$ is the first element of a triple $\{a,b,c\} \in R$,
under the given orientation, and $f(a)=g(a)=a$ otherwise.)   Then $T\ram$ at the sort $D^2$ must include
such partial functions $(f,g)$.  We can extend them to total functions, setting $f(a)=a$ or $g(b)=b$ where undefined (possibly they are globally defined
in one / all e.c. models of $T\ram$.)  In any case we obtain an action of the free semigroup on two elements, giving a theory $T_e$ interpretable
in $T\ram$, and with $T_e \ram = T\ram$.   \end{example}
 \end{section}

\appendix

\begin{section}{Infinitary definablity patterns and the Ellis group}
\label{ellis-sec}
 
There is a standard parallel between definable types and invariant types in model theory; 
in the latter, $(d_px) \phi(x,y)$ is not definable, but rather a union of type-definable sets.  
We consider now  a richer language $\LSb$ reflecting partial infinitary definability of this kind.\footnote{While we will present it directly, it can also be treated as a special case of the construction of $\CU$,
applied to an infinitary Morleyzation 
$\bT$ of $T$, obtained
by  adding a predicate symbol for {\em every complete type} $r$, and axioms $r \to \alpha$ for each $\alpha \in r$.  This is  a primitive universal theory,
whose e.c. models are precisely the models of $T$ realizing all types over $\emptyset$, with the expected interpretation of $r$.
We can form $\co(\bar{T})$; it is equivalent to $\CUb$  as defined below.   Any relation   of $\co(\bar{T})$
  is easily seen to be equivalent to a conjunction of ones of the form $\R_t$ considered below.  This requires extending the $\CU$ construction to primitive  universal theories.}

  The sorts of $\LSb$ are 
the same as those of $\LS$, i.e. indexed by a set $\g$ of formulas of $L$, and a distinguished set of variables.  
   We   restrict $\g$ to have    at most countably many variables of each sort of $L$.   
 
 $\LSb$ contains in particular a relation symbol $\R_t$ for each tuple
$t=(\phi_1,\ldots,\phi_n;\alpha)$, where $\phi_1,\ldots,\phi_n$ are as before formulas $\phi_i(x,y)$,
but now $\alpha(y)$ is a complete type (for a given $\phi_i$, we take  $y$ to be a finite set of variables, 
while all but finitely many variables of $x$ are treated as dummy in $\phi_i$.)

The interpretation of $\R_t$ in a type space $S=S_\g M$ will be

  \[\R_{t} ^{S} =\{(p_1,\ldots,p_n) \in S^n:   \neg (\exists a \in \alpha(M)) \bigwedge_{i \leq n} (\phi_i(x,a) \in p_i) \]

This defines a closed subset of $S^n$.

It is clear that the set of true {pp} sentences is the same for all  models $M$ of $T$ that realize all finitary types over
$\emptyset$.   This determines an irreducible {primitive} universal theory $\TSb$.  
The earlier considerations go through:  $\TSb$ has a compact topological model, hence it is 
 ec-bounded, hence it has a unique universal e.c. model $\CUb$.  
 
 Let $\lambda_T$ be the number of 
 finitary types of $T$ over $\emptyset$.    A model $M$ realizing all types of cardinality $\lambda_T$ exists, and thus
 $|\CUb| \leq 2^{\lambda_T}$.

\begin{lem}\label{aut2}  
\begin{enumerate}
\item  Let $A$ be a  substructure of $S=S_x(M)$.  The $\LSb$-homomorphisms $A \to S$ form a closed set 
$\Hom_{\LSb}(A,S) \subset S^A$, containing the image of $\Aut(M)$ under $\si \mapsto \si_* |A$.
\item Let $A  \subset S$.   
Assume $M$ is $\aleph_0$-homogeneous.   Then the image of $\Aut(M)$
is dense in $\Hom_{\LSb}(A, S)$. 
\end{enumerate}
\end{lem}

\prf  (1) is clear from the definitions.


 (2)   Given finitely many types $p_1,\ldots,p_m \in A$, let $q_i = f(p_i) $, and consider   
       any neighborhood  $U_i$ of $q_i$ in $S$.   We have to find $\si \in \Aut(M)$ with $\si_*(p_i) \in U_i$.
We can find $c$ from $M$ and formulas $\phi_i(x,y)$ such that $U_i$ is defined by $\phi_i(x,c)$.
       Let $r=tp(c)$. Then $S \models \neg \R_{\phi_1,\ldots,\phi_m;r}(q_1,\ldots,q_m)$.
 Since $f$ is an $\LS$-homomorphism, $S \models \neg \R_{\phi_1,\ldots,\phi_m;r}(p_1,\ldots,p_m)$.
By definition of this symbol,  there exists $c'$ in $M$ with $r(c')$ and  $(d_{p_i}x)\phi_i(x,c')$ for each $i$.
 Let $\si \in \Aut(M)$ satisfy $\si(c')=c$ (using the $\aleph_0$-homogeneity of $M$.)  Since
 $\phi_i(x,c') \in p_i$, we have $\phi_i(x,c) \in \si_*(p_i)$.  Thus $\si_*(p_i) \in U_i$, as required.
  
\eprf

Let $\bar{G}=\Aut(\CUb)$,   $\bar{\fg}= \{g \in \bar{G}: (\forall U \in \ft)(gU \meet U \neq \emptyset)\} $,
and $\bCG = \bar{G}/\bar{\fg}$.  


We record the anaolog \lemref{size0}, moving up one power set: 
   
\begin{lem}\label{size}   Let  $\lambda=\lambda_T$, the number of types of $T$ over $\emptyset$ in finitely many variables. 

\begin{enumerate}
\item 
$ |\CUb| \leq 2^\lambda$.
\item  
 $ |\bar{G}| \leq \beth_2(\lambda)$
 \item   $ |\bCG| \leq 2^{\lambda}$
\end{enumerate}
\end{lem}
\prf (1) was already observed; (2) is an immediate consequence.  (3)  is proved as in \lemref{size0}.

 \eprf

\begin{rem} 
 \begin{enumerate}
\item  Any model $A$ of $\TS$ has a canonical `minimal' expansion $A_{min}$ to $\LSb$, 
where 
\[ \R_{\phi;r} \iff \bigvee_{\alpha \in r} \R_{\phi,\alpha} \]
We have $A_{min} \models \TSb$, since if a {pp} sentence $\alpha$ holds in $A_{min}$, say witnessed by $a_1,\ldots,a_n$, then any instance of $\R_{\phi;r}(a)$ holds only because some stronger statement
 $\R_{\phi,\alpha}$ holds; if $\TSb$ rules out $\alpha$ it certainly rules out the stronger version, but that
 involves only $\LS$, whereas $\TSb | \LS = \TS$.
\item   In particular there exists a homomorphism $ \bar{\iota}:  \CU_{min} \to \CUb$.  It restricts
to a homomorphism  $\iota: \CU  \to \CUb| \LS$; this must be an embedding since $\CU$ is e.c.
\item   The embeddings $\CU \to \CUb \to S(M)$ induce maps $\CU/\EL \to \CUb/\EL \to S(M)/\EL$;
since the composition $\CU/\EL \to S(M)/\EL$ is bijective, the two intermediate maps must be too.   
\item  If $\iota$ is bijective, then every invariant type $p$ of $T$ is definable ($p$ is represented by an element of $\CUb$; since $\bar{\iota}$ is   surjective, it must be represented in $\CU_{min}$, which means that
if a type $q(y)$ is contained in $(d_px)\phi(x,y)$, then so is some formula containing $q$; then use compactness.)  
\item  Thus in general $\iota$ is not bijective.  By \propref{1}, it follows in this case that $\CU \not \cong \CUb|\LS$ and in fact there is no embedding $\CU \to \CUb|\LS$.  
\item Remaining in the case that $\CU \not \cong \CUb|\LS$, let $M$ be an  $\aleph_0$-saturated, $\aleph_0$-homogeneous model 
  of $T$, and let $J$ be a copy of $\CU$ in $S(M)$.    Then a retraction $r: S(M) \to J$ cannot be approximated by automorphisms of $M$.  (Otherwise
by  \lemref{aut2}(1) it would be a $\LSb$-homomorphism;  restricted to some image  of $\CUb$
in $S(M)$ it must be an $\LSb$-embedding, yielding in particular an embedding of $\CUb|\LS$ into $\CU$.)
This contrasts with the retraction of $S(M)$ to the image of $\CUb$, and appears to indicate  that $\CU$ cannot be constructed purely using the  topological dynamics of $\Aut(M)$ acting on $S(M)$.

\end{enumerate}
 \end{rem}

\begin{example}\label{lowerbounds-b}  
   There are countable theories
                with    $|\CUb|=\beth_1,              \  |\bCG|=\beth_2$ .  (Compare \exref{lowerbounds}.)         
\begin{enumerate}  
 
\item  Take   the model completion of the theory of graphs with infinitely many unary predicates.  
Let $M$ be $\aleph_0$-saturated of cardinality continuum.  
We see that there are $\beth_2$ invariant types over $M$, with a choice of $0/1$ over each
of the continuum many types over $\emptyset$.  So $|\CUb| = \beth_2$.

\item  To see that one can have $|\bCG| \geq \beth_2$, 
let $L$ have two sorts $A,B$, and infinitely many independent unary predicates $P_n$ on $B$.
A basic relation $R \leq A \times B^2$ is given, and $T_\forall$ asserts that for any $a \in A$, 
$R(a)$ is a tournament on $B$;  further, $R(a)$ respects the lexicographic order:
\[ (\forall x,y,y') \bigwedge_{i<n} (P_i(y) \iff P_i(y')) \wedge \neg P_n(y) \wedge  P_n(y')  \to R(x,y,y') \]
   For $\alpha \in 2^{\omega}$,  let $Q_\alpha = \meet_n P_n ^{\alpha(n)}$, so that the $Q_\alpha$ are the complete types   with respect to the  unary predicates.  Let $\bJ$ be an embedded image of $\CUb$ in $S(M)$.  For any $p \in \bJ$,   $(dp_x)R(x,y,y')$  defines a linear ordering on the sort $B$, so that 
$Q_\alpha < Q_\beta$ if $\alpha$ is lexicographically strictly below $\beta$.  
For any subset 
$W$ of $2^{\omega}$, there exists an automorphism $\si_W$ of $\CUb$, such that $\si_W(p),p$
agree above $Q_\alpha$ iff $\alpha \in W$.  This is a copy of the Hausdorff compact (and separable) group 
$(\Zz/2 \Zz)^{2^{\aleph_0}}$, and shows that $|\bCG| \geq \beth_2$.

    \end{enumerate}
       \end{example}

Here is an example where $\CG, \bCG$ differ.  

\begin{example} \label{tournaments} Let $L$ have two sorts $A,B$, and infinitely many constants $b_1,b_2,\cdots $ in $B$.
A basic relation $R \leq A \times B^2$ is given, and $T_\forall$ asserts that for any $a \in A$,
$R(a)$ is a tournament on $B$; i.e. for $R(a,b,b)$ never holds, and for $b \neq b' \in B$ precisely one of $R(a,b,b')$ and $R(a,b',b)$ hold.
Further, $R(a,x,b_j)$ holds iff $x=b_i$ for some $i<j$.  
$T$ is the model completion.  Let $x,y$ be variables of sorts $A,B$
respectively, and consider the $x$-sort of $\CU$ and $\CUb$.  Then $\CU_x$ reduces to a single point $p$;
where $(dp_x)R(y,y')$ defines a linear order.  On the other hand $\CUb_x$ has two points $p,q$;
$(d_px)R(y,y')$ and $d_qxR(y,y')$ are both linear orderings, opposing on the generic type of $T$
(i.e. on nonconstant elements.)
Thus $|G|=1$, $|\bar{G}| = |\bCG| = 2$.     
\end{example}

%
%
%
%

\ssec{{The Ellis group}}
 In order to  compare with definitions of the Ellis group in the literature  (see \cite{kns}),
we consider a sort $\CU_x$ of $\CU$, corresponding to the set $\gamma_x$ of all formulas 
with distinguished variable $x$ (and some countable set of parameter variables for each sort.)


\begin{cor} \label{ellis} Assume $M$ is an   $\aleph_0$-saturated, $\aleph_0$-homogeneous model of $T$.    Let $E_M$ be the Ellis group associated with the action of $\Aut(M)$ on $S:=S_x(M)$.  Then $E_M \cong  \Aut(\CUb)$.
\end{cor} 
\prf  
 Let $j: \CUb \to {S}$ be an $\LSb$-embedding, $\bJ=j(\CUb)$.   Let $r: {S} \to {\bJ}$ be a retraction.  So $r \circ r = r$.  By \lemref{aut2}(2), for any finite
 $F \subset {S}$, $r|F$ can be approximated in  $S ^{F}$   by   automorphisms of $M$.  Thus $r$ lies in the Ellis semigroup ${ES_M}$, and is idempotent.   If  $b \in {ES_M}$,  then $r \circ b |{\bJ}$ is a homomorphism
 ${\bJ} \to {\bJ}$ hence an isomorphism (\propref{prop1}); let $s$ be the inverse isomorphism; then the homomorphism $(s \circ r): {S} \to {\bJ}$ is again in $ES_M$ by \lemref{aut2}, and $(s\circ r) (br) = r$.  This shows that any element $br$ of $ES_M r$
 generates $ES_M r$ as a left ideal, so $ES_M$ is a minimal left ideal.  
  The Ellis  group $E_M$
 can be taken to be the subsemigroup
 $rEr$, under composition, with identity element $r$. The set ${\bJ}$ is preserved under the elements of $rEr$,
 defining an action of $rEr$ on ${\bJ}$.  Each element of $rEr$ induces an $\LSb$-homomorphism of ${\bJ}$, and so an isomorphism.  Conversely by \lemref{aut2}, any $\LSb$-automorphism of ${\bJ}$ is
 obtained in this way.   We thus have a surjective homomorphism
 $E_M \to \Aut({\bJ})$.  It is injective since if $h \in rEr$ is the identity on ${\bJ}$, then $hr=r$, but $h = hr$
 since $h \in rEr$ and $r^2=r$.  So $E_m \cong \Aut({\bJ}) \cong \Aut(\CUb)$.  
\eprf

\begin{cor} \label{ellis2}  
 Let $M$ be an $\aleph_0$-universal, $\aleph_0$-homogeneous model of $T$.
Then
$E_M$  has cardinality at most $\beth_2(\lambda_T) \leq \beth_3(|L|)$.   
   \end{cor}
   
 \prf  When $x$ consists of countably many variables, this is immediate from
 \lemref{ellis} and \lemref{size}.    
   Note that if we take another copy 
$x'$ of $x$, and let $\g''$ consist of Boolean combinations of $ \gamma_x \union \gamma_{x'}$, then $\CU_{\g''} = \CU_x \times \CU_{x'}$,
and the diagonal is $\bigwedge$-pp-definable, namely it is the relation of omitting $\phi(x,y) \& \neg \phi(x',y)$ for each $\phi$.  Thus $\Aut(\CU_{\g'}) $ projects bijectively to $\Aut(\CU_x)$
and to $\Aut(\CU_{x'})$.  It follows that even if $x$ is allowed to be a large list of variables, 
$\Aut(\CU_x)$ projects bijectively to the projective limit of $\Aut(\CU_u)$ with $u$ ranging over
finite subsets of some fixed countable set of variables.     So we are reduced to that case. \eprf

 
 
\begin{rem} \corref{ellis2} is in fact valid  for any $\aleph_0$-homogeneous model $M$ ($\aleph_0$-saturated or not); the proof is the same, except that the Ellis group will be
isomorphic to the automorphism group of a universal e.c. model of an appropriately stronger primitive universal theory
than $\TSb$, ruling out   types not realized in $M$.

(Incidentally, computing $\Aut(\CUb)$ for \exref{topdyn} gives an example where the Ellis group for homogeneous 
models can look bigger than for the saturated model.    For the saturated model of $T$,  with infinitely many orbits,  $\CUb$ will be isomorphic to $\CU$ (a diagonal copy in each orbit of $\Zz$.)  
In particular $\Aut(\CUb)= \Aut(\CU)$.   But if we use a homogeneous model with $m$ orbits of $\Zz$, $\CUb$ will be the independent product
of a copy of $\CU$ in each orbit, and $\Aut(\CUb)$ will be the wreath product of $Sym(m)$ with $\Aut(\CU)$. )

\end{rem}

Here is an example of a countable theory whose $|\bG|$, and thus the Ellis group, have cardinality $ \beth_3$;  compare \ref{topdyn}.
 
\begin{example}  \label{topdyn2} 
The theory $T$ will again include a  bipartite graph $R \subset P \times Q$.  On $Q$ there are $\aleph_0$ independent equivalence relations $E_n$ with two classes each; they can
be viewed as giving a map $p$ from $Q$ to a torsor $A$ over the group ${\two}^{\Nn}$ (where $\two=\Zz/2\Zz$).  
   There are also commuting definable maps $s_i: Q \to Q$, satisfying $s_i(s_i(x))=x$;
   so that $s_i$ preserves the classes of $E_j$ for $i \neq j$, and flips the two classes of $E_i$.
Thus $p$ is a homomorphism; $p(s_i(x)) = s_i \cdot p(x)$ (where $s_i$ is identified
with the element of ${\two}^{\Nn}$ having a $1$ just in the $i$'th position.)     $T$ is model complete, with universal theory as described above.  

The sort $Q$ is stable, though not stably embedded.  But using \lemref{qe}, mutatis mutandis, $\CUb$ can be computed autonomously on this sort; 
 this implies that
in the sort $Q$, $\CUb$ has a single element in each class of the intersection $\meet_n E_n$;
i.e. $Q(\CUb)$ is a torsor for ${\two}^{\Nn}$, and $p$ induces a bijection $Q(\CUb) \to A$.
While $Q$ has a unique $1$-type, it has continuum many $2$-types; namely for each 
$g \in  {\two}^{\Nn}$ the type  $q_g(x,y)$ 
asserting that $g \cdot p(x)=p(y)$.  These types restricted to $\CUb$ are the graphs of bijections $Q(\CUb)\to Q(\CUb)$, defining again an action of  ${\two}^{\Nn}$ on $Q(\CUb)$, compatible with the others.   

Now each $a \in P$ defines a subset $R(a)$ of $Q$; we prefer to think of
it as a function from $Q$ to $ {\two}$.  Let $h:  {\two}^{\Nn} \to {\two}$ be a homomorphism
(not necessarily continuous.)  We define an atomic type of $\LS$ in the sort $P$, describing 
a function $f: Q \to \two$ such that on the $2$-type $q_g(x,y)$ we have $f(x)=h(g)+f(y)$.  For each $h$
there are precisely two such functions $f,f'$ with the same maximal atomic type, but with $f'=1-f$.
Both are represented in $\CUb$, and each one is atomically $\LS$-definable over the other.  

Given $h_1,\ldots,h_k \in \Hom({\two}^{\Nn}, {\two})$ linearly independent over the $2$-element field, one sees easily that
$q_{h_1},\ldots, q_{h_k}$ are orthogonal in $\CUb$, i.e. the atomic $k$-type is determined by the $1$-types.  
Choose a $GF(2)$-basis $(h_i)_{i \in I}$ for $\Hom({\two}^{\Nn}, {\two})$. Let $a_i \in \CUb$ represent $q_{h_i}$,
and let $b_i=1-a_i$.  Then for any subset $C \subset I \ $, the function exchanging $a_i,b_i$
for $i \in C$ and fixing $a_i,b_i$ for $i \notin C$ preserves all atomic relations $\R_t$ of $\CUb$,
and thus extends to an automorphism of $\CUb$.  It follows that ${\two}^I$ is a homomorphic
image of $\Aut(\CUb)$, which thus has cardinality $2^{2^{2^{\aleph_0}}}$. 
\end{example}

We conclude the appendix with a more general version of \lemref{aut2}  (see \remref{kpt2}.

Let $M \models T$, $S=S_\g(M)$.  View     $S$ as a  compact Hausdorff space
under the usual logic topology; and as   as an $\LS$-structure. 
For $\si \in \Aut(M)$, let $\si_*$ denote the induced $\LS$-automorphism of  $S$.  
Let $A \subset S$.
The set $S^A$ of functions $A \to S$ will be considered as a compact topological space, 
with the topology of pointwise convergence. 

Let $M_A$ denote the structure $M$ expanded with $\phi$-definitions $(d_px)\phi$ for each $\phi \in L$ and each $p \in A$.
By a qf type of $M_A$ over $\emptyset$, we mean a finitely satisfiable collection of formulas of the form $(d_px)\phi$.
(In (2) below, we really just need finitely many such formulas, along with a set of $L$-formulas.)
The hypothesis on realizing types in (2,3) below is thus true whenever $M_A$ is either saturated, or
a qf-saturated existentially closed model of the universal theory of $M_A$.  

\begin{lem}\label{aut}  
\begin{enumerate}
\item  Let $A$ be a  substructure of $S(M)$.  The $\LS$-homomorphisms $A \to S(M)$ form a closed set 
$\Hom_{\LS}(A,S) \subset S^A$, containing the image of $\Aut(M)$ under $\si \mapsto \si_* |A$.
\item Let $A  \subset S$.  
Assume $M$ is $\aleph_0$-homogeneous, and $M_A$ realizes all qf types over $\emptyset$.  Then the image of $\Aut(M)$
is dense in $\Hom_{\LS}(A, S)$.
\item  Assume in (2) that  $M$ is $\lambda$-homogeneous and $M_A$ realizes all qf types in $\lambda$ variables over
$\emptyset$, where
 $\lambda \geq |A| + |L| +|M_0|$,  
$M_0 \leq M$.  Let $f:  A \to S(M)$  be an   $\LS$-homomorphism.   
Then there exists $\si \in \Aut(M)$ such that for all $p \in A$, $\si(p)|M_0 =  f(p)|M_0$.  
\end{enumerate}
\end{lem}

\prf  (1) is clear from the definitions.

(2)  Let $f:  A \to S_\g(M)$  be an   $\LS$-homomorphism.  
%
 Given finitely many types $p_1,\ldots,p_m \in A$, let $q_i = f(p_i) $, and consider   
       any neighborhood  $U_i$ of $q_i$ in $S_\g(M)$.   We have to find $\si \in \Aut(M)$ with $\si_*(p_i) \in U_i$.
We can find $c$ from $M$ and formulas $\phi_i(x,y) \in \g$ such that $U_i$ is defined by $\phi_i(x,c)$.
       Let $r=tp(c)$.  For any $\alpha \in r$, $S(M) \models \neg \R_{\phi_1,\ldots,\phi_m;\alpha}(q_1,\ldots,q_m)$.
 Since $f$ is an $\LS$-homomorphism, $S(M) \models \neg \R_{\phi_1,\ldots,\phi_m;\alpha}(p_1,\ldots,p_m)$.
 Hence for some $c_\alpha$ with $\alpha(c_\alpha)$ we have $\alpha(c_\alpha)  { \wedge } (d_{p_i}x)\phi_i(x,c_\alpha)$ for each $i$.   As a consequence of    
 $\aleph_0$-saturation,   there exists $c'$ with $r(c')$ and  $(d_{p_i}x)\phi_i(x,c')$ for each $i \leq m$.
 Let $\si \in \Aut(M)$ satisfy $\si(c')=c$ (using the $\aleph_0$-homogeneity of $M$.)  Since
 $\phi_i(x,c') \in p_i$, we have $\phi_i(x,c) \in \si_*(p_i)$.  Thus $\si_*(p_i) \in U_i$, as required.
 
 (3) The proof is similar to (2), except that we consider all $p \in A$ and all neighborhoods
 $U$ of $q=f(p)$ defined by some $\phi(x,c)$ with $c$ from $M_0$ (allow $c$ to be a $\lambda$-tuple enumerating $M_0$.) \eprf

\begin{rem} \label{kpt2} On \lemref{aut} (2).  \begin{enumerate}
\item  Assume: for every tuple $c$ from $M$, there exists a formula $\alpha \in tp(c)$ 
such that $\Aut(M)$ is transitive on $\alpha(M)$.  Then the saturation assumption on $M_A$
is not needed in the proof of \lemref{aut}(2). 

\item  Assume every every element of $\CU$ is represented by a definable type in $S(M)$.  Then $M_A=M$, and the hypothesis
of \lemref{aut} (2) is simply that   $M$ is $\aleph_0$-homogeneous and $\aleph_0$-saturated.

\end{enumerate}
\end{rem}

\end{section}

\begin{section}{Universal minimal flow}

\ssec{}
We recall some definitions from topological dynamics.  
For any topological group $G$, a {\em flow} is a compact Hausdorff space $X$ along with a continuous $G$-action on $X$; a morphism of $G$-flows is a continuous $G$-equivariant map.   If $G$ has a dense subset of size $\kappa$ then so does $X$, so $|X| \leq 2^{2^\kappa}$.  
The flow is {\em minimal} if every $G$-orbit is dense.  It is {\em universal minimal} if it admits a  
morphism into any other flow $Y$. 
   All endomorphisms
of a  universal minimal flow $M$ are bijective\footnote{A fact due to Ellis; here is a possibly different proof:    if $f: M \to M$ is an endomorphism, $f(M)$ is a subflow
so $f(M)=M$.  Suppose for contradiction that $f$ is not injective.  Construct an inverse system
$(M_\alpha,f_{\alpha,\beta}: \beta \leq \alpha \leq \lambda)$, where $\lambda=|M|^+$,   each $M_\alpha = M$,
and each $f_{\alpha,\alpha+1} = f$.    
We set $f_{\alpha,\alpha}=Id_M$ and define $f_{\alpha,\beta}$  for $\beta < \alpha$ by   induction on $\alpha$.
At successor stages,    let $f_{\alpha+1,\beta} = f_{\alpha,\beta} \circ f$.  At limit stages $\alpha$,
we must define a map $f_\alpha: M \to \liminv_{\beta<\alpha}M_\beta$.  Such a map exists by
universality of $M$.  Thus $M_\lambda$ can be constructed.  But clearly $|M_\lambda| \geq \lambda > |M|$,
a contradiction.}  It follows that $M$ is unique up to a isomorphism.    The same discussion can be
carried out for {\em pointed minimal flows}, where morphisms, if they exist, are unique; in this case the universal one is easily unique, up to a unique isomorphism.  Any minimal subflow $Y_0$ of the universal pointed minimal flow
is a universal minimal flow (any $G$-map from $F$ to a minimal flow $Y$ must restrict to a map $Y_0 \to Y$.)

Recall the space of ultrafilters $\beta Z = \Hom(2^Z,2)$ on a set $Z$; it is  topologized as a closed subspace of $2^{2^Z}$, and thus compact and Hausdorff.  
For a discrete group $G$, it is easy to see that $(\beta G,1)$ is the universal minimal pointed flow of $G$

\ssec{} \label{B2}
Let $L$ be a countable language, 
   $M$ be a countable atomic structure, prime model of $Th(M)$, 
 $G=\Aut(M)$.   We view $G$ as a topological group, 
by taking $M$ to be discrete and giving $G$ the pointwise convergence topology.
We assume (for simplicity) $Th(M)$ has quantifier elimination,
and let $T$ be the universal theory of $M$.   

Let  $x_m$ be a variable for each $m \in M$, as one does in the definition of `diagrams' in elementary logic.  We have the tautological assignment $\ba: x_m \to m$ for these variables.
Let ${I}$ be the set of finite sets of these variables.   
For any  ${i}  \in {I}$,  let $a_{i}$ be the restriction of $\ba$ to ${i}$, and let $\phi_{i} $ be a formula (in variables $i$) isolating $tp(a_{i})$;
so $V_i = \phi_{i}(M)$ is   the $G$-orbit of $a_i$.  (Note that we treat $a_i$ not as an $|i|$-tuple, i.e. a function
$|i| \to M$, but rather as an $i$-tuple, i.e. a function $i \to M$.)
 
  Let $F_i= \beta V_i$, and let $a_i \in F_i$ denote the principal
ultrafilter on $a_i$.     If $i \subset i'$, we have a natural projection $\pi_{i',i}: F_{i'} \to F_i$.  

 Viewing $I$ as an index set, partially ordered  by inclusion, 
 let $(V,\ba)$ be the inverse limit of all the $(V_{i},a_{i}: {i} \in {I})$, and let 
 let  $(F,\ba)$ be the   inverse limit of the spaces $F_i$.    
 
    Note that $G$ acts on ${I}$ naturally; and if $g({i})={i}'$,  we have a natural bijection $V_{i} \to V_{{i}'}$ (change of variable according to $g$), and hence also $\iota_{i;g}: F_i \to F_{i'}$.        We thus obtain an action of
    $G$ on $V$ and hence on $F$.   
\begin{equation}\lbl{eq4}   g(\iota_{i;g}(a_i) ) = a_{i'}  \end{equation}
    (Here $\iota_{i;g}$ acts on the domain of the tuple $a_i$, within the set of variables, and then $g$ acts on the image of $a_i$ within $M$.)

       By \cite{zucker} Prop. 6.3,    
  $(F,\ba)$ is the universal minimal  pointed  flow of $G$.   A morphism $Y \to F$ being the same as a 
  coherent family of morphisms $Y \to F_i$, we obtain: 

\begin{lem}\label{minflow}  A minimal flow $Y$ of $G=\Aut(M)$ is universal iff there exist 
 continuous   maps  $\alpha_i: Y \to   \beta V_i$, with $ \pi_{i',i} \circ \alpha_i = \alpha_{i'}$,  
 and $\alpha_i (gy)   = \iota_{g \inv (i);g} \alpha_{g \inv (i)} (y)$.   
   \end{lem}

\ssec{}
Recall the construction $T^*_V$, that renders each subset of each $V_i$ externally definable:     we add a sort $V_i^*$ and   new 
relation symbols $R_i \subset V_i^* \times V_i$ to obtain a bigger language $L^*_V$, with no new axioms.   


Let $J$ denote $\co(T^*_V)$ restricted to  sorts corresponding to $R_i$-types on $V_i^*$; thus $V_i$ are parameter sorts.   Likewise let $S$ denote the space of types of $T^*$ over $M$ in variable sorts 
 $V_i^*$ (for some $i$.)       

\bigskip

Let $T\ram$ be the universal part of the minimal Ramsey expansion of $T$ (\thmref{expand2}).   

\begin{prop} \label{B5}  Let $M$ be a countable atomic model.  
Then  the space of expansions of $M$
to a model of $T\ram$, is the universal minimal flow of $\Aut(M)$.  
\end{prop}

\prf   Let $J,S$ be as above.   Recall (\propref{expandd}, \eqref{exp}) that 
 the space of expansions  of $M$ to a model of $T\ram$ is isomorphic to 
$\Hom(J,S)$.   We thus  have to show that $\Hom(J,S)$ is the universal minimal flow of $G=\Aut(M)$.  

Minimality of   $\Hom(J,S)$ as a $G$-flow, i.e. the fact that every orbit is dense,   follows from \lemref{aut2} (2).

By \lemref{minflow}, it suffices to find  continuous maps $\alpha_i: \Hom(J,S) \to   \beta V_i$, functorial in $i$
and compatible with the $G$-action.
 
We will use \lemref{duality1}, for the theory $T^*_{V}$, specifically for $\g=\{R_i\}$ the relations connecting
$V_i$ with $V_i^*$; with $A=B=M$ there, and $N \geq M$ a large model of $T^*$.    But first, fix a homomorphism $\rho: S  \to J= \co(T^*_V)$.  
Then any homomorphism $h: J \to S$ yields an endomorphism $h \circ \rho$ of $S$.  By \lemref{duality1}
we obtain an extension $r_h$ of $tp(\ba)$ to a global $\g$-type, finitely satisfiable in $M$; in particular,
we can restrict attention to the coordinates $i$, obtaining a global $R_i$-type in $V_i$, finitely satisfiable in $M$.
Such a type corresponds precisely to an ultrafilter on $V_i$.  Indeed, for $d \in V_i^*(N)$, let   $s(d/M) :=s(tp(d/M)):= \{c \in V_i(M):  aR_i c\}$; this subset of $V_i(M)$ has the same information as $qftp(d/V_i(M))$.  
Let $\alpha_i(h) = \{ s(a/M):   a R_i y_i \in r_h  \}$.     Then $\alpha_i(h)$  is an ultrafilter on $V_i(M)$;  for instance  if 
$s(a/M) \subset s(a'/M)$, then   $R_i(a,y_i) \& \neg R_i(a',y_i)$ is not satisfied by any element of $M$,
so it is not in $r_h$; hence   if $s(a/M) \in \alpha_i(h)$ then $s(a'/M)  \in \alpha_i(h)$, so that $\alpha_i(h)$ is upwards closed.   A similar argument shows that  $\alpha_i(h)$ is closed under intersections, contains each set or its complement, and that for $i \subset i'$, $\alpha_{i'}(h)$
projects to $\alpha_i(h)$.   

Continuity of $\alpha_i$:  Fix $d \in V_i^*$  and let $j=\rho(tp(d/M))$.
Then $s(d/M) \in \alpha_i(h)$ iff $d R_i y_i \in r_h$ iff $x_i R_i a_i \in h(j)$; the set of $h$ with this property is open by definition of the pointwise convergence topology on $\Hom(J,S(M))$.

  It remains to compare the $G$-actions .   Let  $g \in G=\Aut(M)$, $h \in \Hom(J,S)$, $g\inv h := g \inv \circ h$.
  Fix $i$ and let $i' = g \inv(i)$.   Write $\iota = \iota_{i';g}$.     Let   $w \subset V_i(M)$; we will show that  
\[w \in  \alpha_i ( h)  \iff  w \in \iota  \alpha_{i'} (g \inv h) \]
Let $p$ be a type in $V_i^*$ over $M$ with $s(p)=w$, and $p'(x') $ a type over $M$ of elements of  $V_{i'}^*$,  differing from $p$ only in the change of variable $i' \mapsto i$ determined by $g$, so that $s(p')=\iota \inv(w)$.     Since this change of variable is expressible via an $\R_t$-relation
between $p$ and $p'$, 
it remains true of $h \rho(p), h \rho(p')$.   In particular,
\[  w \in \alpha_i(h) \iff a_i \in s(h \rho(p)) \iff \iota \inv(a_i) \in s(h \rho(p')) \]  
By \eqref{eq4} of \secref{B2},  this is iff $g(a_{i'}) \in s(h \rho(p'))$ iff $a_{i'} \in s(g \inv h \rho(p'))$ iff 
  $\iota \inv(w) \in  \alpha_{i'} (g \inv h)$.
 
\eprf 

\begin{rem} \begin{enumerate}
\item   Let $M$ be any countable structure.   Then \propref{B5}   gives a description of the universal minimal flow of $Aut(M)$ in terms of expansions to $T^{ram}$,
where $T$ is the theory of $M$ expanded by a relation for each $Aut(M)$-orbit on $M^n$.   Alternativey one can use the infinitary pattern space of Appendix A.  
\item  
   In  the case of continuous logic,  $V^*$ should be replaced by the ind-sort of uniformly continuous maps $V(M) \to \Rr$, and $R_i$ by evaluation.     Presumably, a similar comparison to the Weil-Samuel compactification of $G$ should work, but I have not checked any of the details.  
 \end{enumerate}\end{rem}

\begin{rem} \label{bmt/z} The results of \cite{zucker}, \cite{bmt} (for the discrete logic case) read in this light as a dichotomy:  $J$ is sortwise finite or uncountable.

Indeed by \exref{completeba},
$J$ carries a complete Boolean algebra structure on each sort $V_i^*$.    Boolean algebras are always either
 finite, or  admit an infinite set $I$ of pairwise disjoint elements.  A complete Boolean algebra of the latter kind
 must have cardinality at least continuum, since the sums of two distinct subsets of $I$ are never equal. 
   
   If $J$ is sortwise  finite, then $T\ram$ is a sortwise finite expansion of $T$,  hence it has finitely many  qf types of each sort extending any given type of $T$.
  In this case  the model completion $\hat{T}$ of  $T\ram$ is
 $\aleph_0$-categorical if $T$ is, and in any case has dense isolated types, as $T$ does; and the space of expansion of $M$ to a model of $T\ram$ has a comeager $G_\delta$ orbit, namely  the expansions to an atomic model of $\hat{T}$.   
 
 On the other hand if $J$ is uncountable, then $\Hom(J,S(M))$ cannot be metrizable, indeed cannot admit
 a countable basis  for the topology. 
 For suppose is countable basis; an element $b \in B$ can be taken to be of the form $\{ h \in \Hom(J,S): (h(j_1),\ldots,h(j_n)) \in D \}$, with $j_i \in J$,
 and $D$ an open set in $S:=S(M)$.   
 Let $J_0$ be the countable set of all $j_i$ occurring in $B$.  Now if $h_1\neq h_2 \in \Hom(J,S)$, there exists $b_i $
 with $h_1 \in b_i$ and $h_2 \notin b_i$,  and it follows that $h_1(j_i) \neq h_2(j_i)$.   Thus each 
 $h$ is determined by $h(j), j \in J_0$.  But now by \lemref{beth}, each $h$ is definable from finitely many $h(j)$, 
 so $|J| \leq \aleph_0$. 
 
%

\end{rem}

\begin{rem}   \label{Gactions2}   Let us give another proof of \exref{Gactions} in light of \propref{B5}.  
  $\Gamma$ is an infinite group, 
$T$   the theory of free $\G$-actions.   We  give a proof of 
Let $M$ be the prime model  of $T$ (a single $\Gamma$ orbit.)
Then $Aut(M)$  is another copy of $\Gamma$ (acting on the right), with the discrete topology.  
Let $J=\CU\ram(T)$; it is a Boolean algebra with $\G$-action, and write $2$ for the $2$-element Boolean algebra.
 We thus have
two descriptions of the universal minimal flow $U$ of $\Gamma$:   by \propref{B5}, $U$ is the 
 space $Hom(J,S(M^*))$ of expansions of $M$ to a model of $T\ram$, so we can write:
\[ U = Hom_{\LS\ram} (J,S(M^*)) = Hom_{Bool,\G} (J, ^M 2) \]
while by \exref{Gactions}, $U$ is the Stone dual to $J$, i.e. 
\[ U =  Hom_{Bool}(J,2) \]

These are compatible by a   duality analogous to Frobenius reciprocity.      A $\G$-equivariant Boolean homomorphism from $J$ to the algebra of functions from $M$ into $2$ can be viewed as a $\G$-invariant 
function $J \times M \to 2$, where $\G$ acts trivially on $2$.  Thus
\[ Hom_{Bool,\G} (J, ^M 2) = Hom_{Bool,\G}(J \times M, 2 ) \]
By picking a point $m_0 \in M$ and evaluating there, we obtain a map $Hom_{Bool,\G}(J \times M, 2 ) \to Hom_{Bool}(J,2)$ which can easily be seen to be a bijection and a homeomorphism:
\[ Hom_{Bool,\G}(J \times M, 2 ) = Hom_{Bool}(J,2) \]
Compatibility with the right $\G$-action is also easy to check.

\end{rem}

 \end{section}

\begin{section}{ Hausdorff quotients}
 \label{topg} We include here some elementary statements on Hausdorff quotients of topological spaces.   Any topological space $X$ has a 
universal Hausdorff quotient, namely $X/E$ for $E$ the smallest closed equivalence relation on $X$.
In the homogeneous case, one can describe $E$ more effectively.

In this subsection, all quotients are given the quotient topology,
 i.e. the open sets are those whose pullback   is open.     

 Let $(X,\ft)$ be a topological space, $G$ a group acting on $X$ by homeomorphisms.  
  
For $W \subset X$ and $x \in X$, let 
$ W x \inv := \{g \in G: gx \in W \} $.   Also for $U \subset X$,  write 
$W U \inv : = \union_{u \in U} W u \inv  = \{g \in G: gU \meet W \neq \emptyset\}$.

Define the infinitesimal elements of $G$ (acting on $X$) to be 
\[  {\fg_X}= \meet_{\emptyset \neq U \in \ft} U U \inv     =\{g \in G:  \forall U \in \ft \ (  U \neq \emptyset \to  gU \meet U \neq \emptyset )\}  \]
 ${\fg_X}$ is a clearly a subgroup of $G$, invariant under automorphisms of     $(G,X,\ft)$
 and in particular normal. 

We can also write:
\begin{equation} \label{2}  {\fg_X}=\{g \in G:  (\forall U \in \ft) \ g \cdot  U \subseteq \cl(U)\}= \meet_{U \in \ft, u \in U  } \cl(U) u \inv
\end{equation}
Indeed if $gU \subset \cl(U)$, since $U$ is dense in $\cl(U)$, $U$ is not disjoint from
the open set $gU$.  Hence if $gU \subset \cl(U)$ for all $U$, then  $g \in {\fg_X}$.  
Conversely assume $g \in {\fg_X}$.   
Let $V$ be the open set $U \m g \inv \cl(U)$.  Then $gV \meet V = \emptyset$, so $V = \emptyset$, i.e. 
$gU \subseteq \cl(U)$. 

\begin{equation} \label{3}   {\fg_X}= \{g \in G:  (\forall U \in \ft) \ g \cdot  \cl(U) = \cl(U)\} \end{equation}
Let $g \in {\fg_X}$.  By \qref{2}, $gU \subseteq \cl(U)$.
 By continuity of $g$ we have $g \, \cl(U) \subseteq \cl(U)$.  Applying this to $g \inv$,
we have $g \inv \, \cl(U) \subseteq \cl(U)$, equivalently $\cl(U) \subseteq g \cl(U)$.
Thus $g \cl(U) = \cl(U)$.   

For a final characterization of $\fg_X$, recall the Boolean algebra $ro(X)$ of regular open subsets of $X$.  
$U \subset X$ is {\em regular open} if $U=int(cl(U))$.    Complementation in this Boolean algebra takes the form
 $U \mapsto int(cl(X \m U))$.
If $g \in \fg_X$
then $g(cl(U)) = cl(U)$ so if $U$ is regular open, then  $g(U) = g(int(cl(U))) = int (g(cl(U)))= int( cl(U) )= U$.
Thus $\fg_X$ fixes $r.o.(X)$ pointwise.   Conversely if $g $ fixes acts trivially on $ro(X)$,  then for any open $U$ we have
$g(U) \subset g(int(cl(U))) = int(cl(U)) $ or $g(U) \subset cl(U)$, and so $g \in \fg_X$.  
  
 \begin{lem} \label{topglem1}   Let $G \times X \to X$ be a group action, and assume $G,X$ are endowed with a topology so that 
 the action $G \times X \to X$ is  continuous in each variable.   Then ${\fg_X}$ is closed, and
 $G/{\fg_X}$ (with the quotient topology) is Hausdorff.   
\end{lem}

\prf 

To prove  $G/{\fg_X}$ is Hausdorff,  
it suffices to show that if $g_1,g_2 \in G$ and $g:=g_2 \inv g_1     \notin {\fg_X}$, then $g_1,g_2$ are separated by disjoint open 
${\fg_X}$-invariant sets.  Since $g \notin {\fg_X}$, by \qref{2}, for some $u \in X$ and $u \in U \in \ft$ we have
$g \notin cl(U) u \inv$, i.e. $gu \notin cl(U)$.  By  \qref{3}, ${\fg_X}$ stabilizes $cl(U)$, hence also $cl(X \m cl(U))$ and the complement $int(cl(U))$.  Thus $X \m cl(U)$ and $int(cl(U))$ are disjoint ${\fg_X}$-invariant open subsets of $X$;
we have $  u \in int(cl(U))$ and $g  u \in X \m cl(U)$.   So $g_2   u \in int(cl(g_2U))$ and $g_1 u = g_2 gu \in X \m cl(g_2 u)$.
Since $h:G \to X$ defined by $h(x) =   x \cdot u$ is continuous, and $\fg$ is normal, $h \inv (int(cl(g_2U))$ and $X \m int(cl(g_2U))$ 
are disjoint  ${\fg_X}$-invariant open subsets of $G$, separating $g_1$ from $g_2$.

%
\eprf

\begin{lem} \label{topglem2}  Assume $g \mapsto gx_0$ is  a  closed, surjective, continuous map, for any $x_0 \in X$.     Assume also that $X$ is T1 and $G$ is compact as a topological space (or just that the stabilizer of a point of $X$ is compact.)
Then \begin{enumerate}
\item  If $N$ is any  normal subgroup of $G$ with $G/N$ Hausdorff, then $X/N$ is
Hausdorff.  
\item     $X/{\fg_X}$ is the universal Hausdorff quotient space $X_h$ of $X$:  any continuous map from $X$ to a Hausdorff space factors (uniquely) through $X/{\fg_X}$.    
\end{enumerate}
 \end{lem}

\prf  
 
(1)   
 Fix $x_0 \in X$, and let $H=\{g \in G: gx_0=x_0\}$ be the stabilizer.   The topology on $X$
 is just the quotient topology on $G/H$, since the map $g \mapsto gx_0$ is continuous and closed.
 Since  $X$ is T1, $H$ is closed, hence compact.  So the image $\bar{H}$ of $H$ in $G/N$ is compact, hence 
 (as $G/N$ is Hausdorff) closed; and hence $(G/N)/ \bar{H}=
 G/ (HN)$ is Hausdorff.  But this is the same space as $(G/H)/N= X/N$.
 
 (2) 
 Let $f: X \to Y$ be a continuous map into a Hausdorff space.  Then for $g \in {\fg_X}$ we have
 $f(gx)=f(x)$, since $f(gx), f(x)$ cannot be separated by disjoint open sets.  Thus 
 $f$ factors through $f': X /{\fg_X} \to Y$, which is continuous by definition of the quotient topology.
 \eprf
 
 \begin{rem} \label{topglem2r}
 It follows from \lemref{topglem2}
  (1) and (2)  
  that if $N \leq {\fg_X}$ and $G/N$ is Hausdorff, then 
 the same equivalence relation  is induced on $X$ by $\fg_X$ and by $N$. 
 \end{rem}

\begin{lem}\label{second} Let $(X,\ft_1)$ be a   Hausdorff space.  Let $\ft_2$ be a topology on $X^2$ containing the product topology $\ft_1^2$, with $(X^2,\ft_2)$ compact.   Let $\Delta=\{(x,x): x \in X\}$
be the diagonal, and assume  $\Delta= \meet_{i \in I} G_i$ with $G_i$ open, $\lambda= |I|+\aleph_0$.
Then $\ft_1^2=\ft_2$, and $\ft_1$ admits a  basis of cardinality $\lambda$ (and hence is metrizable, if $\lambda=\aleph_0$).  
   \end{lem}
  
 \prf  As the compact $\ft_2$ contains $\ft_1^2$, which is Hausdorff, they are equal, and so $\ft_1$ is compact too.  We have
 $\Delta= \meet_{n \in \Nn} G_n$ with $G_n$ open. If $(a,b) \notin G_n$, find disjoint open $U,V$
 with $a \in U, b \in V$.  By compactness, $X^2 \m G_n$ is covered by   finitely many such $U \times V$.  Thus in all, $X^2 \m \Delta$ is covered by $\lambda$ open $U_i \times V_i$, with $U_i,V_i$ disjoint.  Let $\ft_0$ be the topology generated by these $\lambda$ sets $U_i,V_i$.  Then $(X,\ft_0)$ is Hausdorff, and $\ft_0$ contained in the compact $\ft_1$, so they are also equal.
  \eprf
  \end{section}

   \end{document}